\documentclass[12pt]{amsart}
\input epsf

\usepackage{amsfonts,amsthm,amsmath,amssymb,latexsym,amsbsy}
\usepackage{graphicx,color}
\usepackage[all]{xy}
\usepackage{labelfig}
\usepackage{epsfig}
\usepackage{soul}
\usepackage{wrapfig}
\usepackage{fullpage}
\usepackage{hyperref}
\usepackage{amsfonts,amsthm,amsmath,amssymb,latexsym}
\usepackage{tabularx}
\usepackage{lipsum}
\makeatletter
\renewcommand{\@secnumfont}{\bfseries}
\makeatother

\numberwithin{figure}{section}
\hypersetup{colorlinks=true,linkcolor=blue,citecolor=red}

\def\HH{\mathbb{H}^2}

\def\L{\mathcal{L}}

\theoremstyle{definition}

 \newtheorem{definition}{Definition}[section]
 \newtheorem{remark}[definition]{Remark}
 \newtheorem{example}[definition]{Example}

\newcommand{\g}{\gamma}

\newcommand{\m}{\mathrm{mod}}
\newcommand*\Bell{\ensuremath{\boldsymbol\ell}}
\theoremstyle{plain}

 \newtheorem{proposition}[definition]{Proposition}
 \newtheorem{theorem}[definition]{Theorem}
 \newtheorem{corollary}[definition]{Corollary}
 \newtheorem{lemma}[definition]{Lemma}
 \newtheorem{question}[definition]{Question}
 
\author{A. Basmajian, H. Hakobyan and D. \v Sari\' c}
\date{\today}
\address[Ara Basmajian]{PhD Program in Mathematics, The Graduate Center, CUNY \\ 365 Fifth Ave., N.Y., N.Y., 10016 and\newline Department of Mathematics, Hunter College, CUNY \\ 695 Park Ave., N.Y., N.Y., 10065, USA.}
\email{abasmajian@gc.cuny.edu}

\address[Hrant Hakobyan]{Department of Mathematics, Kansas State University\\ Manhattan, KS, 66506-2602, USA.}
\email{hakobyan@math.ksu.edu}

\address[Dragomir \v Sari\' c]{PhD Program in Mathematics, The Graduate Center, CUNY \\ 365 Fifth Ave., N.Y., N.Y., 10016 and\newline Department of Mathematics, Queens College, CUNY\\ 65--30 Kissena Blvd., Flushing, NY 11367, USA.}
\email{Dragomir.Saric@qc.cuny.edu}

 \thanks{The first author was partially supported by the Simons Foundation Collaboration Grant (359956, A.B.). The second author was partially supported by the Simons Foundation Collaboration Grant, (638572, H.H.). The third author was partially supported by the Simons Foundation Collaboration Grant, (346391, D.\v S.) and by PSCCUNY grants (61582 and 63477, D.\v S.).}

\begin{document}

\subjclass[2010]{30F20, 30F25, 30F45, 57K20}

\title{The type problem for Riemann surfaces  via Fenchel-Nielsen parameters}

\maketitle


\begin{abstract}
A Riemann surface $X$ is said to be of \emph{parabolic type} if it does not support  a Green's function. Equivalently, the geodesic flow on the unit tangent  bundle of $X$ (equipped with the hyperbolic metric) is ergodic. Given a Riemann surface $X$ of arbitrary topological type and a hyperbolic pants decomposition of $X$ we obtain sufficient conditions for parabolicity of $X$ in terms of the Fenchel-Nielsen parameters of the decomposition. In particular, we initiate the study of the effect of twist parameters on parabolicity.

A key ingredient in our work is the notion of \textit{nonstandard half-collar} about a hyperbolic geodesic. We show that the modulus of such a half-collar is much larger than the modulus of a standard half-collar as the hyperbolic length of the core geodesic tends to infinity. 
Moreover, the modulus of the annulus obtained by gluing two nonstandard half-collars depends on the twist parameter, unlike in the case of standard collars.

Our results are sharp in many cases. For instance, for zero-twist flute surfaces as well as for half-twist flute surfaces with concave sequences of lengths our results provide a complete characterization of parabolicity in terms of the length parameters. It follows that parabolicity is equivalent to completeness in these cases. Applications to other topological types such as surfaces with infinite genus and one end (a.k.a. the infinite Loch-Ness monster), the ladder surface, and Abelian covers of compact surfaces are also studied.
\end{abstract}

\tableofcontents


\newpage
\section{Introduction and results}

\subsection{The type problem}
A fundamental question in the classification theory of Riemann surfaces, also known as the \emph{type problem}, is whether a Riemann surface $X$ supports a Green's function. A Riemann surface is said to be of \emph{parabolic type} if it does not support a Green's function (equivalently, Brownian motion on $X$ is recurrent). Classically the class of parabolic surfaces has been denoted by $O_{G}$ , see \cite{Ahlfors-Sario}.

There are numerous known characterizations of parabolic surfaces coming from function theory, dynamics, and geometry. Specifically, if the Riemann surface $X$ is the quotient of the hyperbolic plane by a Fuchsian group, i.e. $X=\mathbb{H} / \Gamma$ then $X$ is parabolic if and only if one of the following conditions holds, see e.g. \cite{Agard,Ahlfors-Sario,Astala-Zinsmeister,Bishop,Fernandez-Melian, Nevanlinna:criterion,Nicholls1,Sullivan,Tukia}:
\begin{enumerate}
\item Harmonic measure of the ideal boundary $\partial_{\infty}X$ vanishes;
\item Geodesic flow on the unit tangent bundle of $X$ (equipped with the hyperbolic metric) is ergodic;
\item Poincare series of $\Gamma$ diverges;
\item $\Gamma$ has the Mostow rigidity property;
\item $X$ has the Bowen  property.
\item Almost every geodesic ray is recurrent. Equivalently, the set of escaping geodesic rays from a point $p\in X$ has zero (visual) measure.
\end{enumerate}

Various sufficient conditions for being of parabolic  type in terms of explicit constructions were classically studied by Myrberg, Ahlfors, Nakai, S. Mori, Ohtsuka, Sario, Nevanlinna and many others (see \cite{SarioNakai} and \cite{Ahlfors-Sario} for references).

The main goal of the present work is to make  transparent the relationship between the  \emph{hyperbolic geometry} of a Riemann surface and its type. Our main results  give  sufficient conditions on the Fenchel-Nielsen parameters of a surface (length and twist parameters on a pants decomposition) to guarantee that it is of parabolic type, see Theorems \ref{thm:general_parabolic-intro} and \ref{thm-intro:general_parabolicwithtwist}. Some of the important aspects of   these sufficient conditions are described next. 
\begin{itemize}
\item[$1.$] \emph{Twists}. For the first time in the literature we explicitly identify the effect of the \emph{twist parameters} on parabolicity. For instance, we show that the intuitive heuristic \emph{``increasing twists preserves parabolicity"}  holds in wide generality.
\item[$2.$] \emph{Sharpness}. Our sufficient conditions are often sharp.  This allows us to obtain a characterization of parabolicity in geometric terms in many cases. For instance, we prove that $X$ is parabolic if and only if it is complete,  provided $X$ is a zero-twist flute surface, or a half-twist flute surface with a concave sequence of lengths  of the pants decomposition, see Theorems \ref{thm:zero-twists-intro} and \ref{thm:half-twists-intro}. 
\item[$3.$] \emph{Generality}. We do not impose any restrictions on the topology of the Riemann surface and thus our results are valid in the general context.
\end{itemize}   

The study of the relationship between the geometry of a Riemann surface and the type problem has a long history. Besides the works mentioned above, Nicholls \cite{Nicholls} and Fern\' andez-Rodr\' iguez \cite{Fernandez-Rodriguez1} obtained sufficient conditions for parabolicity in terms of the growth of the fundamental domains of the corresponding Fuchsian groups. However, a complete characterization of parabolicity in terms of the growth of the fundamental domain is impossible, see \cite{Nicholls}. More recently, Matsuzaki-Rodr\' iguez \cite{Matsuzaki-Rodriguez} considered the  type problem for tight flute surfaces with uniformly distributed cusps. 
The use of twists/shears has also played a crucial role in the celebrated work of Kahn-Markovic \cite{Kahn-Markovic1,Kahn-Markovic2} on  the surface subgroup  and  Ehrenpreis conjectures.

\subsection{General results}
Let $X$ be an infinite type Riemann surface, and $\{X_n\}$  an   {\it exhaustion of ${X}$ by finite area geodesic subsurfaces so that no boundary component of $X_n$ is a boundary component of
 $X_{n+1}$}.  All  exhaustions  in this paper are assumed to be of  this type. 

{We denote by $\partial_0 X_n$ the collection of boundary components of $X_n$. Thus the elements of $\partial_0 X_n$ are pairwise disjoint simple closed geodesics. }By adding additional simple closed geodesics we complete $\partial_0 X_n$ to a pants decomposition of $X_n$. Hence the  Riemann surface $X$ endowed with a conformal hyperbolic metric can  be viewed   as  infinitely many  geodesic pairs of pants glued along their boundary geodesics. In this paper we are not concerned with marked hyperbolic structures. Thus, the choice of   geodesic pairs of pants is given by the {\it lengths} of the boundary geodesics, while the choice in the gluing is given by an angular parameter in the interval 
$(-\frac{1}{2},\frac{1}{2}]$, called the  {\it twist}, see Sections \ref{sec: Contents, notation, and convention} and \ref{Section:gluing}.
The lengths and twists,  called the {\it Fenchel-Nielsen} parameters  (relative to the pants decomposition), determine the conformal hyperbolic metric on $X$ 
(see Section \ref{sec: Contents, notation, and convention} for details).  

Let $P$ be a  pair of pants in the pants decomposition as above that is contained in $X_{n+1}-X_n$. We will denote by
$\alpha$  one of the pants curves of $P$, and by $\gamma$ a simple     orthogeodesic in $P$  from $\alpha$ to one of the other pants curves on the boundary of $P$, see Figure \ref{fig:collars}. The twist along $\alpha$ will be denoted by $t(\alpha)$. We denote by {$\ell (\g)$} the length of a geodesic {$\g$} on a hyperbolic Riemann surface.

For each topological type we give  a sufficient condition for parabolicity.

\begin{theorem}
\label{thm:general_parabolic-intro}
Let $X$ be an infinite type hyperbolic surface with an exhaustion $\{X_n\}$. Suppose there are constants $\alpha_0, \g_0>0$ such that for every pair of pants $P$, and curves $\alpha$ and $\g$ in $P$ as above we have $\ell(\g) \geq \g_0$ and $\ell(\alpha)\geq \alpha_0$. If 
\begin{align}\label{sum:lengths}
\sum_{n=1}^{\infty}\frac{1}{\sum_{\alpha\in \partial_0 X_n} e^{\ell(\alpha)/2}}=\infty
\end{align}
then $X$ is parabolic.
\end{theorem}  

Theorem  \ref{thm:general_parabolic-intro} is a consequence of a more general result, cf., Theorem \ref{thm:general_parabolic}, which is valid without the lower bound assumptions for $\ell(\g)$ and $\ell(\alpha)$.

Theorem \ref{thm:general_parabolic-intro} is a twist free result in the sense that it holds for any choice of twist parameters. We obtain a stronger result by bringing in the twist parameters into the sufficient condition for parabolicity. In order to do this, we make the mild technical assumption that the boundaries of 
$X_{n+1}$ and $X_{n}$  are not too close and  the  connected components of  $X_{n+1}-X_n$ are not too small (that is, not a pair of pants), see Theorem \ref{thm:general_parabolicwithtwist} for the precise hypotheses.

\begin{theorem}\label{thm-intro:general_parabolicwithtwist}
Let $X$ be an infinite type  hyperbolic surface with exhaustion $\{X_n\}$ as in Theorem \ref{thm:general_parabolic-intro} with the assumptions mentioned above.  If
\begin{equation}\label{eq:general_parabolicwithtwist-intro}
\sum_{n=1}^{\infty}\frac{1}{\sum_{\alpha\in \partial_0 X_n} 
e^{\left(1- |t (\alpha)|\right)\frac{\ell (\alpha)}{2}}}=\infty
\end{equation}
then $X$ is of parabolic  type (see Theorem \ref{thm:general_parabolicwithtwist} for the precise formulation).
\end{theorem}

In the above theorem the twist parameter, $t(\alpha)$, is measured with respect to a  pants decomposition which includes the boundary components of the $X_n$. Clearly, condition (\ref{sum:lengths}) implies (\ref{eq:general_parabolicwithtwist-intro}). Therefore, if $X$ satisfies (\ref{sum:lengths}) then not only  $X$ is parabolic but so are all the hyperbolic surfaces obtained by deforming $X$ by twisting along the boundary curves of the exhaustion $\{X_n\}$.  

\begin{remark}[\textsc{Increasing twists preserves parabolicity}]
Since $e^{(1-|t|)\ell/2}$ is decreasing in $|t|$, Theorem  \ref{thm-intro:general_parabolicwithtwist} implies that if $X$ satisfies (\ref{eq:general_parabolicwithtwist-intro}) and $X'$ is a surface obtained from $X$ by increasing the absolute value of twist parameters $t(\alpha)$ then $X'$ is also parabolic. 
\end{remark}

\subsection{Tight flute surfaces}

Arguably the simplest infinite type (hyperbolic) Riemann  surface $X$ is a tight flute surface, see \cite{Basmajian}.  It is obtained by starting with a geodesic pair of pants $P_0$ with two punctures and then consecutively gluing geodesic pairs of pants $P_n$, $n\geq 1$, with one puncture and two boundary geodesics in an infinite chain. Let $\ell_n$ and $t_n$ be the length and twist parameters of the closed geodesic $\alpha_n$ on the boundary after gluing $n$ pairs of pants. We denote the resulting surface by $X=X(\{ \ell_n,t_n\})$, see Figure \ref{fig:tight-flute}. It is relatively simple to see that if an infinite subsequence  of $\{ \ell_n\}$ is bounded above by a positive constant then $X$ is of parabolic  type. When $\ell_n\to\infty$, applying Theorem \ref{thm-intro:general_parabolicwithtwist} we obtain the following.  

\begin{figure}[t]
\begin{center}
\includegraphics[width=5.5in]{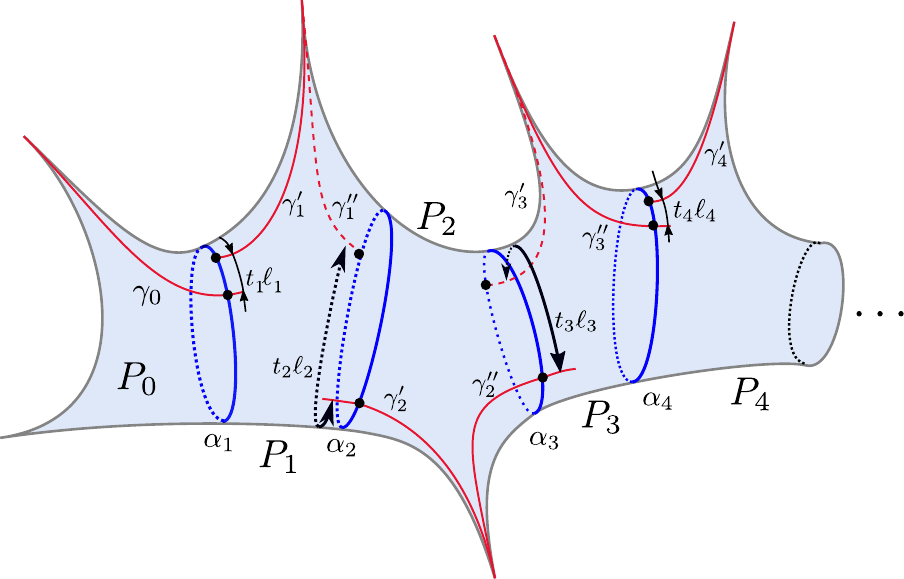}
\caption{A tight flute surface.}
\label{fig:tight-flute}
\end{center}
\end{figure}

\begin{theorem}
\label{thm:intro2}
Let $X=X(\{\ell_n,t_n\})$ be a  tight flute surface such that $\ell_n\to\infty$. Then $X$ is of parabolic  type if
\begin{equation}\label{eq: tight flute parabolic}
\sum_{n=1}^{\infty}e^{-( 1-|t_n|) \frac{\ell_n}{2}}=\infty .
\end{equation}
\end{theorem}

If we set the twists $t_n$ equal to zero in (\ref{eq: tight flute parabolic}) we have that a flute surface $X=X(\{\ell_n,0\})$ is parabolic if $\sum e^{-\ell_n/2}=\infty$. It turns out that this condition is not only sufficient but also necessary.  
Moreover, we prove the following, see Theorem \ref{thm:zero-twists}. 

\begin{theorem}[Parabolicity of zero-twist flutes]\label{thm:zero-twists-intro}
A zero-twist tight flute surface $X=X(\{\ell_n,0\})$ is parabolic if and only if one of the following holds: 
 \begin{enumerate}
\item $X$ is complete, 
\item $\sum_{n=1}^{\infty}e^{-\ell_{n}/2}=\infty$.
 \end{enumerate}
 \end{theorem}

If  $t_n={1}/{2}$ for all $n\in\mathbb{N}$ then we obtain a {\it half-twist tight flute} 
$X(\{ \ell_n, 1/2\})$. In this case equation  (\ref{eq: tight flute parabolic}) becomes $\sum_ne^{-\ell_n/4}=\infty$.
%
Unlike the zero-twist case we do not know if this condition is necessary and sufficient for parabolicity for an arbitrary sequence $\{\ell_n\}$. However, we show that the condition is sharp in many cases. To do this we first obtain a sufficient condition for non-completeness and hence non-parabolicity for half-twist tight flutes. 

\begin{theorem}
\label{thm:half-twist-incomplete-intro}
A half-twist tight flute surface $X(\{ \ell_n,\frac{1}{2}\})$ is incomplete if   
\begin{align}\label{eqn:half-twist-incomplete-intro}
\sum_n e^{-\frac{\sigma_n}{2}}<\infty,
\end{align}
where $\sigma_n = \ell_n-\ell_{n-1}+\cdots +(-1)^{n-1}\ell_1$.
\end{theorem}

Using  Theorem  \ref{thm:half-twist-incomplete-intro}
we  identify a class of half-twist tight flute surfaces for which we have a characterization of parabolicity. 

We say that $\{\ell_n\}$ is a \emph{concave sequence} if  there is a non-decreasing concave function $f:[0,\infty)\to[0,\infty)$ such that $\ell_n=f(n)$ for $n\geq 0$. Equivalently, $\ell_n$ is concave if it is non-decreasing and for $n\geq 1$ the following holds:
\begin{align}\label{ineq:concave}
2 \ell_{n}\geq \ell_{n+1}+\ell_{n-1}.
\end{align}

For half-twist surfaces corresponding to concave sequences we show that $e^{-\sigma_n/2} \asymp e^{-\ell_n/4}$. Theorems \ref{thm:intro2} and  \ref{thm:half-twist-incomplete-intro} then give the following characterization, see Theorem \ref{thm:half-twists}.

\begin{theorem}[Parabolicity of half-twist flutes]\label{thm:half-twists-intro} 
Let $X=X(\{\ell_n,1/2\})$, where $\{\ell_{n}\}$ is a concave sequence. Then $X$ is parabolic if and only if one of the following conditions holds
\begin{enumerate}
\item $X$ is complete,
\item $\sum_n e^{-\ell_n/4} =\infty$.
\end{enumerate}
\end{theorem}

Given Theorems \ref{thm:zero-twists-intro} and \ref{thm:half-twists-intro} one may think that a tight flute surface is parabolic if and only if it is complete. This is not the case. Indeed, let $X$ be obtained by taking out a sequence of points from the unit disk that converge to every point of the unit circle. Then $X$ is the tight flute surface that is the union of countably many pairs of pants, complete and not of parabolic type (see also  \cite{Kinjo} and \cite{HaasSusskind}). It is not known if there are such examples among half-twist tight flute surfaces. In Section \ref{sec:flute} we construct examples of half-twist tight flutes (necessarily with not concave $\{\ell_n\}$) for which Theorems \ref{thm:half-twists-intro} and \ref{thm:half-twist-incomplete-intro} do not apply, see Example \ref{example:parameters}. A particular case of that is the following.

\begin{example}\label{example-intro}
Let $X_s=X(\{s\ell_n,1/2\})$ be a tight flute surface, where for $n\geq1$ we have
\begin{align*}
\ell_{2n} &= \ln(n+1)+2\ln n,\\
\ell_{2n+1} &=3\ln(n+1).
\end{align*}
Applying the above results we obtain that $X_s$ is parabolic if $s\in(0,4/3]$, and $X_s$ is incomplete if $s>2$ (see Example \ref{example:parameters} for the proofs). For $s\in(4/3,2]$ the results of this paper are inconclusive. It would be interesting to know whether $X_s$ is complete and non-parabolic if $s\in(4/3,2]$.
\end{example}

Motivated by the discussion above we ask the following.

\begin{question}
Suppose $\ell_n\to\infty$. 
\begin{enumerate}
\item Is $X=X(\{\ell_n,1/2\})$ incomplete if and only if (\ref{eqn:half-twist-incomplete-intro}) holds?
\item Is $X'=X(\{\ell_n,t_n'\})$ parabolic if $X=X(\{\ell_n,t_n\})$ is parabolic and $t_n<t_n', \forall n\geq 1$?
\item Given  $\{\ell_n\}$ is there a sequence of twists  $\{t_n\}$  such that  $X=X(\{\ell_n,t_n\})$ is parabolic.
\end{enumerate}
\end{question}


\begin{figure}[h]
\begin{tabular}{c c}
\includegraphics[width=4in]{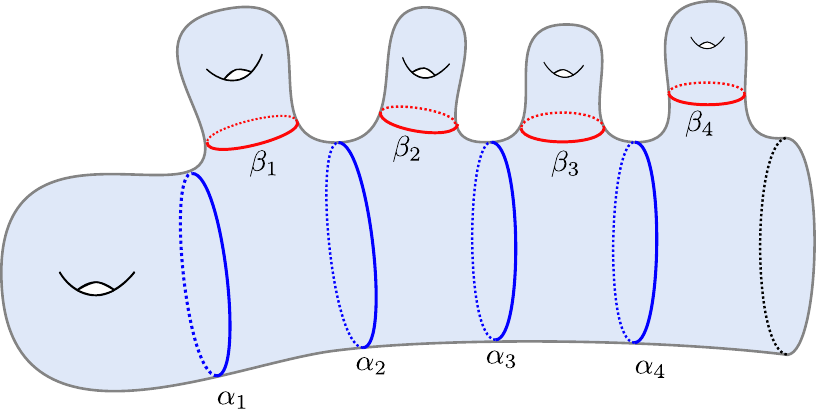}
\end{tabular}
%
%
%
%
\caption{Surface $X^{\infty}_1$ (the Loch-Ness monster) of infinite genus with one non-planar  end.}
\label{fig:loch-ness}
\end{figure}

\subsection{Applications to various surfaces and regular covers} 
Besides considering flute surfaces in this paper we apply the sufficient conditions for parabolicity (e.g. Theorems \ref{thm:general_parabolic-intro} and \ref{thm-intro:general_parabolicwithtwist}) to other topological types as well. Here we mention three such examples: (1) the Loch-Ness monster (surface of infinite genus and one non-planar topological end) first studied in \cite{PS}; (2) the complement of the Cantor set (uncountably many ends); (3) topological Abelian covers of compact surfaces.

\subsubsection{Loch-Ness monster}

Let  $X^{\infty}_1$ be a hyperbolic Loch-Ness monster  as in Figure \ref{fig:loch-ness}. Suppose that the lengths of geodesics which cut off the genus, denoted by $\beta_n$, are uniformly bounded above. We show, see Theorem \ref{thm:infinitegenus}, that $X^{\infty}_1$ is of parabolic  type if
\begin{align*}
\sum_{n=1}^{\infty}  e^{-(1-|t(\alpha_n)|) 
\frac{\ell(\alpha_{n})}{2}} =\infty.
\end{align*}

In the above theorem the twist parameter, $t(\alpha_n)$,  is measured relative to the endpoints in $\alpha_n$  of the orthogeodesic  from $\beta_{n-1}$ to $\alpha_{n}$, and the orthogeodesic from 
$\alpha_{n}$ to $\beta_n$.

\subsubsection{Complement of a Cantor set}
Let $X_{\infty}$ be a genus zero surface whose space of topological ends is a Cantor set as in Figure  \ref{fig:cantor}. The surface 
$X_{\infty}$ is homeomorphic to the complement of a Cantor set on the Riemann sphere and has an exhaustion $\{X_n\}$, where $X_n$ is a genus zero surface with $2^n$ geodesic boundary curves for every $n\geq 1$, see the discussion before Theorem \ref{thm:cantor}. As before we denote by $\partial_0 X_n$ the collection of boundary components of $X_n$. 

It is well-known that if the lengths of the boundary geodesics of $X_n$'s are uniformly bounded from below then the surface $X_{\infty}$ is not of parabolic  type, \cite{McM}. In the opposite direction we show (see Theorem \ref{thm:cantor}) that if there is a constant $C\geq 1$ such that for every $n\geq 1$ and all $\alpha \in \partial_0 X_n$ we have
\begin{align}\label{ineq:cantor}
\ell (\alpha) \leq C {n}/{2^n}.
\end{align}  
then $X_{\infty}$ is of parabolic type.

It is an open problem whether $X_{\infty}$ can be parabolic if the lengths of the boundary geodesics decay slower than in (\ref{ineq:cantor}) (e.g., if there is a constant $k>0$ such that $\ell(\alpha)\lesssim n^{-k}$ for $\alpha\in \partial_0 X_n$).

\subsubsection{Abelian covers of closed surfaces}
In  \cite{Mo} and \cite{Rees} it was shown that a hyperbolic Riemann surface $X$ is of parabolic type if it is  a
$\mathbb{Z}$ or $\mathbb{Z}^2$ \emph{geometric} cover 
$\pi : X \rightarrow Y$  over  a closed Riemann surface.  Our methods give an alternative proof of this result along with a generalization   to hyperbolic Riemann surfaces $X$ which are \emph{topological} covers of a closed Riemann surface, see Theorem \ref{thm:top. covering group}.  In fact, the  hyperbolic structure on $X$ can be chosen so that it is quasiconformally distinct from the hyperbolic structure on the geometric cover  but in a suitable sense has  Fenchel-Nielsen parameters
  that agree with the parameters of the regular cover for almost all pants curves.
See  Example \ref{ex: generalization of geometric cover}  for details.

\subsection{Tools of the trade: Extremal distance in nonstandard and standard collars}\label{subsection:Ahlfors-criterion}

A key ingredient in the proofs of our results is the characterization of parabolicity in terms of the extremal distance. See Section \ref{sec:modulus} for the definition and properties of extremal distance. The method of extremal length (or the length-area principle) was initiated by Ahlfors in 1935 for the study of the type problem for simply connected Riemann surfaces, see  \cite{Ahlfors:length-area}. He showed that a simply connected Riemann surface $X$ is parabolic if and only if there is a conformal metric $\rho(z)|dz|$ on $X$ and $r_0>0$ such that 
\begin{eqnarray}\label{Ahlfors}
\int_{r_0}^{\infty} \frac{dr}{L(r)} = \infty,
\end{eqnarray}
where $L(r)$ is the $\rho$-length of the circle of radius $r$ centered at some point $z_0\in X$.
Ahlfors' criterion (\ref{Ahlfors}) was generalized and reformulated by several authors and was later often referred to as the modular test. 

 Let $\{X_n\}$ be an exhaustion of $X$ by a family of relatively compact regions with piecewise analytic boundary such that $\overline{X}_n\subset X_{n+1}$. Denote by $\beta_n$ the boundary of $X_n$. 
Let $\lambda_{X_n-X_1}(\beta_1,\beta_n)$ be the  extremal length of the family of curves contained in $X_n - X_1$ which connect $\beta_1$ and $\beta_n$.  The following  characterization of parabolicity is due to Nevanlinna \cite{Nevanlinna:criterion}, see also \cite[page 328]{SarioNakai}. 

\noindent \textsc{Modular test}.
\textit{The Riemann surface $X$ is parabolic if and only if}
\begin{align}\label{Ahlfors-Sario}
 \lambda_{X_n-X_1}(\beta_1,\beta_n)\to\infty\quad \mathrm{as }\quad n\to\infty.
\end{align} 

Informally, $X$ is parabolic if the extremal distance between any compact subset of $X$ and its ideal boundary is infinite, i.e. $\partial_{\infty}X$ cannot be reached in finite time. Equivalently, $X$ is parabolic if and only if the capacity of $\partial_{\infty} X$ vanishes, see  \cite{Sario}. 

Since $\lambda_{X_n-X_1}(\beta_1,\beta_n)$ is difficult to compute or estimate, one usually uses condition (\ref{Ahlfors-Sario}) in conjunction with the so-called \emph{serial rule} for the extremal length. For that we suppose that the connected components $\beta_{k,j}$ of $\beta_k$ are contained in pairwise disjoint collars (topological annuli), denoted by $A_{k,j}$. Let $\lambda_{k,j}$ be the extremal length of the path family in $A_{k,j}$ connecting its boundary components, and denote $\lambda_k = \sum_j \lambda_{k,j}$. From the serial rule and the fact that $A_{k,j}$'s are disjoint it follows that $\lambda_{X_n-X_1}(\beta_1,\beta_n)\geq  \sum_{k=1}^{n-1} \lambda_k.$ Therefore, by (\ref{Ahlfors-Sario}) $X$ is parabolic, provided
\begin{align*}
\sum_{k=1}^{\infty}{\lambda_k} = \sum_{k=1}^{\infty} \sum_{j=1}^{|\partial_0(X_n)|} \lambda_{k,j}= \infty,
\end{align*}
where $|\partial_0(X_n)|$ denotes the number of boundary components of $X_n$.

Thus, if $X$ is a Riemann surface with an exhaustion $\{X_n\}$, where the boundary components of $X_n$ are geodesics then we would like to construct disjoint collars around these boundary geodesics and calculate  or estimate the extremal distance between their boundaries. 


\begin{figure}[t]
\begin{center}
\includegraphics[height = 2in]{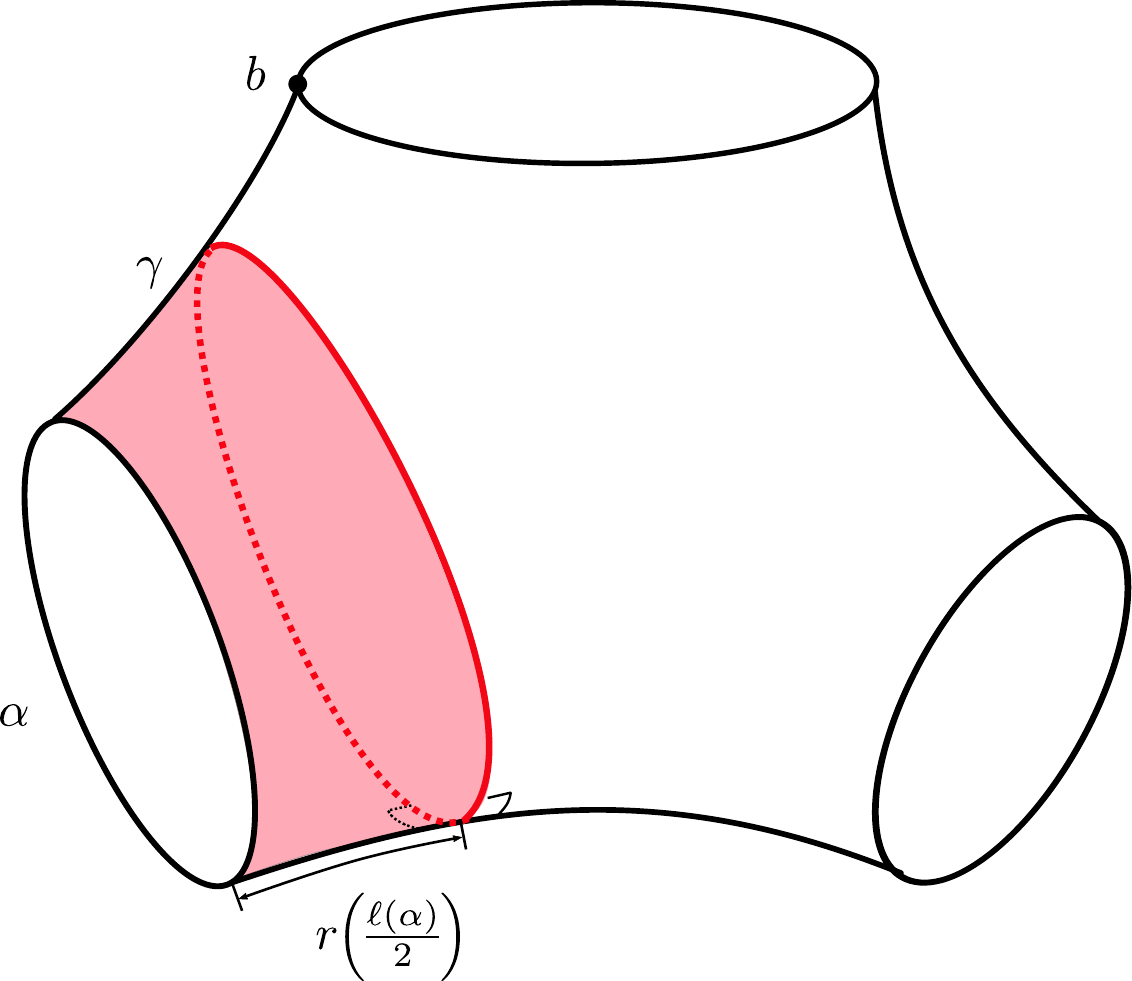} \qquad \qquad  \includegraphics[height = 2in]{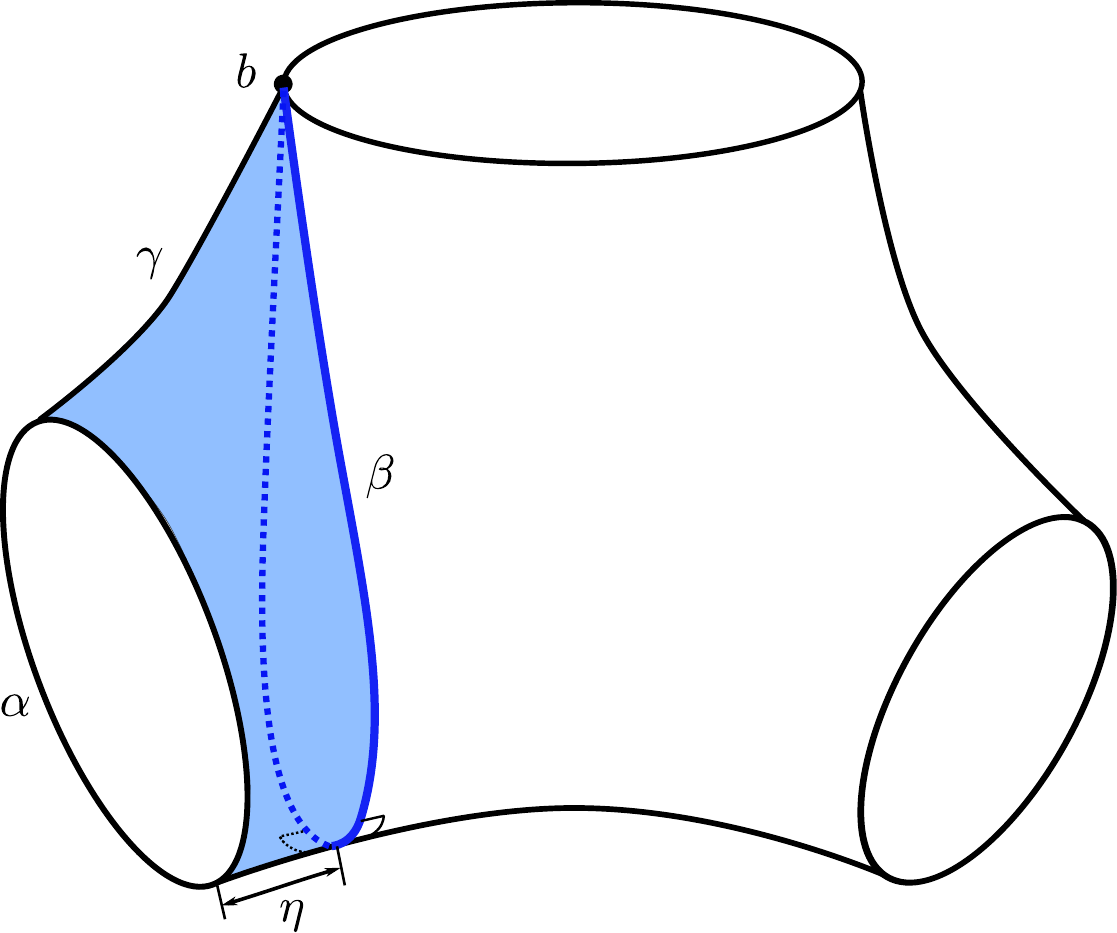}
\end{center}
\caption{Standard (left) and nonstandard (right) half-collars about a geodesic $\alpha$ of length $\ell$.} \label{fig:collars}
\end{figure}

The well-known collar lemma  (see \cite{Buser}) tells us that a  simple closed geodesic of length $\ell$ on a hyperbolic Riemann surface is guaranteed to have a  collar (annular neighborhood)  of width 
$\text{arcsinh}\left( \frac{1}{\sinh \frac{\ell}{2}}\right)$, see Figure \ref{fig:collars} for the picture of a half-collar.  We call this a {\it standard collar}. The important point is that the width only depends on  the length of the geodesic and not the ambient hyperbolic structure of the surface.   The extremal distance between the boundary components of the one-sided  standard collar up to a constant multiple is bounded below by $\frac{e^{-\ell/2}}{\ell}$  (see \cite{Maskit} and Lemma \ref{lem:collar}).  This is a good asymptotic estimate when $\ell\to 0$ (or in the thin parts of the surface)  but not for large $\ell$. 

To deal with large $\ell$ (or thick parts of the surface), we introduce what we call a nonstandard half-collar about a geodesic in a pair of pants. This half-collar will depend  on local data of the pair of pants as opposed to the standard half-collar which depends only on the length of the closed geodesic. 
Most of the sufficient conditions for parabolicity we obtain follow from the extremal distance bounds on the nonstandard half-collars and  collars, which we describe next.

Let $\alpha, \alpha_1$, and $\alpha_2$ be the  boundary geodesics (we allow $\alpha_1$ or $\alpha_2$   to be a puncture) of  a pair of pants $P$, and $\gamma$ the unique simple  orthogeodesic between  $\alpha$ and $\alpha_1$, see Figure \ref{fig:collars}.  Letting  $b$ be the endpoint  of $\gamma$ on $\alpha_1$,  there exist exactly two shortest  geodesic segments from $b$ to the   simple orthogeodesic from $\alpha$ to $\alpha_2$. These segments have equal length and the  union of the two  connect to make a geodesic loop $\beta$ with non-smooth point $b$. The {\it nonstandard half-collar} around 
$\alpha$ is the region in $P$ between $\alpha$ and the geodesic loop $\beta$. It is topologically an annulus which we denote by $R_{\alpha,\gamma}$, see Figure \ref{fig:collars}.
A more general type of collar has been considered by Parlier in a different context \cite{Parlier, Parlier1}.

In order to simplify the notation, we say that two positive quantities $a$ and $b$ satisfy $a\gtrsim b$ if $a/b$ is greater than or equal to a positive constant; $a\lesssim b$ if $a/b$ is less than or equal to a positive constant; and $a\asymp b$ if $a/b$ is between two positive constants or equivalently, if $a\gtrsim b$ and $a\lesssim b$.

Neither   type of half-collar (standard or nonstandard) contains the other except when  $\alpha_1$ is a puncture (that is, $\ell (\g )=\infty$) in which case  the nonstandard half-collar contains the standard half-collar  (see Figure \ref{fig:comparing-collars}). Nevertheless, for $\ell=\ell(\alpha)$ large, the nonstandard half-collar produces a  larger extremal distance between the boundary components than the standard half-collar, see Theorem \ref{thm:half-collar} and Corollary \ref{cor:asymp}. For example, when  $\ell (\gamma )\geq \g_0>0$, we have  the following asymptotic behavior for $\lambda(R_{\alpha,\gamma})$ as $\ell(\alpha)\to\infty$:
\begin{align}\label{modest:no-twists}
\lambda(R_{\alpha,\gamma}) \gtrsim e^{-\ell(\alpha)/2},
\end{align}   
while the extremal distance between the boundary components of the standard half-collar, as noted above, is comparable to $e^{-\ell(\alpha)/2}/\ell(\alpha)$.

\begin{figure}[t]
\begin{center}
\includegraphics[height = 2in]{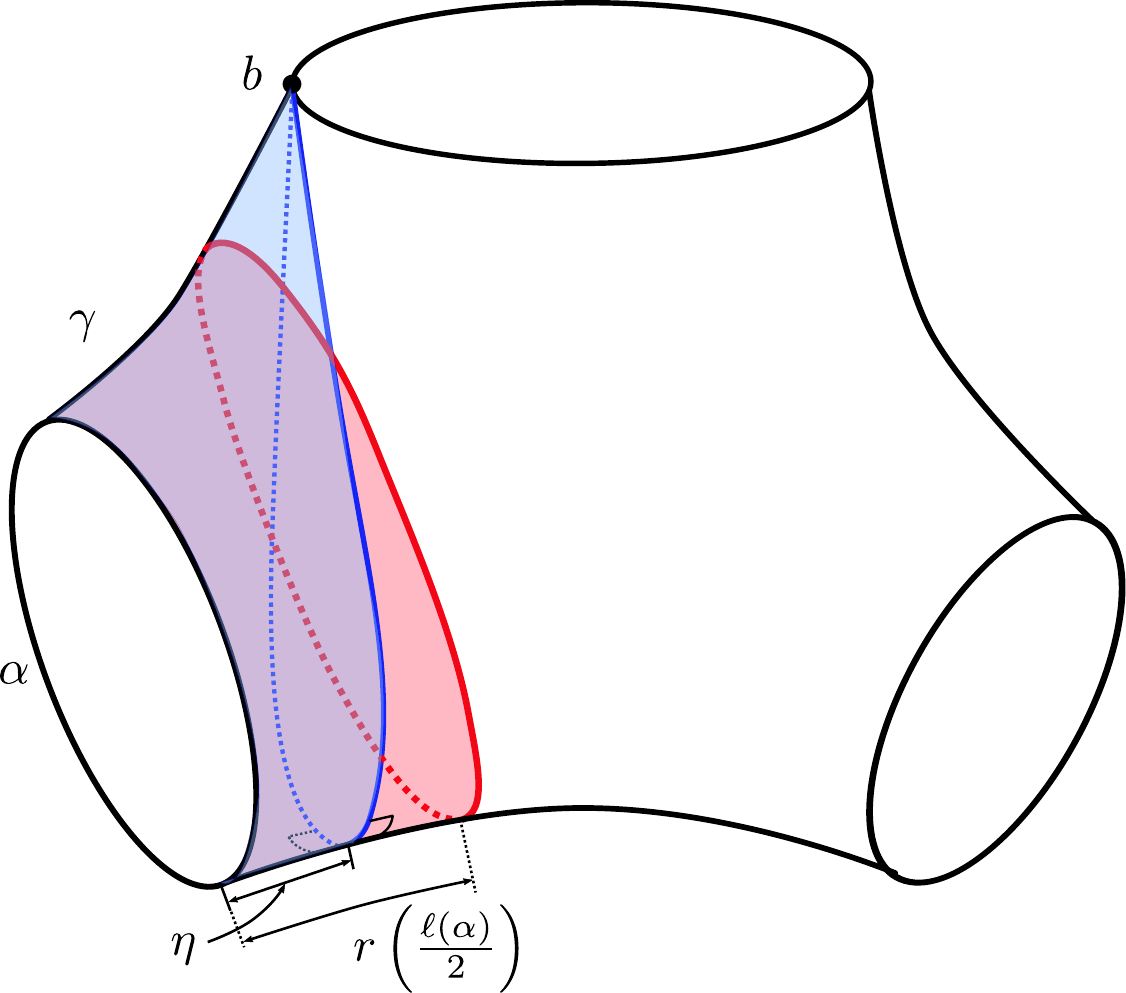}
\qquad\qquad
\includegraphics[height = 2.5in]{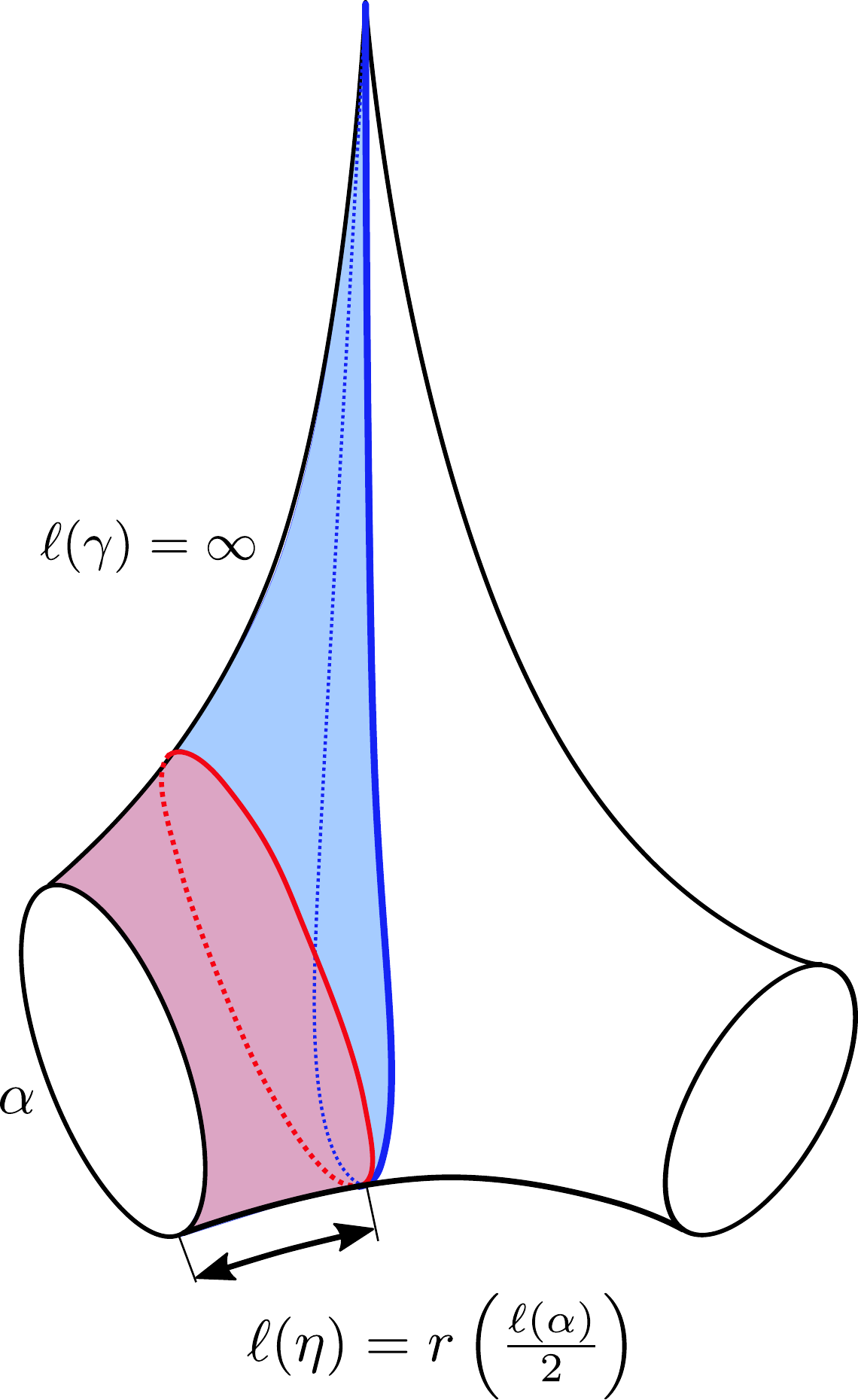}
\end{center}
\caption{Pink and blue regions above are the standard and nonstandard half-collars about the hyperbolic geodesic $\alpha$, respectively. When the length of the orthogeodesic $\gamma$ from $\alpha$ to another boundary component of the pants is infinite (on the right) then the standard collar is completely contained in the non-standard collar. As $\ell=\ell(\alpha)\to\infty$ the modulus of the standard collar is comparable to $e^{-\ell/2}/\ell$, while for  the nonstandard collar the lower bound is of the order $e^{-\ell/2}$.
}\label{fig:comparing-collars}
\end{figure}

When two standard half-collars $R$ and $R'$ around geodesics of the same length  are glued by an isometry along the geodesics,  the obtained surface $\widehat{R}$ is invariant under the isometric reflection in the geodesic regardless of the twist. Consequently, the extremal distance between the boundary curves of  the collar $\widehat{R}$ equals the sum of the extremal distances between the boundary components of the two half-collars, i.e.,
\begin{align}\label{asymp:standard-collar}
\lambda(\widehat{R}) = \lambda(R) + \lambda (R') \asymp {e^{-\ell/2}}/{\ell}
\end{align}
On the other hand, when two  nonstandard half-collars are glued, see Figure \ref{fig:glued-ns-collars-topview}, the extremal distance between the boundary components of the glued surface may significantly increase  depending on the twist of the gluing (Theorem \ref{thm:gluings}). For instance, we show that if $R^t_{\alpha,\gamma,\gamma'}$ is the annular region obtained by gluing two nonstandard half-collars $R_{\alpha,\gamma}$ and $R_{\alpha,\gamma'}$ with twist $t \in(-\frac{1}{2},\frac{1}{2}]$ and $\ell (\gamma ) =\ell (\gamma') =\infty$ then 
\begin{align}\label{modest-with-twist}
\lambda (R^t_{\alpha,\gamma,\gamma'})\gtrsim e^{-(1-|t|)\frac{\ell(\alpha)}{2}},
\end{align}
provided $\ell (\alpha ) \geq \ell_0\geq 2$. 
Therefore, comparing the nonstandard and standard two-sided collars we have from (\ref{asymp:standard-collar}) and (\ref{modest-with-twist}) the following asymptotic inequality, as $\ell(\alpha)\to\infty$, 
\begin{align}\label{comparing-collars}
\frac{\lambda(R^t_{\alpha,\gamma,\gamma'})}{\lambda(\widehat{R})}\gtrsim \ell(\alpha) {e^{|t|\frac{\ell(\alpha)}{2}}}.
\end{align}
In particular, as $\ell(\alpha)\to\infty$  the ratio between the non-standard and standard collars grows linearly if $t=0$ and exponentially as soon as there is a non-trivial twist $t\neq0$ involved.



\begin{figure}
\includegraphics[height = 3in]{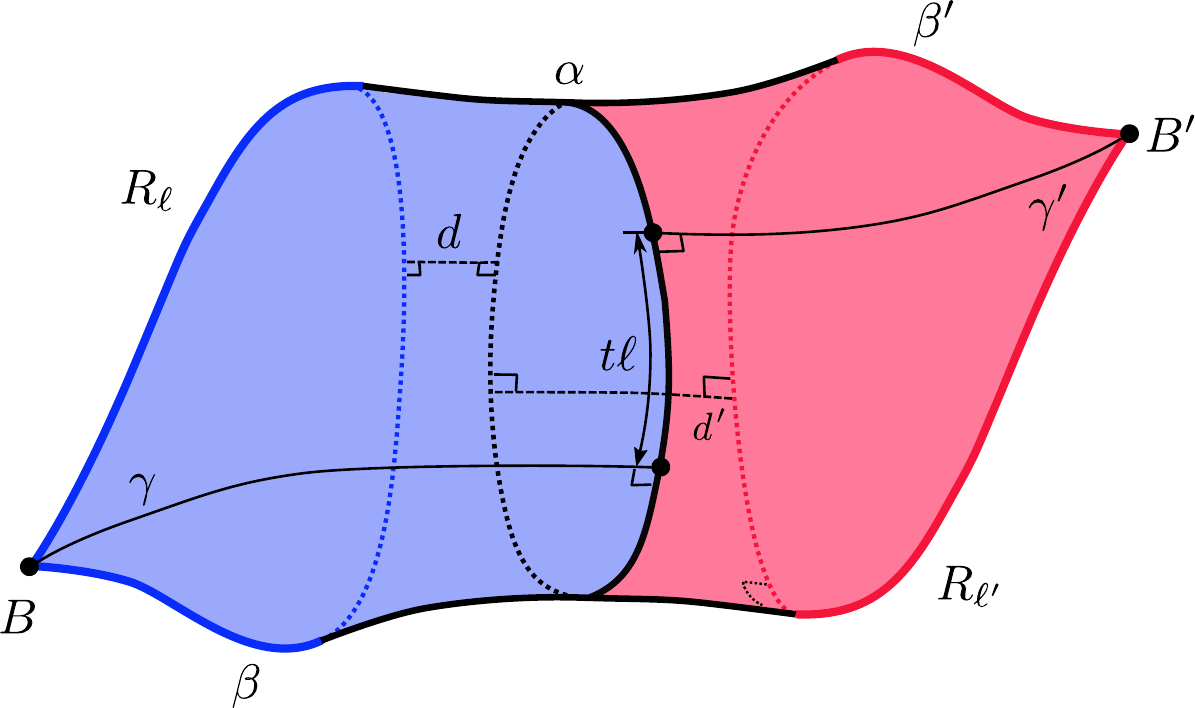}
\caption{{Let $R^t_{\alpha,\gamma,\gamma'}$ be the nonstandard collar (as above) obtained by gluing  nonstandard half-collars $R_{\alpha,\gamma}$ and $R_{\alpha,\gamma'}$ along the common geodesic $\alpha$ of length $\ell$ with a twist $t$.  We write $R^t_{\alpha}$ if $\ell(\g)=\ell(\g')=\infty$.  Denoting by  $\hat{R}_{\alpha}$ the standard collar around $\alpha$, we have $\lambda (R^t_{\alpha}) \gtrsim \ell e^{t\ell/2} \lambda (\hat{R}_{\alpha})$ as $\ell\to\infty$, see (\ref{comparing-collars}). 
}}\label{fig:glued-ns-collars-topview}
\end{figure}

The extremal distance estimates for nonstandard collars follow from our   
  technical tool 
 on  general collars about a simple closed geodesic, see Corollary \ref{cor:collar_ext}. 
We  show that the extremal distance between the boundary components of  a general collar is comparable to the extremal length of the curve family of geodesic orthorays based at the core curve
$\alpha$, see Section \ref{sec:modulus} for the definition of the extremal length of curve families. To achieve this we use logarithmic coordinates and express the universal cover of the collar as a region bounded by two graphs in the plane. Then the extremal distance between the  boundary components of the collar in the Riemann surface is related to the extremal length of curves connecting the top graph to the bottom graph in the universal cover. These curve families degenerate as the length of the core curve $\ell(\alpha )\to\infty$ and our key result is an estimate of the extremal length of such degenerating families of curves. A similar approach was also used in the context of the Teichm\"uller theory, see \cite{HS,HS:PLMS}. However the degeneration of the families in our setting are much more involved and the estimates do not follow from any previous work. 

\subsection{Outline of the paper} The rest of this paper is organized as follows. In Section \ref{sec: Contents, notation, and convention}  we introduce geodesic pairs of pants, the Fenchel-Nielsen parameters and the construction of infinite type Riemann surfaces from the geodesic pairs of pants. In Section  \ref{sec:modulus} we recall the  definition and basic properties of  modulus of curve families. Sections \ref{sec:between_graphs} - \ref{Section:gluing} are the technical core of the paper. In Section \ref{sec:between_graphs} we obtain estimates for the moduli of degenerating curve families connecting two graphs of real functions over a compact interval. In Sections \ref{sec:non-standard-collar}  and \ref{Section:gluing} we apply the results of Section \ref{sec:between_graphs} and prove the main modulus bounds for the collars around geodesics,  in particular here we prove the estimates (\ref{modest:no-twists}) and  (\ref{modest-with-twist}). In Section \ref{sec:type_problem}  we recall Nevanlinna's modular test of parabolicity (\ref{Ahlfors-Sario}) and prove a slight generalization which is used in our applications. In Section \ref{sec:general-theorems} we combine the previous results obtaining our most general sufficient conditions on the Fenchel-Nielsen parameters which guarantee parabolicity of an arbitrary infinite type Riemann surface. In particular, Theorems  \ref{thm:general_parabolic-intro} and \ref{thm-intro:general_parabolicwithtwist} follow from the results in Section \ref{sec:general-theorems}. In Section \ref{sec:flute} we consider the applications to flute surfaces which imply Theorems \ref{thm:intro2} - \ref{thm:half-twists-intro}. Section \ref{sec: A trip to the Managerie: Applications to various topological types} describes sufficient conditions for parabolicity for various topological types of infinite type surfaces including the infinite Loch-Ness monster, surfaces with finitely many ends, surfaces with a Cantor set of ends, and topological $\mathbb{Z}^r$ covers of compact surfaces. In particular, we recover the results of Mori and Rees on conformal $\mathbb{Z}$ and $\mathbb{Z}^2$ covers of compact surfaces.


We conclude the introduction by listing some of the notation used in the text together with the sections where the corresponding quantities are defined.

\vskip 0.5cm
 
\begin{center}
 \begin{tabular}{| l | c | c |}
\hline \hline 
 {\bf Definition} & {\bf Section} &{\bf Notation} \\
 \hline
\hline 

\hline
 Length and twist parameters& \ref{sec: Contents, notation, and convention} & $\ell (\alpha ), t(\alpha )$  \\
\hline
Modulus and extremal length & \ref{sec:modulus} & $\mathrm{mod}\Gamma$ \\
\hline
 Extremal distance & \ref{sec:modulus} and \ref{sec:type_problem} & $\lambda (R)$, $\lambda_{X_n-X_1} (\beta_1,\beta_n )$  \\
\hline
 Simply degenerating families of functions & \ref{sec:deg-fam}
  &  $\{(f_\ell, g_\ell)\}_{\ell \geq \ell_0}$   \\
\hline
Standard half-collar & \ref{sec:non-standard-collar} & $R_{\alpha}$ \\
\hline
nonstandard half-collar & \ref{sec:non-standard-collar} & $R_{\alpha ,\gamma}$  \\
\hline
 nonstandard collar& \ref{Section:gluing}  & $R_{\alpha ,\gamma ,\gamma'}^{t(\alpha )}$ \\
\hline

Geodesic subsurface& \ref{sec:type_problem} & $X_n$ \\
\hline
 Boundary components of $X$& \ref{sec:type_problem} &$\partial_0 X $ \\
\hline
 Tight flute surface&\ref{sec:flute}
  &$X(\{ \ell_n,t_n\})$  \\
\hline

\hline

\end{tabular} 
\end{center}
\label{Table:notation}
\vskip .5cm
\noindent {\it Acknowledgements.}  We are grateful to the referees for  the  careful reading of the paper that led to specific suggestions that improved the presentation and  simplified  the statement of Lemma \ref{lem:lift} as well as  the proof of Corollary \ref{cor:collar_ext}. We would also like to thank Mario Bonk, Misha Lyubich and Dennis Sullivan for helpful comments.

\section{Riemann surfaces of infinite topological type}
\label{sec: Contents, notation, and convention}


Every Riemann surface $X$ in this paper is assumed to admit a {\it hyperbolic metric}, that is a conformal metric of constant curvature equal to $-1$. Thus, $X$ is not  conformal to the Riemann sphere $\bar{\mathbb{C}}$, the complex plane $\mathbb{C}$, the punctured complex plane $\mathbb{C}\setminus\{ 0\}$ or the torus. See  \cite{Buser} for background on hyperbolic geometry.

We will interchangeably use the terms Riemann surface and hyperbolic surface for the same object. 
A Riemann surface $X$ is of {\it infinite topological type } if its fundamental group $\pi_1(X)$ is infinitely generated. 

 A {\it geodesic pair of pants} is a complete hyperbolic surface (homeomorphic to a sphere minus three disks) whose boundary components are either closed geodesics or punctures with at least one boundary a closed geodesic. A {\it tight}  pair of pants is a geodesic pair of pants that has at least one puncture. In Figure \ref{fig:comparing-collars} we illustrated two geodesic pairs of pants, on the left with three boundary geodesics and on the right with two boundary geodesics and a puncture. The geodesic pair of pants on the right is a tight pair of pants.

Consider a geodesic pair of pants $P$ which is not tight and fix a boundary geodesic $\alpha$ of $P$. Let $\alpha_1$ be another closed geodesic on the boundary of $P$. 
Let $\gamma$ be the orthogeodesic from $\alpha$ to $\alpha_1$. The foot $x\in\alpha$  of $\gamma$ on $\alpha$ is called a {\it marked point}. If $P$ is a tight pair of pants with one boundary puncture, we choose $\gamma$ to be the simple orthogeodesic in $P$ from $\alpha$ to the puncture. If $P$ has two punctures then we choose one puncture and repeat the construction above.

Let $P'$ be another geodesic pair of pants with boundary geodesic $\alpha'$. Assume $\ell (\alpha )=\ell (\alpha')$.  
We identify $\alpha$ and $\alpha'$ by an isometry to obtain a bordered hyperbolic surface from the two pairs of pants. The isometric identification $\alpha\equiv \alpha'$ is determined by the relative position  of the marked points $x\in\alpha$ and $x'\in\alpha'$ which is recorded by the {\it twist parameter} $t(\alpha )\in (-\frac{1}{2},\frac{1}{2}]$. 
Namely, if $x=x'$ then $t(\alpha )=0$. If $x\neq x'$ then $\alpha - \{ x,x'\}$  consists of two arcs and $|t(\alpha )|$ is the length of the shorter arc divided by $\ell (\alpha )$. If $|t(\alpha )|=\frac{1}{2}$ then we have $t(\alpha )=\frac{1}{2}$. If $|t(\alpha )|<\frac{1}{2}$ then we orient $\alpha$ as a part of the boundary of $P$. If the shorter of the two arcs of $\alpha -\{ x,x'\}$ is $(x,x')$ for the orientation of $\alpha$ then $t(\alpha )=|t(\alpha )|$; otherwise $t(\alpha )=-|t(\alpha )|$. 

By glueing countably many geodesic pairs of pants in this manner, we  obtain a not necessarily complete surface $X$ with hyperbolic metric induced by the hyperbolic metric on the geodesic pairs of pants. The choices in the gluings are given by the twist parameters and the geodesic pairs of pants are uniquely determined by the lengths of the boundary geodesics called the {\it length parameters}. When the boundary geodesic is a puncture then by convention the length is zero. Therefore the hyperbolic metric on $X$ is uniquely determined by the length and twist parameters on the boundary geodesics of the pairs of pants called the {\it Fenchel-Nielsen parameters}. Since we do not consider the space of  Riemann surfaces to have a base point surface and need not consider marked Riemann surfaces, we are content to use the twist parameters in $(-\frac{1}{2},\frac{1}{2}]$ in order to describe all hyperbolic metrics.

Finally, the surface $X$ obtained by gluing countably many geodesic pairs of pants might not be complete in the induced hyperbolic metric. The boundary of the metric completion of $X$ consists of simple closed geodesics and bi-infinite simple geodesics (see \cite{Alvarez-Rodriguez},\cite{Basmajian},\cite{Basmajian-Saric}).  By attaching funnels to the closed geodesics and attaching geodesic half-planes to the bi-infinite geodesics of the boundary of the metric completion of $X$, we obtain a hyperbolic surface $\widehat{X}$ homeomorphic to $X$ with a geodesically complete hyperbolic metric such that the inclusion $X\to \widehat{X}$ is an isometric embedding. Any infinite type  hyperbolic surface can be obtained as the above by gluing of countably many geodesic pairs of pants and by attaching funnels and half-planes (see \cite{Alvarez-Rodriguez} and \cite{Basmajian-Saric}). The hyperbolic surface structure is completely determined by the length and twist parameters called {\it Fenchel-Nielsen} parameters.

We are mainly interested in determining whether a hyperbolic surface is or is not of parabolic type. The geodesic flow on the unit tangent bundle of a hyperbolic surface with a funnel preserves two disjoint open subsets and hence cannot be ergodic. Therefore a hyperbolic surface with a funnel supports a Green's function and thus it is not of parabolic type. In our  constructive approach to hyperbolic surfaces, a funnel appears only if a boundary geodesic of a pair of pants is not glued to another boundary geodesic. For this reason, we always assume that a {\it boundary component} of a pair of pants which is {\it not glued} to another boundary component is a {\it puncture}. Thus we are not considering surfaces with funnels because {they are known to not be of parabolic type}. Under our assumption, a hyperbolic surface obtained by gluing countably many geodesic pairs of pants could still be incomplete due to a possible accumulations of boundary geodesics of the pairs of pants (see \cite{Basmajian}). However,  determining for which Fenchel-Nielsen parameters precisely $X$ is incomplete appears to be a difficult problem.




\section{Modulus of a curve family}
\label{sec:modulus}

Let $X$ be an arbitrary Riemann surface which supports a conformal hyperbolic metric. Denote by $\Gamma$ a family of curves in $X$ that are locally rectifiable in the charts. A {\it metric} $\rho$ on $X$ is an assignment in each local chart $z=x+iy$ of a metric $\rho (z)|dz|$  invariant under transition maps.  We require that $\rho$ is non-negative and Borel measurable. 

A metric $\rho$ on $X$ is {\it allowable} for $\Gamma$ if the $\rho$-length satisfies
$$
\ell_{\rho}(\gamma)=\int_{\gamma}\rho (z)|dz|\geq 1
$$
for each $\gamma\in \Gamma$. If a curve $\gamma$ is not rectifiable then we set $\ell_{\rho}(\gamma )=\infty$.

\begin{definition}
The {\it modulus} $\mathrm{mod}\Gamma$ of the family $\Gamma$ is defined by
$$
\mathrm{mod}\Gamma =\inf_{\rho}\iint_X\rho ^2(z)dxdy
$$
where the infimum is over all allowable metrics $\rho$ for $\Gamma$ (see \cite{Ahlfors:QClectures,Fletcher-Markovic}).

The {\it extremal length}  $\lambda (\Gamma )$ of the curve family $\Gamma$ is defined by
$$
\lambda (\Gamma )=\frac{1}{\mathrm{mod}\Gamma} .
$$ 
\end{definition}

It is clear that any information about the modulus gives equivalent information about the extremal distance. We slightly favor the modulus for the simplicity of the subadditivity formula (compare inequality (\ref{eq:subadd-extremal-lenght}) and Lemma \ref{lemma:mod-properties}, Property 2.).  Additionally, the definition of an allowable metric, as being a metric where all curves have length at least one, makes geometric arguments more streamlined.

The modulus and the extremal length of a family of curves is invariant under conformal mappings and quasi-invariant under quasiconformal mappings (for example, see \cite{Ahlfors:QClectures}).

Let $R=\{ z:r_1<|z|<r_2\}$ be an annulus with  inner radius $r_1\geq 0$ and  outer radius $r_2\leq \infty$. Let $\Gamma$ be the family of all curves in $R$ with one endpoint on $|z|=r_1$ and the other  on $|z|=r_2$. It is well-known that the  modulus of $\Gamma$  is (see \cite{LV})
$$
\mathrm{mod}\Gamma =\frac{1}{\lambda (\Gamma )}=\frac{2\pi}{\log\frac{r_2}{r_1}}.
$$

Consider a radial segment $\{ z=re^{i\theta}:r_1<r<r_2\}$ and a conformal map $z\mapsto \log z$ defined on $R-\{ z=re^{i\theta}:r_1<r<r_2\}$. The image of $R-\{ z=re^{i\theta}:r_1<r<r_2\}$ under $z\mapsto\log z$ is the rectangle $Q=\{ z=x+iy:\log r_1<x<\log r_2,\theta <y<\theta +2\pi\}$. Let $\Gamma_Q$ be the family of all curves connecting the left and right sides of $Q$. A direct computation shows that (see \cite{LV})
\begin{equation}
\label{eq:slit-mod}
\mathrm{mod}\Gamma =\mathrm{mod}\Gamma_Q.
\end{equation}

Our goal is to recognize the conditions under which the equality (\ref{eq:slit-mod}) holds in a general doubly connected region $R$. 
Any doubly connected region $R$ on a Riemann surface is conformally equivalent to an annulus $\{r_1<|z|<r_2\}$ in the complex plane $\mathbb{C}$. Let $\Gamma$ be the family of curves connecting one components $\partial_1R$ to the other component $\partial_2R$ of the boundary of $R$. Since the modulus of a family of curves is a conformal invariant, it follows that $\mathrm{mod}\Gamma =\frac{2\pi}{\log r_2/r_1}$. 

Fix a Jordan arc $\tau$ connecting two boundary components of $R$ with  endpoints $z_i\in\partial_iR$ for $i=1,2$. 
Let $\Gamma_{\tau}$ be the family of curves in $R-\tau$ connecting  $\partial_1R$ to $\partial_2R$. Then
$
\mathrm{mod}\Gamma\geq\mathrm{mod}\Gamma_{\tau}
$
and the strict inequality is possible. However, we observe

\begin{lemma}
\label{lem:cut}
Let $R$ be a doubly connected domain and $\tau$ a Jordan arc connecting the two boundary components of $R$. If there exists an anticonformal map $c:R\to R$ which pointwise fixes $\tau$ then
$
\mathrm{mod}\Gamma =\mathrm{mod}\Gamma_{\tau}.
$
\end{lemma}

\begin{proof}
Upon conformally mapping $R$ onto an annulus, the image of $\tau$ is pointwise fixed by an anticonformal map of the annulus. Thus the image of $\tau$ is a radial segment and we obtain
$
\mathrm{mod}\Gamma =\mathrm{mod}\Gamma_{\tau}.
$
\end{proof}

Next, we list some important properties of the modulus, which will be used repeatedly throughout the paper, see \cite{LV} for the proofs of these results.

\begin{lemma}\label{lemma:mod-properties} Let $\Gamma_1,\Gamma_2,\ldots$ be curve families in $X$. Then
\begin{itemize}
  \item[1.]\textsc{Monotonicity:} If $\Gamma_1\subset\Gamma_2$ then $\m(\Gamma_1)\leq
      \m(\Gamma_2)$.
  \item[2.] \textsc{Subadditivity:} $\m(\bigcup_{i=1}^{\infty} \Gamma_i) \leq
      \sum_{i=1}^{\infty}\m(\Gamma_i).$
  \item[3.] \textsc{Overflowing:} If $\Gamma_1<\Gamma_2$ then $\m (\Gamma_1) \geq \m
      (\Gamma_2)$.
\end{itemize}
\end{lemma}

The notation $\Gamma_1<\Gamma_2$ above denotes the fact that for every curve $\g_2\in\Gamma_2$ there is a curve $\g_1\in \Gamma_1$ such that $\g_1\subset\g_2$. If this is the case we say $\Gamma_1$ \textit{minorizes} $\Gamma_2$.

The subadditivity property for the extremal length is given by
\begin{equation}
\label{eq:subadd-extremal-lenght}
\lambda (\bigcup_{i=1}^{\infty} \Gamma_i) \geq
     \frac{1}{ \sum_{i=1}^{\infty}\frac{1}{\lambda (\Gamma_i)}}.
\end{equation}

When the curve families $\Gamma_1, \Gamma_2, etc. $ have  disjoint supports (i.e. are contained in disjoint domains) the inequality in the subadditivity property turns into equality. 

\begin{lemma}
\label{lem:disjoint}
Let $\Gamma_n$, for $n=1,2,\ldots$, be at most a countable set of families of curves such that the support of any two families are disjoint. If $\Gamma =\cup_n\Gamma_n$ then
$$
\mathrm{mod}\Gamma =\sum_n\mathrm{mod}\Gamma_n.
$$
\end{lemma}

The following property is most conveniently expressed in terms of extremal length (see \cite[Section IV.3, page 135]{GM}).

\begin{lemma}[Serial rule]
Assume that $\{R_n\}_{n=1}^{\infty}$ are mutually disjoint doubly connected domains separating boundaries of a doubly connected  domain $R$. Let $\Gamma$ be the curve family connecting the two boundary components of $R$ and let $\Gamma_n$ be the curve family that connects the two boundary components of $R_n$. Then
$$
\lambda (\Gamma )\geq\sum_n\lambda (\Gamma_n).
$$
\end{lemma}



An allowable metric $\rho_1$ for a family of Jordan curves $\Gamma$ is {\it extremal} if $$\mathrm{mod}\Gamma=\int_X\rho(z)^2dxdy.$$

The following sufficient condition for extremality of a metric is known as  Beurling's criterion,  {\cite{Ahlfors:Confinvariants}.

\begin{lemma}[Beurling's criterion]
  The metric $\rho_1$ is extremal for $\Gamma$ if there is a subfamily $\Gamma_1\subset\Gamma$ such that
  \begin{itemize}
    \item $\int_{\gamma}\rho_1 |dz| = 1, \mathrm{for\ all}\ \gamma\in\Gamma_1$
    \item for any real valued $h$ on $X$ satisfying $\int_{\gamma}h |dz| \geq 0,
        \mathrm{for\ all}\ \gamma\in\Gamma_1$ the following holds
    $$ \iint_{X} h\rho_1 dxdy \geq 0.$$
  \end{itemize}
\end{lemma}

Beurling's criterion  can be applied to a family of curves consisting of vertical segments in the complex plane to find an explicit expression for the modulus of this family (note the similarity to Ahlfors' integral (\ref{Ahlfors})).

\begin{lemma}[see Lemma 4.1 in \cite{HS}]
\label{lem:vertical-modulus}
Given a measurable set $E\subset\mathbb{R}$,
let $\Gamma =\{\gamma_x\}_{x\in E}$ be a family of curves such that $\gamma_x$ is contained in a vertical line through $x$ .
Then
$$
\mathrm{mod}\Gamma =\int_E\frac{dx}{\ell_E(x)},
$$  
where $\ell_E(x)$ is the Euclidean length of $\gamma_x$.
\end{lemma}

We will also need to use the notion of  extremal distance between two boundary components of an annulus.

\begin{definition}\label{def:extremal-distance}
Let $R$ be a doubly connected region in a Riemann surface $X$. The {\it extremal distance}  between boundary components of $R$  is
$$
\lambda (R):=\lambda (\Gamma_R)=\frac{1}{\mathrm{mod}\Gamma_R} ,
$$
where $\Gamma_R$ is the curve family in $R$  connecting the boundary components.
\end{definition}



\section{Modulus of curve families between graphs}
\label{sec:between_graphs}

 Let $x_1<x_2$ and  $f,g:\mathbb{R}\to\mathbb{R}$ be two continuous periodic functions  with period $x_2-x_1$. We estimate the modulus of curves connecting the graph of $f$ to the graph of $g$. For simplicity, we assume that $f(x)>g(x)$ for all $x\in\mathbb{R}$. Define $\Pi (x+iy)=x$, and let $\mathcal{R}$ be the region bounded by $f$ and $g$. 

Let $\Gamma$ be the family of curves in $\mathcal{R}$ connecting the graphs of $f$ and $g$ such that $\Pi (\gamma (0))\in [x_1,x_2]$ for every $\gamma\in\Gamma$.  Let $0<\delta <x_2-x_1$ be fixed  and denote by $\Gamma^{\leq \delta}$ the family   of all $\gamma :[0,1]\to\mathcal{R}$ in $\Gamma$  such that $\{\Pi (\gamma (t));t\in [0,1]\}$ is an interval of length at most $\delta$. Let $\Gamma^{>\delta}$ be the family of all $\gamma\in\Gamma$ such that $\{\Pi (\gamma (t));t\in [0,1]\}$ is an interval of length  greater than $\delta$.
Then by subadditivity  (see Lemma \ref{lemma:mod-properties}, property 2)  we have,

$$\mathrm{mod}\Gamma \leq\mathrm{mod}\Gamma^{\leq\delta}+\mathrm{mod}\Gamma^{>\delta}.
$$

We first observe that $\mathrm{mod}\Gamma^{>\delta}$ can be easily estimated in terms of $\delta$ and is bounded even if the curves in $\Gamma$ degenerate. 

\begin{lemma}
\label{lem:>d}
Under the above assumptions,
$$
\mathrm{mod}\Gamma^{>\delta}\leq A/\delta^2,
$$
where $A$ is  the Euclidean area between the graphs and above $[x_1-\delta ,x_2+\delta]$.
\end{lemma}

\begin{proof}
Let $D_{\delta}$ be the region bounded by the graphs of $f$ and $g$ such that $\Pi(z) \in [x_1 - \delta, x_2 + \delta ]$. Let $\rho (z)=1/\delta$ for all $z\in D_{\delta}$, and set  $\rho (z)=0$ for $z\notin D_{\delta}$. Then $\rho$ is an allowable metric for $\Gamma^{>\delta }$ and the lemma follows. 
\end{proof}

We next estimate $\mathrm{mod}\Gamma^{\leq \delta}$. 
In \cite{HS}, an estimate for $\mathrm{mod}\Gamma^{\leq \delta}$ is given when the degeneration of the domain is done by vertical shrinking. We need an estimate for more general degeneration of the domain where not only the vertical direction is shrinking but also the shape of $f(x)$ and $g(x)$ is changing in the process.

Note that each $\gamma\in \Gamma^{\leq\delta}$ lies inside the region $D_{\delta}$, used in the proof  of Lemma \ref{lem:>d}. Define
$$
m_{\delta}(x)=m(x):=\min_{a,b\in [x-\delta ,x+\delta ]} [f(a)-g(b)] 
$$
for $x\in [x_1,x_2]$. Equivalently,  
$m_{\delta} (x)=  \min_{a \in [x-\delta ,x+\delta ]} f(a)  
-\max_{b\in [x-\delta ,x+\delta ]} g(b)$.  The quantity  
$m_{\delta}(x)$ is the height of  the tallest rectangle between the graphs of $f$ and $g$ whose vertical sides are contained in $x-\delta$ and $x+\delta$.
For fixed $\delta >0$, $m_{\delta} (x)$ is a continuous function. 

\begin{lemma}
\label{lem:allowable}
The metric $\rho$ defined by $\rho (z)=\frac{1}{m(\Pi (z))}$ for all $z$ between the graphs of $f(x)$ and $g(x)$ with $\Pi (z)\in [x_1-\delta ,x_2+\delta ]$, and by $\rho (z)=0$ elsewhere is allowable for the curve family $\Gamma^{\leq\delta}$.
\end{lemma}

\begin{proof}
Let $\gamma\in\Gamma^{\leq\delta}$. Fix $z=\gamma (t)$ for some $t\in [0,1]$ and denote by 
$I_{\delta}=[\Pi (z)-\delta , \Pi (z)+\delta ]$ 
the closed interval centered at $\Pi (z)$.  Then $\gamma$ connects the top and bottom of the rectangle $I_{\delta }\times [\max_{b\in I_{\delta}}g(b),\min_{a\in I_{\delta}} f(a)]$. Thus the Euclidean length of $\gamma$ is at least $m(\Pi (z))$.
\end{proof}

We use the above lemma to find an effective estimate for $\mathrm{mod}\Gamma^{\leq \delta}$. It turns out that the estimate is, up to a positive multiplicative constant, equal to the modulus of the vertical arcs connecting the two graphs.

 For each $\delta >0$ and each pair $(f,g)$ we set, 
\begin{align}\label{deviation}
c_{\delta}:=\inf_{x} \frac{m_{\delta} (x)}{[f(x)-g(x)]}.
\end{align}

Since $\frac{m_{\delta} (x)}{[f(x)-g(x)]}$ is continuous on 
$[x_1,x_2]$  it is easy to see that $0<c_{\delta}\leq 1$. Moreover, $c_{\delta}\to1$ as $\delta\to0$.
 Geometrically,   $\frac{m_{\delta} (x)}{[f(x)-g(x)]}$ measures  how far the region above $[x-\delta,x+\delta]$ and 
 between the graphs of $f$ and $g$  is from being a rectangle. 
 Thus $c_{\delta}$ is the largest deviation of an inscribed  rectangle  of width 
 $2\delta$ with sides parallel to the coordinate axes is from having height $(f(x)-g(x))$ for any $x \in \mathbb{R}$.  We  call  $c_{\delta}$  the 
 {\it $\delta$\--rectangle deviation} between  $f$ and $g$.

\begin{lemma}
\label{lem:mod} For any $\delta >0$,
$$
\int_{x_1}^{x_2}\frac{1}{f(x)-g(x)}dx
\leq \mathrm{mod}\Gamma^{\leq \delta}
\leq \frac{1}{c_{\delta}^2}\int_{x_1-\delta}^{x_2+\delta}\frac{1}{f(x)-g(x)}dx.
$$
\end{lemma}

\begin{proof}
Recall that $\rho (z)=1/m_{\delta} (\Pi (z))$ is allowable and we have
$$
\mathrm{mod}\Gamma^{\leq\delta}\leq\int_{x_1-\delta}^{x_2+\delta}\int _{g(x)}^{f(x)}\frac{1}{m_{\delta}(x)^2}dydx=\int_{x_1-\delta}^{x_2+\delta}\frac{f(x)-g(x)}{m_{\delta}(x)^2}dx\leq 
\frac{1}{c_{\delta}^2}\int_{x_1-\delta}^{x_2+\delta}\frac{1}{f(x)-g(x)}dx.
$$
The family $\Gamma^{v}$ of vertical segments connecting the graph of $f(x)$ to graph of $g(x)$ above the interval $[x_1,x_2]$ is a subfamily of $\Gamma^{\leq\delta}$ so that $\mathrm{mod}\Gamma^{v}\leq\mathrm{mod}\Gamma^{\leq\delta}$. Lemma \ref{lem:vertical-modulus} gives the left-hand  inequality.
\end{proof}

\begin{theorem} \label{thm: family modulus comparison}
For any  $\delta >0$,
\begin{equation} \label{eq: family modulus comparison}
\mathrm{mod} \Gamma^{v}  \leq \mathrm{mod} \Gamma  \leq 
\frac{3}{c^{2}_{\delta}}     \mathrm{mod} \Gamma^{v} + \frac{A}{\delta^2},
\end{equation}
where $A$ is equal to the Euclidean area between the graphs of $f$ and $g$ above the interval $[x_1,x_2]$. 
\end{theorem}

\begin{proof}
The left-hand side of  (\ref{eq: family modulus comparison}) follows from  $\Gamma^{v}  \subset  \Gamma$ and the monotonicity of modulus. 

By Lemma \ref{lem:>d} we have $\mathrm{mod} \Gamma^{>\delta}  \leq \frac{A}{\delta^{2}}.$
Next note that by  Lemma  \ref{lem:mod}, 

\begin{equation} \label{eq:inequality for modulus}
\mathrm{mod} \Gamma^{\leq \delta}
\leq
\frac{1}{c_{\delta}^2}
\int_{x_1-\delta}^{x_2+\delta}\frac{1}{f(x)-g (x)}dx
\leq  
\frac{3}{c_{\delta}^2} \int_{x_1}^{x_2}\frac{1}{f(x)-g (x)}dx,
\end{equation}
where the second inequality  above follows from breaking the integral  over the intervals,
$[x_1-\delta,x_1]$,  $[x_1,x_2]$, $[x_1-\delta,x_1]$ and using the periodicity of $f$ and $g$. Now by the 
Beurling's criterion, Lemma \ref{lem:vertical-modulus},
the  last integral is $\mathrm{mod} \Gamma^{v}$. 

Finally, using the fact that
 $\mathrm{mod} \Gamma \leq 
  \mathrm{mod} \Gamma^{\leq \delta}
+ \mathrm{mod} \Gamma^{>\delta}$ yields the right-hand side of (\ref{eq: family modulus comparison}).
 \end{proof}

\subsection{Modulus of degenerating families of curves}
\label{sec:deg-fam}

 Fix $\ell_0 \geq 0$ and $[x_1,x_2]\subset\mathbb{R}$.  Consider a  setting  where we have a family of  continuous periodic function pairs  $\{(f_\ell, g_\ell)\}$  all having the same period $x_2-x_1$  and depending  on a positive real parameter $\ell \geq \ell_0$.  In  later applications of the results in this section, the family of periodic functions depends on the lengths of simple closed geodesics and as such  we use the subscript 
 $\ell$ as a common notation for the length. 
 
\begin{definition}
A family of continuous periodic function pairs 
$\{(f_\ell, g_\ell)\}_{\ell \geq \ell_0}$ as above is {\it  simply degenerate} if 
\begin{enumerate}
\item $f_{\ell}(x)>g_{\ell}(x)$ for  all $x \in \mathbb{R}$ and  
$\ell \geq \ell_0$,
\item  the graphs of $(f_\ell, g_\ell)$  pointwise go  to 0 as $\ell \rightarrow \infty$ and  

\item the area bounded by their graphs above the interval $[x_1,x_2]$ is at most  1,  for all $\ell$.  
 \end{enumerate}
 \end{definition}
 The choice of   $1$ in condition (3)  is not crucial, and serves merely as a matter of convenience, as long as it is finite.
 
 \begin{remark}  
 \label{rem:mod-gamma_l}
Given a simply degenerate family $\{(f_{\ell},g_{\ell})\}$, let $\Gamma_{\ell}$ be the curve family 
connecting the	 part of the graph of $f_{\ell}$ over $[x_1,x_2]$ to the graph of $g_{\ell}$, i.e., $\Pi(\gamma(0))\in[x_1,x_2]$,  and let $\Gamma^{v}_{\ell}$
be the vertical subfamily of $\Gamma_{\ell}$. We observe that as $\ell\to\infty$ we have
$$\mathrm{mod} \Gamma_{\ell}\geq\mathrm{mod} \Gamma_{\ell}^v\to\infty.$$ 
Since the first inequality follows from the monotonicity of modulus it is enough to show that $\mathrm{mod} \Gamma_{\ell}^v\to\infty$. For that let 
$M_{\ell}=\mathrm{max}\{ |f_\ell (x)-g_\ell (x)| : x \in [x_1,x_2]\},$ and note that  $M_{\ell} \rightarrow 0$.  Therefore, by Lemma \ref{lem:vertical-modulus} we have
$\mathrm{mod} \Gamma_{\ell}^v \geq M_{\ell}^{-1}{|x_2-x_1|}\to\infty,$ as desired.

%

 \end{remark}
 

Next we formulate a condition that implies that the modulus of the vertical curve family $\Gamma_{\ell}^v$ is comparable to the modulus of the full family $\Gamma_{\ell}$,  as $\ell$ goes to infinity. 

Recall, that $c_{\delta, \ell}$ is the $\delta$\--rectangle deviation between $f_{\ell}$ and $g_{\ell}$ as in (\ref{deviation}).

\begin{corollary} \label{cor: families comparable}
Suppose $\{(f_\ell, g_\ell)\}_{l\geq l_0}$ is a  simply degenerating  family. 
If there exists a  positive real-valued function 
$\delta=\delta({\ell})$ bounded above by the period $x_2-x_1$ so that, 
\begin{enumerate}
\item d=$\inf_{\ell \geq \ell_0}  \left[ \left[\delta(\ell)\right]  \right]^2  \mathrm{mod} \Gamma^{v}_{\ell}]
>0$,
\item c= $\inf_{\ell \geq \ell_0}  \left( c_{\delta(\ell), \ell} \right)>0$,
\end{enumerate}
then  for all $\ell \geq \ell_0$
$$1\leq \frac{\mathrm{mod} \Gamma_{\ell} }{  \mathrm{mod} \Gamma^{v}_{\ell}}\leq \frac{3}{c^2}+\frac{3}{d} .$$ 
\end{corollary}

\begin{proof}
For $\ell \geq \ell_0$, plugging  
$\delta=\delta({\ell})$ into inequality   (\ref{eq: family modulus comparison}) of  Theorem \ref{thm: family modulus comparison}, noting that the area between the graphs and above the interval $[x_1-\delta ,x_2+\delta ]$ is at most $3$, and dividing  by $\mathrm{mod} \Gamma^{v}_{\ell}$  
  we obtain,
\begin{equation}\label{ineq:mod-comparison}
1 
\leq \frac{\mathrm{mod} \Gamma_{\ell}}{\mathrm{mod} \Gamma^{v}_{\ell}} 
\leq \frac{3}{c^{2}_{\delta(\ell), \ell}}+\frac{3}{\left[\delta(\ell)\right]^{2}  \mathrm{mod} \Gamma^{v}_{\ell} }.
\end{equation}
The right-hand  side, by our assumptions,  is bounded above by 
$\frac{3}{c^{2}}+\frac{3}{d}$ and we are done.
\end{proof}

\begin{remark}  
There are simply degenerate families of functions where we must allow	 that
 $\delta({\ell})\to 0$  as $\ell \rightarrow \infty$ in order to be able to apply Corollary \ref{cor: families comparable}.
\end{remark}



\begin{example}
\label{ex:non-example}
We next give an example to show that item  (2) in Corollary \ref{cor: families comparable} is necessary for (\ref{ineq:mod-comparison}) to hold.
That is, a judicious choice of $\delta({\ell})$ satisfying the hypotheses of Corollary \ref{cor: families comparable}
does not always exist. Namely we give a simply degenerating family for which  the ratio of $\mathrm{mod} (\Gamma_{\ell})$  to  $\mathrm{mod} (\Gamma^{v}_{\ell})$ goes to infinity as $\ell\to\infty$. 



We first start with a computation involving a degenerating family of generalized quadrilaterals as $\epsilon\to 0$. Given  $\epsilon\in (0,1)$, let $\Omega_{\epsilon}$ be the domain obtained from the rectangle $[0,1]\times [0,\epsilon ]$ in $\mathbb{C}$ by removing  the closed segments $L_k=\{ \frac{k}{N_{\epsilon}}+ti:\epsilon^2\leq t\leq \epsilon\}$ for $k=2,\ldots ,N_{\epsilon}-1$, where $N_{\epsilon}>\frac{1}{2(\epsilon -\epsilon^2)}$ (see the grey region in Figure \ref{fig:sharpness}).


\begin{figure}[h]
\includegraphics[scale=1.5]{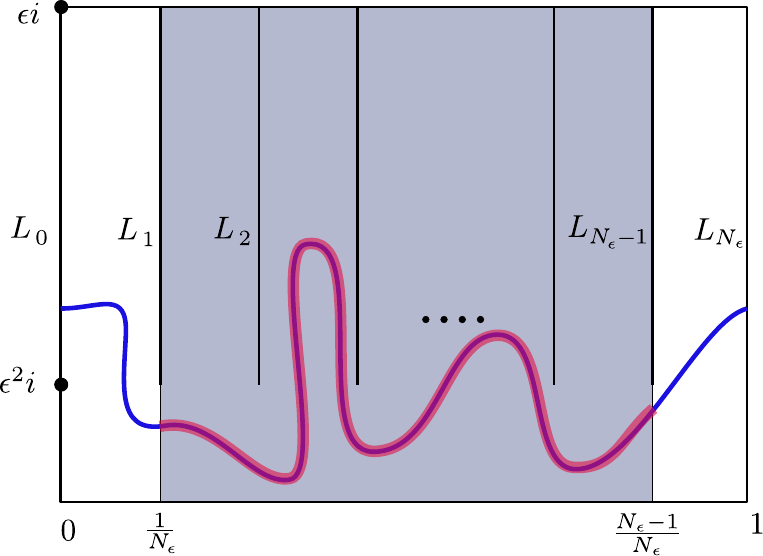}
\caption{The gray domain $\Omega_{\epsilon}'$ is a subset of $\Omega_{\epsilon}$ whose vertical sides contain $L_1$ and $L_{N_{\epsilon}-1}$. Any curve in $G_{\epsilon}$ contains a curve in $G_{\epsilon}'$.}\label{fig:sharpness} 
\end{figure}

The domain $\Omega_{\epsilon}$ has two vertical sides and the complement of the vertical sides of the boundary of  $\Omega_{\epsilon}$ has two components: the top and the bottom. The bottom component is the interval  $[0,1]$ on the real axis, and the top  component is the union of $[0,1]\times\{\epsilon\}$ and the segments $L_k$.

Let $\Gamma_{\epsilon}$ be the family of arcs in $\Omega_{\epsilon}$ connecting top to bottom.
Let $\Gamma_{\epsilon}^{v}$ be the curves in $\Gamma_{\epsilon}$ that are vertical. 

Let $G_{\epsilon}$ be the family of curves in $\Omega_{\epsilon}$ connecting the vertical sides. Let $\Omega_{\epsilon}'=\Omega_{\epsilon}\cap [\frac{1}{N_{\epsilon}},\frac{N_{\epsilon}-1}{N_{\epsilon}}]\times [0,\epsilon ]$. Let $G_{\epsilon}'$ be the family of curves in $\Omega_{\epsilon}'$ that connect
$\{\frac{1}{N_{\epsilon}}+ti: 0\leq t\leq \epsilon^2\}$ and $\{\frac{N_{\epsilon}-1}{N_{\epsilon}}+ti: 0\leq t\leq \epsilon^2\}$ as in Figure \ref{fig:sharpness}.

Note that $G_{\epsilon}'<G_{\epsilon}$ hence $\mathrm{mod} G_{\epsilon}'\geq\mathrm{mod} G_{\epsilon}=\frac{1}{\mathrm{mod} \Gamma_{\epsilon}}$. By \cite[Lemma 5.4]{HS:PLMS} if $N_{\epsilon}=[\frac{1}{\epsilon^2}]$ we have
$
\mathrm{mod} G_{\epsilon}'\leq {C}{\epsilon^2}
$
for some $C>0$. 
Therefore
$
\frac{\mathrm{mod}\Gamma_{\epsilon}^{v}}{\mathrm{mod} \Gamma_{\epsilon}}\leq {C}{\epsilon}\to 0
$
as $\epsilon\to 0$.

The top side of the domain $\Omega_{\epsilon}$ can be approximated by the graph of a continuous function such that $
\frac{\mathrm{mod}\Gamma_{\epsilon}^{v}}{\mathrm{mod} \Gamma_{\epsilon}}\leq {2C}{\epsilon}$ and we have the same conclusion for a simply degenerate family of functions.
\end{example}

\begin{remark}
The construction in Example \ref{ex:non-example} gives  regions not bounded by graphs of continuous functions which do not satisfy item 2) in  Corollary \ref{cor: families comparable}. We point out that similarly constructed regions which are not bounded by graphs of continuous functions may satisfy the assumption of  Corollary \ref{cor: families comparable} and one can prove that  Corollary \ref{cor: families comparable} applies to these more general regions although we do not pursue this here. \end{remark}


\section{Modulus of half-collars}
\label{sec:non-standard-collar}

Let $X$ be a Riemann surface endowed with its conformal hyperbolic metric and  $\alpha$  a simple closed geodesic on $X$.
A {\it collar} about  $\alpha$ is an annular  (doubly-connected) open neighborhood of $\alpha$. A {\it half-collar} about $\alpha$ is an annular neighborhood with 
 $\alpha$  being one of  its  boundary components. In this section we will compute the extremal distances between the two boundary components of standard and non-standard collars around simple closed geodesics using the results from Section \ref{sec:between_graphs}.

\subsection{General collars}
The most direct way for computing the extremal distance between the two boundary components of a collar is to use Lemma \ref{lem:cut}. 
The fixed point set of the reflection symmetry of a general collar is not easily identified. For this reason we lift the collar to the universal covering and make use of a curve family in the universal covering which is related to the curve family connecting the two boundaries of the collar on the surface $X$.

Let $\alpha$ be a simple closed geodesic on a conformally hyperbolic Riemann surface $X$ of length $\ell$ and let $R$ be a collar or half-collar about $\alpha$. Fix the universal covering of $X$ to be the upper half-plane $\HH$ such that the positive $y$-axis covers $\alpha$. Let $\tilde{R}$ be the image under $z(w)=\frac{1}{\ell}\log w$ of the component of the pre-image of $R$ in $\HH$ that contains the positive $y$-axis. Note that $\tilde{R}$ is a universal cover of $R$ with a covering translation $h(z)=z+1$ such that $\tilde{R}/<h>=R$, where $<h>$ is the cyclic group generated by $h$. 

\begin{definition}
We will call $\tilde{R}$ the {\it universal cover of $R$ in  logarithmic coordinates}. 
\end{definition}

The domain $\tilde{R}$  lies between graphs of two functions $f,g:\mathbb{R}\to\mathbb{R}$ such that $f(x+1 )=f(x)$, $g(x+1 )=g(x)$ and $ f(x)> g(x)>0$ for all $x\in\mathbb{R}$ (see Figure \ref{fig:log coordinates}).

\begin{lemma}
\label{lem:lift}
Given a collar or half-collar $R$ about a simple closed geodesic $\alpha$ on a Riemann surface $X$, let $\tilde{R}$ be  universal cover of $R$ in  logarithmic coordinates and  $I$ a fundamental interval for the action of $<h>$ on one boundary component of $\tilde{R}$. Denote by $\Gamma_R$ the curve family in $R$ that connects the two boundary components of $R$. 
Consider the curve family $\Gamma$ in $\tilde{R}$ starting in $I$ and ending at  the other boundary component. Then
$$\frac{1}{3}\mathrm{mod}\Gamma -\frac{2}{3}{A}\leq\mathrm{mod} \Gamma_R\leq \mathrm{mod} \Gamma ,$$
where $A$ is equal to the Euclidean area of the part of $\tilde{R}$ over the interval $I$.
 \end{lemma}

\begin{proof}
The domain $\tilde{R}$  lies between graphs of two functions $f,g:\mathbb{R}\to\mathbb{R}$ such that $f(x+1 )=f(x)$, $g(x+1 )=g(x)$ and $ f(x)> g(x)$ for all $x\in\mathbb{R}$. Without loss of generality we assume that $I$ lies on the graph of $f$. Let $\Omega$ be the set of points in $\tilde{R}$ below ${I}$. Then $\Omega$ is a fundamental set for the action of $<h>$. 
 
 Assume that $\rho (z)|dz|$ is an allowable metric for the family $\Gamma$.  We define a metric $\bar{\rho}(z)|dz|$ for $z\in\Omega$ by
 $$
 \bar{\rho}(z)=\sqrt{\sum_{k=-\infty}^{\infty}[\rho (z+k )]^2}.
 $$
 and denote its projection to $R$ by $\bar{\rho}$ again. Since $\iint_{\tilde{R}
 }\rho^2(z)dxdy=\iint_{\Omega}\bar{\rho}^2(z)dxdy$ we have that the series defining $\bar{\rho}$ converges almost everywhere when $\iint_{\tilde{R}
 }\rho^2(z)dxdy<\infty$.
 
 Let $\gamma\in\Gamma_R$ and we compute its $\bar{\rho}$-length. Consider the lift $\tilde{\gamma}$ of $\gamma$ that starts on ${I}$ and ends on the graph of $g$. Then $\tilde{\gamma}\in\Gamma$. Lift $\bar{\rho}$ to $\widetilde{\bar{\rho}}$. Then the $\bar{\rho}$-length of $\gamma$ is equal to $\widetilde{\bar{\rho}}$-length of $\tilde{\gamma}$. 
 
 We divide $\tilde{{\gamma}}$ into arcs $\{\tilde{{\gamma}}_k\}_k$ that lie in different translates $\Omega_k:=h^k(\Omega )$ of $\Omega$ by elements of the  group $<h>$. On each $\Omega_k$, we have
 $$
\widetilde{\bar{\rho}} (z)\geq \rho (z).
$$
Therefore $l_{{\bar{\rho}}} (\gamma )= l_{\widetilde{\bar{\rho}}}(\tilde{\gamma})\geq l_{\rho}(\tilde{\gamma})\geq 1$ and the $\bar{\rho}$ metric is allowable for $\Gamma_R$. Since
$$
\iint_R\bar{\rho}^2(z)dxdy=\sum_{k=-\infty}^{\infty}\iint_{\Omega_k} \rho^2(z)dxdy=\iint_{\tilde{R}}\rho^2(z)dxdy
$$
we have that
$$
\mathrm{mod}\Gamma_R\leq \mathrm{mod}\Gamma .
$$ 

Assume that $\bar{\rho }(z)|dz|$ is an allowable metric for the family $\Gamma_R$. 
Let $\Omega_{1}$ be the subdomain of $\tilde{R}$ which contains $\Omega$ and  points whose $x$-coordinate differ by at most $1$ from the  $x$-coordinate of a point in $\Omega$.  
We take a lift $\rho (z)|dz|$ of $\bar{\rho}(z)|dz|$ to the region $\Omega_{1}$. Then we define a metric $\rho_{1}(z)|dz|$ on the universal covering by
\[
\rho_{1}(z)= \left\{
\begin{array}{ll}
\rho (z) &  ,\mathrm{ for }\ z\in\Omega\\
\max\{ \rho (z),1\} &   ,\mathrm{ for }\ z\in\Omega_{1}\setminus\Omega\\
0 &  ,\mathrm{ otherwise}\\
\end{array}\right. \]

The metric $\rho_{1}(z)|dz|$ is allowable for the family $\Gamma$. This is because any curve  $\gamma\in\Gamma$ is either completely contained in  $\Omega_{1}$ or connects the vertical sides of one of the components in
$\Omega_{1} - \Omega$. In the first case, 
we have $\ell_{\rho_{1}}(\gamma) \geq  
\ell_{\rho}(\gamma) =  \ell_{\bar {\rho}}(\pi(\gamma)) \geq 1,$ 
where $\pi : \tilde{R}  \rightarrow  R$  is the covering map. 
In the second case, $\ell_{\rho_{1}}(\gamma) \geq
\int_{\gamma \bigcap \{\Omega -\Omega_{1}\}}
|dz|\geq  1. $

We have
$$
\iint_R\bar{\rho}^{2}(z)dxdy=\iint_{\Omega}\rho^{2} (z)dxdy\geq\frac{1}{3}\iint_{\Omega_{1}}\rho^{2} (z)dxdy
\geq \frac{1}{3}\iint_{\Omega_{1}}\rho^{2}_{1}(z)dxdy-\frac{2}{3}{A}
$$
where the last inequality follows by 
$\rho^{2} _{1} \leq \rho^{2} +1$ on 
$\Omega_{1} -\Omega$ and $A$ is the area of $\Omega$. 
Taking the infimum over all allowable $\bar{\rho}$ we obtain
$$
 \mathrm{mod}\Gamma_R\geq \frac{1}{3} \mathrm{mod}\Gamma -\frac{2}{3}{A}.
 $$
\end{proof}

From now on we use subscript $\ell$ to emphasize the dependence on the length $\ell$ of the closed geodesic $\alpha$. Let $f_{\ell}:\mathbb{R}\to\mathbb{R}$ and $g_{\ell}:\mathbb{R}\to\mathbb{R}$ be the functions whose graphs are the upper and lower boundaries of the universal covering $\tilde{R}$. Note that $f_{\ell}(x+1)=f_{\ell}(x)$ and $g_{\ell}(x+1)=g_{\ell}(x)$.
The above lemma compares the modulus of the curve family $\Gamma_R$ connecting the two boundary components of $R$ to the modulus of a curve family $\Gamma_\ell$ in the universal covering $\tilde{R}$ in the logarithmic coordinates which connects a fundamental interval on the graph of $f_\ell$ to the graph of $g_\ell$ inside $\tilde{R}$. Since the modulus of $\Gamma_\ell$ cannot be directly computed, we compare it to the modulus of a family of vertical curves in the universal covering.
We give the estimate under the assumption that the length $\ell$ of a closed geodesic $\alpha$ is bounded below by some fixed $\ell_0>0$.

 Denote by $\Gamma^{v}_{\ell}$ the vertical family of curves in $\tilde{R}$ between these  graphs and above $[-\frac{1}{2},\frac{1}{2}]$.  Recall that  
$c_{\delta(\ell), \ell}$ measures, up to scale $\delta(\ell)$,  how far the area between the graphs are  from being a rectangle.  See section \ref{sec:between_graphs} above Corollary \ref{cor: families comparable} for the precise definition. 
In the following corollary we put together Corollary \ref{cor: families comparable} and  Lemma \ref{lem:lift}  to derive a criteria for the $\mathrm{mod} \Gamma_R$ to be comparable to   
 $\mathrm{mod}  \Gamma^{v}_{\ell}$.

Here and in what follows the notation $$a\asymp b$$ means that there is a constant $1\leq C<\infty$ such that $C^{-1}\leq a/b \leq C$.

\begin{corollary}
\label{cor:collar_ext}
 For each $\ell \geq \ell_0$, let $R$ be a  collar or half-collar about a geodesic of length $\ell$ and let $\Gamma_R$ be the curve family connecting the two boundary components of $R$.  Let $\{(f_\ell, g_\ell)\}_{\ell\geq \ell_0}$ be  the lifts of the boundary components of $R$ to the universal covering $\tilde{R}$ in  logarithmic coordinates such that $\{(f_\ell, g_\ell)\}_{\ell\geq \ell_0}$ is a simply degenerate family.
 If there exists a  positive real-valued function 
$\delta=\delta({\ell})$ so that, 
\begin{enumerate}
\item $d=\inf_{\ell \geq \ell_0}  {\delta(\ell)^2}{ \mathrm{mod} \Gamma^{v}_{\ell}}>0$,
\item $c=\inf_{\ell \geq \ell_0}  \left( c_{\delta(\ell), \ell} \right)>0$,
\end{enumerate}
then   
$${\mathrm{mod} \Gamma_R  }\asymp {  \mathrm{mod}\Gamma^{v}_{\ell}}$$
when $\ell\to\infty$. The bound on the constant of the above comparison depends on $c$ and $d$ and it goes to infinity when either of them goes to zero.
\end{corollary}

\begin{proof} Since 
$\{(f_\ell, g_\ell)\}_{\ell\geq \ell_0}$ is a simply degenerating family, Remark \ref{rem:mod-gamma_l} shows that $\mathrm{mod}\Gamma^{v}_{\ell}\to \infty$ as $\ell\to\infty$.

Since $A\leq 1$, Lemma \ref{lem:lift}  gives,

\begin{equation} \label{eq: equation 1}
\frac{1}{3}\left[\mathrm{mod}\Gamma_{\ell}-
{2}{}\right]
 \leq \mathrm{mod}\Gamma_R \leq \mathrm{mod}\Gamma_{\ell}  
\end{equation}

Corollary \ref{cor: families comparable} applied to both sides of 
equation (\ref{eq: equation 1}) gives us

\begin{equation} 
\frac{1}{3}\left[\mathrm{mod}\Gamma_{\ell}^v-
2\right]
 \leq \mathrm{mod}\Gamma_R \leq
 (\frac{3}{c^2}+\frac{3}{d} ) \mathrm{mod}\Gamma_{\ell}^v 
\end{equation}

and hence
$$ \frac{1}{3}\left[1-
\frac{2}{ \mathrm{mod}\Gamma_{\ell}^v}\right]
\leq 
\frac{\mathrm{mod} \Gamma_R}{\mathrm{mod}\Gamma_{\ell}^v}
\leq  
\frac{3}{c^2}+\frac{3}{d} 
$$

Since  $\mathrm{mod}\Gamma^{v}_{\ell}\to \infty$ as $\ell\to\infty$, we have $ \frac{1}{3}\left[1-
\frac{2}{ \mathrm{mod}\Gamma_{\ell}^v}\right]
\geq\frac{1}{6}$ for all $\ell\geq \ell_1>0$ with $\ell_1$ large enough and 
the result follows for $\ell\geq \ell_1$. The result follows for $\ell_0\leq \ell\leq\ell _1$ by  continuity.
\end{proof}

\begin{remark}
The family  $\Gamma^{v}_{\ell}$ corresponds to  the set of geodesics between boundary components of $R$  which are orthogonal to the  core geodesic of $R$. In particular,
 the modulus of  these orthogonals is the same as the modulus of $\Gamma_{\ell}^v$. Hence, Corollary \ref{cor:collar_ext} could be rephrased purely  in  hyperbolic terms on the  collar.

\end{remark}

\subsection{Standard collars}
The   {\it standard collar} about  $\alpha$  is the set of points   a distance less than $r (\frac{\ell (\alpha)}{2})$  from $\alpha$, where
\begin{align}\label{def:r}
r(x):=\sinh^{-1}\left(\frac{1}{\sinh x}\right)
\end{align}

  The standard collar is bounded by  two  equidistant curves. It is well-known that a standard collar always exists and disjoint simple closed geodesics have disjoint standard collars, see \cite{Buser}. The {\it standard half-collar} consists of the points on one side of the standard collar. Note that by (\ref{def:r}) for large $x$ we have the following asymptotics 
\begin{align}\label{asymptotics:r-big}
\begin{split}
r(x) &\asymp {1/e^{x}}, \quad \mbox{ as }  x\to\infty.
\end{split}
\end{align}
 On the other hand, since $\sinh^{-1}(t) = \ln(t+\sqrt{t^2+1})$, from (\ref{def:r}) we have 
\begin{align}\label{asymptotics:r-small}
\begin{split}
r(x)=\ln\left(\frac{1+\cosh x}{\sinh x}\right) = \ln\left(\frac{2}{x}\right)+o(x), \quad \mbox{ as } x\to 0.
\end{split}
\end{align}

The extremal length of the curve family $\Gamma_R$ connecting the two boundary components of the standard full  collar neighborhood $R$  about $\alpha$ was  computed by Maskit, see \cite{Maskit}. For the convenience of the reader, we give the computation below.

\begin{lemma}
\label{lem:collar}
 Let $R$ be the standard half-collar about a simple closed geodesic of length $\ell$. Then
\begin{align}\label{eqn:Maskit}
\lambda (\Gamma_R)=\frac{1}{\ell}\arctan \Big{[}\frac{1}{\sinh \frac{\ell}{2}}\Big{]}.
\end{align}
\end{lemma}

\begin{proof}
By the  collar lemma  the geodesic $\alpha$ of length $\ell$ has a one-sided collar of length $\rho =\sinh^{-1}\frac{1}{\sinh \frac{\ell}{2}}$,  see for instance \cite[page 94]{Buser}.
We lift the collar $R$ to the upper half-plane so that $\alpha$ is lifted to the geodesic with endpoints $0$ and $\infty$. One lift $\tilde{R}$ of $R$ is between two hyperbolic geodesics orthogonal to the $y$-axis,  intersecting the imaginary axis at points $iy_1$ and $i y_2$, with $0<y_1<y_2$, and between a Euclidean ray from $0$ subtending angle $\theta$ with the $y$-axis. Thus, $\tilde{R}$ is an annular sector 
$$\tilde{R}=\{re^{i\phi} : \phi\in(\pi/2-\theta,\pi/2), r\in(y_1,y_2)\}.$$
Therefore $\lambda(\Gamma_R) = \lambda(\tilde{\Gamma}_R)$, where $\tilde{\Gamma}_R$ is the family of curves connecting the rays $\{ re^{i\phi}:  \phi = \pi/2, r>0 \} $ and $\{ re^{i\phi}:  \phi = \pi/2-\theta, r>0 \}$ in the quadrilateral $\tilde{R}$. The standard extremal length formula for curves in an annular sector, cf.  \cite[Remark 7.7]{Vaisala}, then gives us
\begin{align*}
\lambda(\tilde{\Gamma}_R) =\frac{\theta}{\ln\frac{y_2}{y_1}} = \frac{\theta}{\ell}.
\end{align*}
Finally, the last equality implies  (\ref{eqn:Maskit}), 
since $\tan \theta=\sinh\rho ={1}/{\sinh \frac{\ell}{2}},$
cf.  \cite[p. 162]{Beardon}.
%
\end{proof}


We note here that from (\ref{eqn:Maskit}) we have the following asymptotic behavior for the extremal distance $\lambda (R):=\lambda (\Gamma_R)$ (see Definition \ref{def:extremal-distance}) between the boundary components of the standard collar $R$ about a simple closed geodesic of length $\ell$:

\begin{align}\label{asymp:standard}
\begin{split}
\lambda (R)&\asymp \frac{1}{\ell e^{\ell/2}}, \quad\mbox{ as } \ell\to\infty,\\
\lambda (R)&\asymp \frac{1}{\ell}, \quad \mbox{ as } \ell\to0.
\end{split}
\end{align}

\subsection{Nonstandard half-collars}
\label{sec:nonstandard}

One of the main objects in our study is what we call a {\it nonstandard half-collar}. Let $P$ be a geodesic pair of pants with three boundary geodesics $\alpha$, $\alpha_1$ and $\alpha_2$.  Fix a boundary geodesic $\alpha$ of $P$ and  an orthogeodesic   $\gamma$ from 
$\alpha$ to one of the other boundary geodesics $\alpha_1$ of $P$.  Let $b$ be the endpoint of $\gamma$ located on the other geodesic $\alpha_1$. Since every geodesic pair of pants has a unique decomposition into two right angled hexagons, it follows that $b$ lies on two such identical right angled hexagons. On each such hexagon we drop  a perpendicular from $b$  to the other simple orthogeodesic emanating from $\alpha$ (see Figure \ref{fig:collars}).  
 Then the union of these two  perpendiculars is a piecewise geodesic loop $\beta$ with a non-smooth point at $b$, and the annular domain bounded by   $\alpha$ and $\beta$ is what we call the  {\it nonstandard half-collar} $R_{\alpha ,\gamma}$ about $\alpha$.  
 
 When a geodesic pair of pants $P$ has a puncture $\alpha_1$, then $\gamma$ is a geodesic ray orthogonal to $\alpha$ that converges to the puncture. In this case the geodesic loop $\beta$ becomes a bi-infinite geodesic which converges to the puncture in both directions and is orthogonal to the orthogeodesic between $\alpha$ and the third boundary component of $P$ (see the right side of Figure \ref{fig:comparing-collars}). Again $R_{\alpha ,\gamma }$ is the annular domain between $\alpha$ and $\beta$.



%

There are three parameters associated with a nonstandard half-collar: the length $\ell (\alpha )$ of  $\alpha$, the length $\ell (\gamma )$ of the orthogeodesic $\gamma$, and the length $\ell (\eta )$ of the geodesic segment $\eta$ from the boundary geodesic $\alpha$ to the geodesic loop $\beta$. 
Using basic hyperbolic geometry, the three quantities are related by (see  \cite[Theorem 2.3.1 (iv)]{Buser})
\begin{equation}
\label{eq: 3 quantities}
\tanh \ell (\eta ) =\frac{\tanh \ell (\gamma )}{\cosh \ell (\alpha )/2}.
\end{equation}
So the nonstandard half-collar 
can be parametrized (determined) by the length of the boundary geodesic $\ell (\alpha )$ and the length $\ell (\gamma )$ with the constraint 
$\ell (\gamma )> r (\frac{\ell (\alpha )}{2})$ coming from the collar lemma.

When $\ell (\gamma )=\infty$, that is when the orthogeodesic goes out a cusp, the nonstandard half-collar contains the standard half-collar. On the other hand, for  $\ell (\gamma)<\infty$, using the quadrilateral formula in hyperbolic geometry, one can show that  
\begin{equation}
\label{eq:eta-alpha}
\ell (\alpha ) \leq 2r(\ell (\eta )).
\end{equation} 
Therefore, in this case, it is not hard to see that neither collar contains the other, see Figure  \ref{fig:comparing-collars}.
 Nevertheless, using the above notation  we have the following result.
 

 \begin{theorem}
\label{thm:half-collar}
Let $R_{\alpha ,\gamma}$ be the  nonstandard  half-collar about a boundary geodesic $\alpha$ of length $\ell (\alpha ) > 1$ on a pair of pants. Then
\begin{align}\label{non-standard:main}
\lambda (R_{\alpha ,\gamma}) \asymp \ell (\eta), 
\end{align}
where $\ell (\eta )$  is given by equation  (\ref{eq: 3 quantities}).

\end{theorem}

\begin{remark}
Using (\ref{eq: 3 quantities}), the estimate (\ref{non-standard:main}) can be rephrased in terms of the lengths $\ell (\alpha )$ and $\ell (\gamma)$ rather than $\ell(\eta)$, see Corollary \ref{cor1} below. In particular, in contrast to (\ref{asymp:standard}) we have 
\begin{align}
\lambda(R_{\alpha,\gamma}) \asymp \frac{1}{e^{\ell(\alpha)/2}},
\end{align} 
as $\ell(\alpha)\to\infty$,
 for $\ell(\gamma)>1$.
\end{remark}
 

%
%
%
%

\begin{proof}

Consider a  lift  of  $R_{\alpha ,\gamma}$ to  $\mathbb{H}$ such that the lift of $\alpha$ is on the imaginary axis, and the lift of $\beta$ lies on the semicircle $C$,  as indicated in the right top part of Figure \ref{fig:lift}. The semicircle $C$  lies on the circle 
 \begin{equation}
 |w-\cosh r(\ell (\eta ))|=\sinh r(\ell (\eta )).
 \end{equation}
The shaded region in the right top part of Figure \ref{fig:lift} is the fundamental domain for the action of the covering transformations.
%

%
%
The conformal map $z=\frac{1}{\ell (\alpha )} \log w$ sends  the shaded region in Figure \ref{fig:lift} onto a region bounded by the graphs of functions
$$f_{\ell (\alpha )}(x)=\frac{\pi}{2\ell (\alpha )},$$
$$g_{\ell (\alpha )}(x)=\frac{1}{\ell (\alpha )}\cos^{-1}\frac{\cosh x\ell (\alpha )}{\cosh r(\ell (\eta ))}$$
and vertical lines $x=-1/2$ and $x=1/2$.

The inequality $\frac{\pi}{2}-\cos^{-1}t\geq t$ gives
$$
g_{\ell (\alpha )}(x)\leq \frac{\pi}{2\ell (\alpha )}-\frac{1}{\ell (\alpha )}\frac{\cosh x\ell (\alpha )}{\cosh r(\ell (\eta ))}\leq \frac{\pi}{2\ell (\alpha )}-\frac{1}{2\ell (\alpha )}\frac{e^{|x|\ell (\alpha )}}{e^{r(\ell (\eta ))}}=:h_{\ell (\alpha )}(x).
$$

Set  $\delta (\ell (\alpha ) )=\frac{1}{\ell (\alpha )}$. Let $R_{\alpha ,\gamma}'$ be the sub-annulus of $R_{\alpha ,\gamma}$ whose lifted boundary components are the graphs of the functions $f_{\ell (\alpha )}(x)$ and $h_{\ell (\alpha )}(x)$. We first estimate $\lambda (R_{\alpha ,\gamma}')$ using Corollary \ref{cor:collar_ext}. For $\delta_2 \in [-\delta, \delta]$ we have,
\begin{align}\begin{split}
m(x)& =  \frac{\pi}{2\ell (\alpha )}-h_{\ell (\alpha )}(x+\delta_2)
= \frac{e^{-r(\ell (\eta ))}}{2\ell (\alpha )}  e^{|x+\delta_2 | \ell (\alpha )}\\
&\geq  \frac{e^{-r(\ell (\eta ))}}{2\ell (\alpha )}  e^{(|x|-|\delta_2| ) \ell (\alpha )}
 =\frac{e^{-r(\ell (\eta ))}}{2\ell (\alpha )}  e^{|x|\ell (\alpha )}  e^{-|\delta_2|  \ell (\alpha )} \\ 
 &\geq  \frac{e^{-r(\ell (\eta ))}}{2\ell (\alpha )}  e^{|x|\ell (\alpha )}  e^{-1}
=(f_{\ell (\alpha )}(x)-h_{\ell (\alpha )}(x)) e^{-1}.
\end{split}\end{align}
This gives $c\geq e^{-1}$. Let $\Gamma_{\alpha ,\gamma}^{'v}$ be the vertical family above $[-\frac{1}{2},\frac{1}{2}]$ between the graphs of $f_{\ell (\alpha )}(x)$ and $h_{\ell (\alpha )}(x)$. Then we have
\begin{align}
\begin{split}
\mathrm{mod}\Gamma_{\alpha ,\gamma}^{'v}
&=\int_{-1/2}^{1/2}\frac{1}{f_{\ell (\alpha )}(x)-h_{\ell (\alpha )}(x)}dx\\
&= \int_{-1/2}^{1/2}\frac{2\ell (\alpha )}{e^{-r(\ell (\eta ))}e^{|x|\ell (\alpha )}}dx\\
&=4(1-e^{-\frac{1}{2}\ell (\alpha )})e^{r(\ell (\eta ))} >  4(1-e^{-\frac{1}{2}\ell (\alpha )})e^{\frac{\ell (\alpha )}{2}}
\end{split}
\end{align}
This implies that $b=\inf_{\ell (\alpha )\geq 1}  {\delta(\ell (\alpha ))^2}{ \mathrm{mod} \Gamma^{'v}_{\alpha ,\gamma}}>0$ and the conditions of Corollary \ref{cor:collar_ext} are met for the family $\Gamma_{\alpha ,\gamma}^{'v}$. Therefore
$$
\lambda (R^{\prime}_{\alpha ,\gamma})\asymp \frac{1}{\mathrm{mod}\Gamma_{\ell (\alpha )}^{'v}}\asymp e^{-r(\ell (\eta ))}.
$$

Since $\lambda (R_{\alpha ,\gamma})\geq  \lambda (R_{\alpha ,\gamma}')$, we have
$\lambda (R_{\alpha ,\gamma}) \gtrsim  e^{-r(\ell (\eta ))}$.

For the upper bound  for  $\lambda (R_{\alpha ,\gamma})$ we need to show that $\mathrm{mod} \Gamma_{R_{\alpha ,\gamma}} \geq c_1e^{r(\ell (\eta ))}$ for some constant $c_1>0$ where $\Gamma_{R_{\alpha ,\gamma}}$ is the family of curves connecting the boundary components of $R_{\alpha ,\gamma}$. Let $\Gamma_{R_{\alpha ,\gamma}}^\perp$ be the geodesic arcs connecting two boundary components of $R_{\alpha ,\gamma}$ that start orthogonal to the boundary geodesic $\alpha$. The family $\Gamma_{R_{\alpha ,\gamma}}^\perp$ is in a one to one correspondence with the family $\Gamma_{\alpha ,\gamma}^v$ of the vertical segments above the interval $[-\frac{1}{2},\frac{1}{2}]$ connecting the graphs of  $f_{\ell (\alpha )}(x)$ and $g_{\ell (\alpha )}(x)$. Then we have
$$
\mathrm{mod}\Gamma_{R_{\alpha ,\gamma}}\geq\mathrm{mod}\Gamma_{R_{\alpha ,\gamma}}^\perp=\mathrm{mod}\Gamma_{{\alpha ,\gamma}}^v.
$$
Note that  there exists $c_0>0$ such that $\pi /2-\cos^{-1} t\leq c_0t$ for $0\leq t\leq 1$. Then for $-1/2\leq x\leq 1/2$ we have
$$
g_{\ell (\alpha )}(x)=\frac{1}{\ell (\alpha )}\cos^{-1}\frac{\cosh |x|\ell (\alpha )}{\cosh r(\ell (\eta ))}\geq\frac{\pi}{2\ell (\alpha )}-\frac{c_0}{\ell (\alpha )}\frac{\cosh |x|\ell (\alpha )}{\cosh r(\ell (\eta ))}
$$
and
$$
f_{\ell (\alpha )}(x)-g_{\ell (\alpha )}(x)\leq\frac{c_0}{\ell (\alpha )}\frac{\cosh |x|\ell (\alpha )}{\cosh r(\ell (\eta ))}\leq\frac{2c_0}{\ell (\alpha )}\frac{e^{|x|\ell (\alpha )}}{e^{r(\ell (\eta ))}}.
$$
By Lemma \ref{lem:vertical-modulus} we have
$$
\mathrm{mod}\Gamma^v_{\alpha ,\gamma}=\int_{-\frac{1}{2}}^{\frac{1}{2}}\frac{dx}{f_{\ell (\alpha )}(x)-g_{\ell (\alpha )}(x)}\geq \frac{e^{r(\ell (\eta ))}}{2c_0}\int_{-\frac{1}{2}}^{\frac{1}{2}}\ell (\alpha )e^{-|x|\ell (\alpha )}dx\geq c_1e^{r(\ell (\eta ))}
$$
for some $c_1>0$. This finishes the proof.
\end{proof}

\begin{remark}
A key point in Theorem \ref{thm:half-collar} is the fact that the asymptotics of $\lambda(R_{\alpha,\gamma})$ is sharp, up to a bounded multiplicative constant. This is a consequence of comparing the extremal length of \emph{all} curves connecting the graphs of $f_{\ell(\alpha)}(x)$ and $g_{\ell(\alpha)}(x)$ over $[-1/2,1/2]$ to the modulus of the \emph{vertical subfamily}, using the theory developed in Section \ref{sec:between_graphs}. Much easier (and not sharp) estimates  can be obtained by comparing to families of curves in rectangles. For instance, $\lambda(R_{\alpha,\gamma})$ can be bounded below by considering the family of curves connecting $f_{\ell(\alpha)}(x)=\pi/(2\ell(\alpha))$ and 
$$y=\max_{(-\frac{1}{2},\frac{1}{2})} g_{\ell(\alpha)}(x) = g_{\ell}(0) = \frac{1}{\ell (\alpha )}\cos^{-1}\frac{1}{\cosh r(\ell (\eta ))}.$$

The extremal length of the latter family is estimated by (if say $\ell(\gamma)>1$)
\begin{align}\label{asymp:weak}
f_\ell(x)-g_{\ell}(x)&\asymp\frac{\pi}{2\ell(\alpha)}-\frac{1}{\ell(\alpha)} \left[\frac{\pi}{2} - \frac{1}{\cosh r(\ell(\eta))}\right]
\asymp \frac{e^{-r(\ell(\eta))}}{\ell(\alpha)}\asymp \frac{e^{-\ell(\alpha)}}{\ell(\alpha)},
\end{align}
which is much weaker than (\ref{non-standard:main}) or (\ref{asymp:ns2}).
\end{remark}

\begin{remark}
The universal covering $\tilde{R}$ in the logarithmic coordinates agrees with the lift of the nonstandard collar to the strip model of the hyperbolic plane (see  \cite[Example 7.9]{BeardonMinda}). In fact, the expression for $g_{\ell}(x)$ is an explicit formula for a geodesic in the strip model up to a linear map.
\end{remark}

\begin{corollary}\label{cor1}
\label{cor:asymp}
 Let $R_{\alpha ,\gamma }$ be the  nonstandard half-collar  as in Figure \ref{fig:comparing-collars}.  Then, for $\ell (\alpha )\to \infty$, we have 
\begin{align}\label{asymp:ns2}
\lambda 
(R_{\alpha ,\gamma}) \asymp 
\begin{cases}
 \ell (\gamma )e^{-\ell (\alpha )/2},  & \mbox{ if } 0<\ell (\gamma )<1,\\
 e^{-\ell (\alpha )/2}, &\mbox{ if }  \ell (\gamma )\geq 1.
\end{cases}
\end{align}


For $\ell (\alpha )\leq 1$ we have
\begin{align}\label{asymp:small}
\lambda (R_{\alpha ,\gamma })\gtrsim \frac{1}{\ell (\alpha )}.
\end{align}

 In particular, $\ell (\alpha )$ stays bounded between two positive constants if and only if $\lambda(R_{\alpha ,\gamma})$ stays bounded between two other positive constants. 
\end{corollary}

\begin{proof}
By (\ref{eq: 3 quantities}), for fixed $\ell(\gamma)$, we have that   $\ell(\alpha)\to \infty$  if and only if $\ell(\eta)\to0$. Therefore, combining (\ref{non-standard:main}) with (\ref{asymptotics:r-small}) and (\ref{eq: 3 quantities}) we obtain, respectively, 
\begin{align}
\lambda(R_{\alpha,\gamma})
 & \asymp \ell(\eta) \asymp \frac{\tanh(\ell(\gamma))}{e^{\ell(\alpha)/2}},
\end{align}
as $\ell(\alpha)\to\infty$, which implies (\ref{asymp:ns2}) by using standard estimates for $\tanh(\ell(\eta))$.

The estimate (\ref{asymp:small}) follows from  (\ref{asymp:weak}) since $\ell(\alpha)<1$. Finally, the last statement follows immediately from (\ref{asymp:ns2}) and (\ref{asymp:small}).
\end{proof}



\begin{corollary}
\label{cor:r(d)tol}
Let $P$ be a pairs of pants with boundary lengths $\ell (\alpha ),\ell (\alpha_1 ),\ell (\alpha_2 )$ with $\ell (\alpha )>0$ and $\ell (\alpha_1 )<2\coth^{-1}(\cosh 1)$. Then the nonstandard half-collar $R_{\alpha ,\gamma}$ about the geodesic $\alpha$ satisfies
$$
\lambda({R_{\alpha ,\gamma}})\asymp e^{-\ell (\alpha )/2},
$$
as $\ell(\alpha)\to\infty$.
\end{corollary}
\begin{proof}
The hexagon formula, cf. \cite{Beardon}, applied to a half of the pair of pants with boundary lengths $\ell (\alpha ),\ell (\alpha_1 ),\ell (\alpha_2 )$ gives
\begin{align*}
\cosh \ell (\gamma )&=\coth\frac{\ell (\alpha_1 )}{2}\coth\frac{\ell (\alpha )}{2} +\frac{\cosh( \ell (\alpha_2 )/2)}{\sinh (\ell (\alpha_1 )/2)\sinh (\ell (\alpha )/2)}\\
&\geq \coth\frac{\ell (\alpha_1 )}{2}\coth\frac{\ell (\alpha )}{2} \\
&\geq \cosh 1
\end{align*}
where $\gamma$ is the orthogeodesic between $\alpha$ and $\alpha_1$. 

Then we get that $\ell (\gamma ) \geq 1$ which extends the same conclusion to this case by Corollary \ref{cor1}. 
\end{proof}


\section{Gluing nonstandard half-collars with a twist}\label{Section:gluing}

Let $P$ be a geodesic pair of pants with  boundary geodesic curve $\alpha$ and an orthogeodesic $\gamma$ from $\alpha$ to another boundary geodesic or a puncture $\alpha_1$ of $P$. Let $P'$ be another geodesic pair of pants with  boundary geodesic  $\alpha'$ and an orthogeodesic $\gamma'$ from $\alpha'$ to another boundary geodesic or a puncture $\alpha_1'$ of $P$. If $\ell (\alpha )=\ell (\alpha')$ then we can glue $\alpha$ to $\alpha'$ by an isometry such that the relative position of the marked points is given by the twist parameter $t(\alpha )\in (-\frac{1}{2},\frac{1}{2}]$ (see section \ref{sec: Contents, notation, and convention}). 

Let $R_{\alpha ,\gamma ,\gamma'}^{t (\alpha )}$ be the union of the two nonstandard half-collars $R_{\alpha ,\gamma}$ and $R_{\alpha',\gamma'}$ around $\alpha \equiv\alpha'$ with the twist parameter $t(\alpha )$. 
We give an estimate of the extremal distance between the two boundary components  of the annulus $R_{\alpha ,\gamma ,\gamma'}^{t (\alpha )}$. When the twist is not zero the extremal distance between the two boundary components  of $R_{\alpha ,\gamma ,\gamma'}^{t (\alpha )}$ is larger than the sum of the extremal distances between the boundary components of the two nonstandard half-collars $R_{\alpha ,\gamma}$ and $R_{\alpha',\gamma'}$. 

\begin{theorem}
\label{thm:gluings}
Let $R_{\alpha ,\gamma ,\gamma'}^{t (\alpha )}$ be an annulus obtained by isometrically  gluing two nonstandard half-collars along their boundary geodesics $\alpha$ and $\alpha^{\prime}$  of equal length   with  twist $-\frac{1}{2}\leq t(\alpha )< \frac{1}{2}$.
Given $\ell_0\geq 2$, there exists $C>0$
 such that for all $\ell (\alpha )\geq \ell_0$ we have
$$
\lambda(R_{\alpha ,\gamma ,\gamma'}^{t (\alpha )})\geq \frac{C}{\max\{ e^{r(\ell (\eta ))-|t(\alpha)|\frac{\ell (\alpha )}{2}},e^{r(\ell (\eta'))-|t(\alpha )|\frac{\ell (\alpha )}{2}}\}},
$$
where  $\eta$ and $\eta'$ are  orthogonals  on the respective nonstandard half-collars as in section \ref{sec:nonstandard}. 
\end{theorem}

\begin{remark} \label{rem:expansion}  By equation (\ref{eq: 3 quantities}) we have $r(\ell (\eta )) \asymp \frac{\ell (\alpha )}{2}-\ln^{-} \ell (\gamma )$, where $\ln^{-}x=\min\{\ln x ,0\}$. 
Then for $\min\{\ell (\gamma ),\ell (\gamma^{\prime})\} \leq 1$ the above estimate is equivalent to 
\begin{equation} \label{eq: lower bound} 
\lambda(R_{\alpha ,\gamma ,\gamma'}^{t (\alpha )})\geq C \min\{\ell (\gamma ),\ell (\gamma^{\prime})\}e^{-\left(\frac{1}{2}-\frac{|t(\alpha )|}{2}\right)\ell (\alpha )},
\end{equation}
and  for $\min\{\ell (\gamma ),\ell (\gamma^{\prime})\}\geq 1$,
\begin{equation} \label{eq: lower bound without min} 
\lambda(R_{\alpha ,\gamma ,\gamma'}^{t (\alpha )})\geq C e^{-\left(\frac{1}{2}-\frac{|t(\alpha )|}{2}\right)\ell (\alpha )}.
\end{equation}
\end{remark}

If the two nonstandard half-collars come from  tight pairs of pants we have $r(\ell (\eta ))=r(\ell (\eta'))=\ell (\alpha )/2$ and $\ell (\gamma )=\ell (\gamma^{\prime})=\infty$,  and  we immediately obtain

\begin{corollary}
\label{cor:twist1/2}
Let $R_{\alpha ,\gamma ,\gamma'}^{t (\alpha )}$ be the annular region obtained by gluing two nonstandard half-collars $R_{\alpha ,\gamma}$ and $R_{\alpha ',\gamma'}$ with twist $t(\alpha )$ and $\ell (\gamma )=\ell (\gamma' )=\infty$. If $\ell (\alpha ) \geq \ell_0\geq 2$ and $-\frac{1}{2}\leq t(\alpha )< \frac{1}{2}$ then, for the constant $C$ from Theorem \ref{thm:gluings},
$$
\lambda (R_{\alpha ,\gamma ,\gamma'}^{t (\alpha )} )\geq C e^{-(\frac{1}{2}-\frac{|t(\alpha )|}{2})\ell (\alpha )} .
$$
\end{corollary}

It is interesting to note that if two standard half-collars are glued together then the extremal distance between the boundary components does not increase with twist. In fact, the extremal distance between the boundary components of the glued standard collars is equal to twice the extremal distance between the boundary components of the half standard collar due to the fact that the glued collar is symmetric about the closed core geodesic.

\begin{proof}[Proof  of Theorem \ref{thm:gluings}]

The two nonstandard half-collars  have  corresponding parameters $\ell (\alpha )$, $\ell (\eta )$, $\ell (\gamma )$, 
and $\ell (\alpha ')=\ell (\alpha )$,  $\ell (\eta')$, and $\ell (\gamma')$, respectively. Note that the orthogeodesics $\gamma$ and $\gamma'$ do not meet at the geodesic $\alpha$ unless $t(\alpha )=0$. In fact, the distance between $\gamma$ and $\gamma'$ along the geodesic $\alpha$ is $|t(\alpha )|\ell (\alpha )$ (see Figure \ref{fig:glued-ns-collars-topview}). Denote by $\Gamma_{\beta,\beta'}$ the family of curves connecting  $\beta$ to  $\beta'$ in $R_{\alpha ,\gamma ,\gamma'}^{t (\alpha )}$, where $\beta$ and  $\beta'$  are the piecewise geodesic loops on the respective  half-collars. Note that $\lambda(R_{\alpha ,\gamma ,\gamma'}^{t (\alpha )})=1/\mathrm{mod}\Gamma_{\beta ,\beta'}$.



We extend $\gamma$ until it hits $\beta'$  and continue to call the extension $\gamma$. Cut $R_{\alpha ,\gamma ,\gamma'}^{t (\alpha )}$ along $\gamma$ and choose a lift $Q_{t(\alpha )}$ of the simply connected region $R_{\alpha ,\gamma ,\gamma'}^{t (\alpha )}-\gamma$ to the upper half-plane  so that  $\alpha$ is lifted to the segment $[e^{-\ell (\alpha )/2}i,e^{\ell (\alpha )/2}i]$ of the $y$-axis and the  two lifts of $\gamma$ are geodesic arcs orthogonal to the $y$-axis that pass through the points $e^{-\ell (\alpha )/2}i$ and $e^{\ell (\alpha )/2}i$ as in Figure \ref{fig:lift}.

The lift of the  geodesic loop $\beta$ of the left half-collar lies on a geodesic with endpoints $-e^{r(\ell (\eta ))}$ and $-e^{-r(\ell (\eta ))}$. Let $\tilde{\gamma}$ be a  lift of  $\gamma$ and denote by  $\tilde{b}$ the lift of  $b$,   where $\tilde{\gamma}$
  meets  the geodesic with endpoints $-e^{r(\ell (\eta ))}$ and $-e^{-r(\ell (\eta ))}$. The lift of the geodesic loop $\beta'$  of the right nonstandard half-collar that intersects the lifts of $\gamma$ lies on two geodesics with endpoints $e^{t(\alpha )\ell (\alpha)-\ell(\alpha )-r(\ell (\eta'))},e^{t(\alpha )\ell (\alpha)-\ell (\alpha )+r(\ell (\eta'))}$ and $e^{t(\alpha )\ell (\alpha )-r(\ell (\eta'))},e^{t(\alpha )\ell (\alpha )+r(\ell (\eta'))}$. The lift of $b'$ is the intersection of the two geodesics above (see Figure 
\ref{fig:lift}).

\begin{figure}[h]
\includegraphics[scale=1]{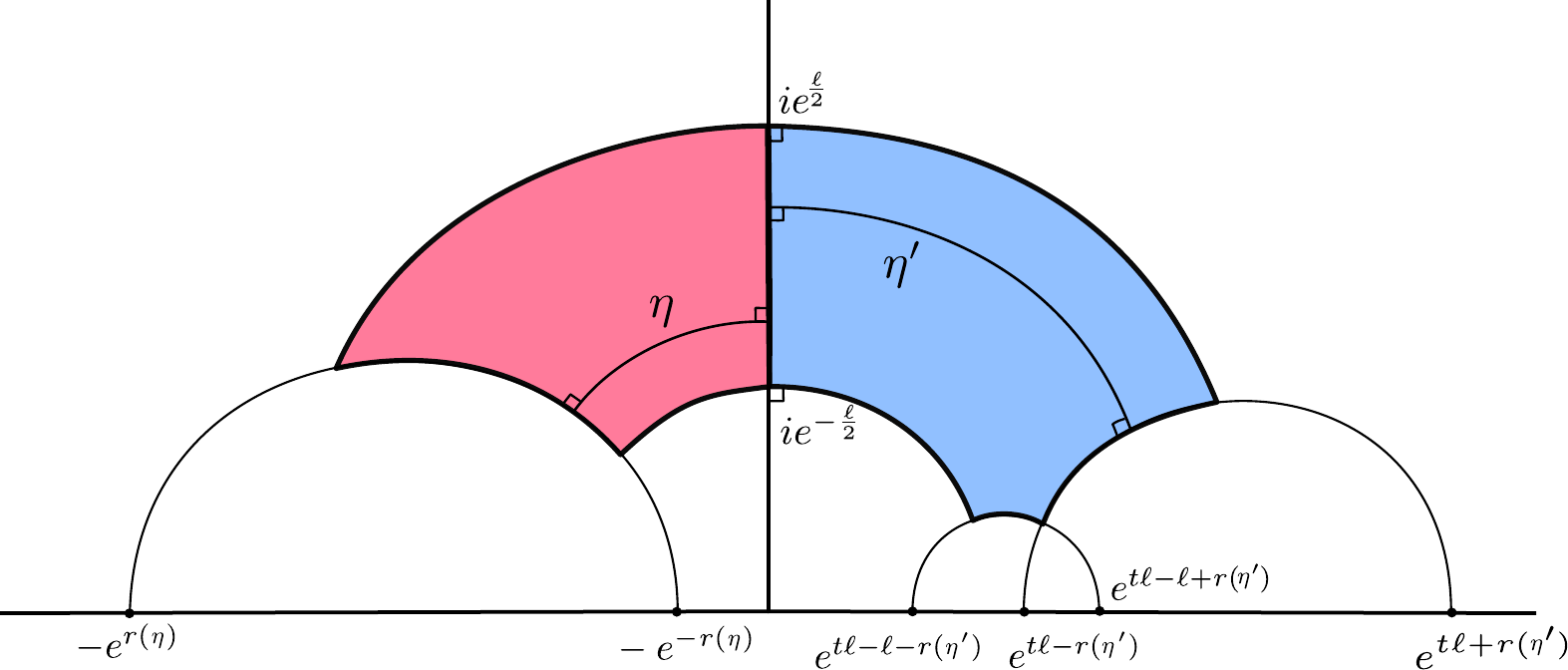}
\caption{The lift of the glued nonstandard collars. Here $\ell=\ell(\alpha )$.}\label{fig:lift}
\end{figure}


Let $\tau_{\ell (\alpha )}(w)=e^{\ell (\alpha )}w$ be the hyperbolic translation corresponding to the geodesic $\alpha$. Then $\tilde{A}_{t(\alpha )}=\cup_{k\in\mathbb{Z}}\tau_{\ell (\alpha )}^k(Q_{t(\alpha )})$ is a universal covering of $R_{\alpha ,\gamma ,\gamma'}^{t (\alpha )} $. We apply the map $z=\frac{1}{\ell (\alpha )}\log w$ to $\tilde{A}_{t(\alpha )}$. The image of the boundary component of $R_{\alpha ,\gamma ,\gamma'}^{t (\alpha )} $ that covers $\beta$ is the graph over $\mathbb{R}$ of a function $f_{t(\alpha )}(x)$ which is invariant under the translations $x\mapsto x+k$, for all $k\in\mathbb{Z}$, since $\tau_{\ell (\alpha )}(z)$ is conjugated by $z=\frac{1}{\ell (\alpha )}\log w$ to $x\mapsto x+1$. To find $f_{t(\alpha )}(x)$ on the interval $[-1/2,1/2]$, we find the $y$-coordinate of the image of the geodesic with the endpoints $-e^{-r(\ell (\eta ))}$ and $-e^{r(\ell (\eta ))}$. From the equation of the circle containing the geodesic
$$
|w+\cosh r(\ell (\eta ))|=\sinh r(\ell (\eta ))
$$
and by $w=e^{\ell (\alpha )z}$ with $z=x+iy$
we obtain 
\begin{equation*}
\begin{split}
f_{t(\alpha )}(x)=\frac{1}{\ell (\alpha )}\cos^{-1}\left(-\frac{\cosh \ell (\alpha ) x}{\cosh r(\ell (\eta ))}\right)=
\frac{1}{\ell (\alpha )} \left(\pi -\cos^{-1}\frac{\cosh \ell (\alpha ) x}{\cosh r(\ell (\eta ))}\right)\\
\geq\frac{1}{\ell (\alpha )}\left( \frac{\pi}{2}+\frac{e^{\ell (\alpha ) |x|}}{2e^{r(\ell (\eta ))}}\right)=:h_1(x)
\end{split}
\end{equation*}
for $-1/2\leq x\leq 1/2$, where we used the inequality $\frac{\pi}{2}-\cos^{-1} s\geq s$ for $0\leq s\leq 1$. Note that $\frac{\cosh \ell (\alpha ) x}{\cosh r(\ell (\eta ))}\leq 1$ for $x\in [-\frac{1}{2},\frac{1}{2}]$ by  (\ref{eq:eta-alpha}). The functions  $f_{t(\alpha )}$ and $h_1(x)$  extend to  $\Bbb R$ by invariance under 
$x \mapsto x+1$.

Similarly, the image in the $z$-plane of the boundary component of $\tilde{A}_{t(\alpha )}$ that covers the geodesic loop $\beta'$ is the graph of the function $g_{t(\alpha )}(x)$ obtained by taking the $y$-coordinate of the image of the geodesic with endpoints $e^{-r(\ell (\eta'))+t(\alpha )\ell (\alpha )}$ and $e^{r(\ell (\eta'))+t(\alpha )\ell (\alpha )}$ over the interval $[-1/2+t(\alpha ),1/2+t(\alpha )]$ and extending it by $g_{t(\alpha )}(x+k)=g_{t(\alpha )}(x)$ for all $k\in\mathbb{Z}$. We obtain
$$
g_{t(\alpha )}(x)=\frac{1}{\ell (\alpha )}\cos^{-1}\frac{\cosh \ell (\alpha ) (x-t(\alpha ))}{\cosh r(\ell(\eta'))}\leq \frac{1}{\ell (\alpha )}\left( \frac{\pi}{2}-\frac{e^{\ell (\alpha )|x-t(\alpha )|}}{2e^{r(\ell (\eta'))}}\right)=:h_2(x)
$$
for $-1/2+t(\alpha )\leq x\leq 1/2+t(\alpha )$  and extend $g_{t(\alpha )}(x)$ and $h_2(x)$ to  $\Bbb R$ by invariance under 
$x \mapsto x+1$. The region $Q_{t(\alpha )}$ is mapped to the region above the segment $[-1/2,1/2]$ and  between the graphs of $f_{t(\alpha )}(x)$ and $g_{t(\alpha )}(x)$ for $x\in [-1/2,1/2]$.  The graphs of $h_1(x)$ and $h_2(x)$ bound a subregion between the graphs of $f_{t(\alpha )}(x)$ and $g_{t(\alpha )}(x)$ that is invariant under the translation $x\mapsto x+1$ and  projects to a subring $R'$ of  $R_{\alpha ,\gamma ,\gamma'}^{t (\alpha )} $, see Figure \ref{fig:log coordinates}. Let $\Gamma_{R'}$ be the family of curves that connects the two boundary components of $R'$. Since each $\gamma\in\Gamma_{\beta ,\beta'}$ contains a subarc in $\Gamma_{R'}$ it follows that
$$
\mathrm{mod}\Gamma_{\beta ,\beta'}\leq\mathrm{mod}\Gamma_{R'}.
$$
Therefore we only need to estimate $\mathrm{mod}\Gamma_{R'}$ from  above.

\begin{figure}[h]
\includegraphics[scale=2]{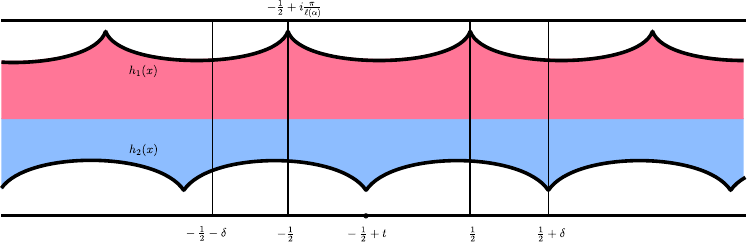}
\caption{The lift in log coordinates of the glued nonstandard collars.}\label{fig:log coordinates}
\end{figure}


Let $\Gamma_{R'}^v$ be the family of vertical lines connecting the graphs of $h_1(x)$ and $h_2(x)$ above $-\frac{1}{2}\leq x\leq\frac{1}{2}$. 
We show that the conditions of Corollary \ref{cor:collar_ext} are satisfied.  
Let $\delta =\delta (\ell (\alpha))=1/\ell (\alpha )$ and  define
$$
m_{\delta}(x)=\inf_{|\delta_1|,|\delta_2|\leq\delta}\left[h_1(x+\delta_1)-h_2(x+\delta_2)\right]
$$
for $x\in [-1/2 ,1/2 ]$. 

We estimate $m_{\delta}(x)$ from below similar to the proof of Theorem \ref{thm:half-collar}. 
Define $$k_1(x)=\frac{e^{\ell (\alpha ) |x|}}{2\ell (\alpha )e^{r(\ell (\eta ))}}\text{\ for\ } x\in [-\frac{1}{2},\frac{1}{2}] \text{\quad and \quad}
k_2(x)=\frac{e^{\ell (\alpha ) |x-t(\alpha )|}}{2\ell (\alpha ) e^{r(\ell (\eta'))}}\text{\ for\ } x\in [-\frac{1}{2}+t(\alpha ),\frac{1}{2}+t(\alpha )]$$
and extend both functions to $\Bbb R$ with period one.
Note that
$$h_1(x)-h_2(x)=k_1(x)+k_2(x).$$ 

Assume that  $|\delta_1|\leq\delta$ and $x,x+\delta_1\in [-1/2,1/2]$. Then we have
$$
k_1(x+\delta_1)=\frac{e^{\ell (\alpha ) |x+\delta_1|}}{2\ell (\alpha ) e^{r(\ell (\eta ))}}\geq e^{-\ell (\alpha )|\delta_1|}\frac{e^{\ell (\alpha ) |x|}}{2\ell (\alpha ) e^{r(\ell (\eta ))}}\geq e^{-1}k_1(x)$$
because $|\delta_1|\ell (\alpha )\leq\delta (\ell (\alpha )) \ell (\alpha )=1$.

Assume that $\ell (\alpha ) >2$, $x\in [-\frac{1}{2},\frac{1}{2}]$ and $x+\delta_1<-\frac{1}{2}$.  We refer the reader to 
Figure \ref{fig: periodicity}  in order to easily follow the arguments in this case. 

\begin{figure}[h] 
\includegraphics[scale=1.5]{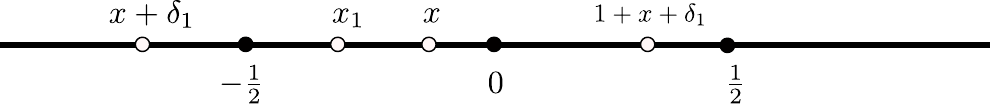}
\caption{Periodicity of the definition of $k_2(x)$.} \label{fig: periodicity}
\end{figure}

%

Then $|\delta_1|<\frac{1}{2}$ and $-\frac{1}{2}\leq x<0$.
Using the facts that $k_1(x)=k_1(-x)$ and $k_1(x+1)=k_1(x)$ we obtain
$$
k_1(x+\delta_1)=k_1(-(1+x+\delta_1)).
$$
Let $x_1=-(1+x+\delta_1)$. By $x+\delta_1<-\frac{1}{2}$ we immediately get $x_1+1/2>0$ which in turn gives
$$
|x-x_1|\leq |x+\frac{1}{2}|+|x_1+\frac{1}{2}|=x+\frac{1}{2}-\frac{1}{2}-x-\delta_1=-\delta_1\leq \delta.
$$
Then, by $x,x_1\in [-\frac{1}{2},\frac{1}{2}]$ and the above, we get
$$
k_1(x+\delta_1)=k_1(x_1)\geq e^{-1}k_1(x).
$$
For $x\in [-1/2,1/2]$ and $x+\delta_1>1/2$ we similarly obtain $k_1(x+\delta_1)\geq e^{-1}k_1(x)$. 

Finally, since $k_2(x)=k_1(x-t)$ we have that
$$
k_2(x+\delta_2)\geq e^{-1}k_2(x)
$$
for all $x$ and $|\delta_2|\leq\delta$.

Thus 
for $|\delta_1|,|\delta_2|\leq\delta$, and all $x$ we have
$$
h_1(x+\delta_1)-h_2(x+\delta_2)\geq e^{-1}[h_1(x)-h_2(x)]
$$
which implies 
$
\inf_{\ell (\alpha )\geq \ell_0} c_{\delta (\ell (\alpha ) ), \ell (\alpha )}\geq e^{-1}>0.
$ Thus condition (2) of Corollary \ref{cor:collar_ext} is satisfied.

To  verify  condition (1) in Corollary \ref{cor:collar_ext}  we find a lower bound on $\mathrm{mod}\Gamma_{R'}^v$. 
Assume first that $0\leq t\leq \frac{1}{2}$. Then $-\frac{1}{2}+t\leq 0$ and thus 
$[-\frac{1}{2}+t,\frac{1}{2}]\supset [0,\frac{1}{4}]$. For $x\in [\frac{1}{8},\frac{1}{4}]$, we have that
$$
k_1(x)\leq\frac{e^{\frac{\ell (\alpha )}{4}}}{2\ell (\alpha )e^{r(\ell (\eta ))}}\quad\mathrm{and}\quad k_2(x)\leq \frac{e^{\frac{3\ell (\alpha )}{8}}}{2\ell (\alpha ) e^{r(\ell (\eta'))}}.
$$
By setting $r^{*}:=\min\{ r(\ell (\eta )),r(\ell (\eta'))\}\geq\frac{\ell (\alpha )}{2}$ we obtain, for $x\in [\frac{1}{8},\frac{1}{4}]$,
$$
k_1(x)+k_2(x)\leq \frac{e^{\frac{3\ell (\alpha )}{8}}}{\ell (\alpha )e^{r^{*}}}
$$
which gives
$$
\mathrm{mod}\Gamma_{R'}^{v}=\int_{-\frac{1}{2}}^{\frac{1}{2}}\frac{dx}{k_1(x)+k_2(x)}\geq \int_{\frac{1}{8}}^{\frac{1}{4}}\frac{dx}{k_1(x)+k_2(x)}\geq \frac{1}{8}\ell (\alpha ) e^{\frac{3\ell (\alpha )}{8}}.
$$
Thus $\inf_{\ell (\alpha )\geq \ell_0}  \delta (\ell (\alpha ))^2 \mathrm{mod} \Gamma_{R'}^v\geq \inf_{\ell (\alpha )\geq \ell_0} \frac{e^{\frac{3\ell (\alpha )}{8}}}{8\ell (\alpha )} >0$ 
and the condition (1) in Corollary \ref{cor:collar_ext} is satisfied when $0\leq t(\alpha ) \leq \frac{1}{2}$. 

We prove the condition (1) in Corollary \ref{cor:collar_ext} under the assumption that $\frac{1}{2}\leq t(\alpha )\leq 1$ (which is equivalent to $-\frac{1}{2}\leq t(\alpha )\leq 0$). Note that $[-\frac{1}{2},-\frac{1}{2}+t(\alpha )]\supset [-\frac{1}{4},-\frac{1}{8}]$. For $x\in [-\frac{1}{4},-\frac{1}{8}]$ we have that $k_1(x)\leq \frac{e^{\frac{\ell (\alpha )}{4}}}{2\ell (\alpha )e^{r(\ell (\eta ))}}$ as before. On the other hand we have $k_2(x)=\frac{e^{\ell (\alpha ) |1+x-t(\alpha )|}}{2\ell (\alpha ) e^{r(\ell (\eta'))}}$ for $x\in [-\frac{1}{4},-\frac{1}{8}]\subset [-\frac{1}{2},-\frac{1}{2}+t(\alpha )]$. Since $-\frac{1}{4}\leq 1+x-t(\alpha )\leq \frac{3}{8}$ we obtain $k_2(x)\leq \frac{e^{\frac{3\ell (\alpha )}{8}}}{2\ell (\alpha ) e^{r(\ell (\eta'))}}.$ Then  condition (1) in Corollary \ref{cor:collar_ext} is satisfied as in the above paragraph. 

The lower estimate for $\mathrm{mod} \Gamma_{R'}^{v}$ obtained above is crude. The only purpose was to show that the condition (1) in Corollary \ref{cor:collar_ext} is satisfied. We now proceed to obtain a finer upper estimate in order to prove the desired inequality in the theorem.

Let $\Gamma_{R'}^{v}$ be the vertical family above $[-1/2,1/2]$ connecting the graphs of $h_1$ and $h_2$. 
We proceed to compute  $\mathrm{mod}\Gamma_{R'}^{v}=\int_{-1/2}^{1/2}\frac{dx}{h_1(x)-h_2(x)}.$ In order to facilitate the integration, we will assume that $0\leq t(\alpha )\leq 1$ with the understanding that the twist $t(\alpha )\in [1/2,1]$ gives the same surface as the twist $t(\alpha )-1\in [-1/2,0]$.
The integration over $x\in [-1/2,1/2]$ is divided into four intervals. For $-1/2\leq x\leq -1/2+t(\alpha )/2$ we have the inequality $h_1(x)-h_2(x)\geq \frac{1}{2\ell (\alpha )}\frac{e^{-\ell (\alpha ) x}}{e^{r(\ell (\eta ))}}$ which gives
\begin{equation}
\label{eq:first-int}
\int_{-\frac{1}{2}}^{-\frac{1}{2}+\frac{t(\alpha )}{2}}\frac{dx}{h_1(x)-h_2(x)}\leq 2e^{r(\ell (\eta ))}\int_{-\frac{1}{2}}^{-\frac{1}{2}+\frac{t(\alpha )}{2}}\ell (\alpha ) e^{\ell (\alpha )x}dx\leq 2e^{r(\ell (\eta ))-\frac{\ell (\alpha )}{2}+t(\alpha )\frac{\ell (\alpha )}{2}}.
\end{equation}

For $-1/2+t(\alpha )/2\leq x\leq -1/2+t(\alpha )$ we have the inequality $h_1(x)-h_2(x)\geq \frac{1}{2\ell (\alpha )}\frac{e^{\ell (\alpha ) (1+x-t(\alpha ))}}{e^{r(\ell (\eta'))}}$ which gives
\begin{equation}
\label{eq:second-int}
\begin{split}
\int_{-\frac{1}{2}+\frac{t(\alpha )}{2}}^{-\frac{1}{2}+t(\alpha )}\frac{dx}{h_1(x)-h_2(x)}\leq 2e^{r(\ell (\eta'))}e^{t(\alpha )\ell (\alpha ) }e^{-\ell (\alpha )}\int_{-\frac{1}{2}}^{-\frac{1}{2}+\frac{t(\alpha )}{2}}\ell (\alpha ) e^{-\ell (\alpha ) x}dx\\ \leq 2e^{r(\ell (\eta'))-\frac{\ell (\alpha )}{2}+t(\alpha )\frac{\ell (\alpha )}{2}}.
\end{split}
\end{equation}

For $-1/2+t(\alpha )\leq x\leq t(\alpha )/2$ we have the inequality $h_1(x)-h_2(x)\geq \frac{1}{2\ell (\alpha )}\frac{e^{\ell (\alpha ) (t(\alpha )-x)}}{e^{r(\ell (\eta'))}}$ which gives
\begin{equation}
\label{eq:third-int}
\int_{-\frac{1}{2}+t(\alpha )}^{\frac{t(\alpha )}{2}}\frac{dx}{h_1(x)-h_2(x)}\leq 2e^{r(\ell (\eta'))}e^{-t(\alpha )\ell (\alpha )}\int_{-\frac{1}{2}+t(\alpha )}^{\frac{t(\alpha )}{2}}\ell (\alpha ) e^{\ell (\alpha ) x}dx\leq 2e^{r(\ell (\eta'))-t(\alpha )\frac{\ell (\alpha )}{2}}.
\end{equation}

For $t(\alpha )/2\leq x\leq 1/2$ we have the inequality $h_1(x)-h_2(x)\geq \frac{1}{2\ell (\alpha )}\frac{e^{\ell (\alpha )x}}{e^{r(\ell (\eta ))}}$ which gives
\begin{equation}
\label{eq:fourth-int}
\int_{\frac{t(\alpha )}{2}}^{\frac{1}{2}}\frac{dx}{h_1(x)-h_2(x)}\leq 2e^{r(\ell (\eta ))}\int_{\frac{t(\alpha )}{2}}^{\frac{1}{2}}\ell (\alpha ) e^{-\ell (\alpha ) x}dx\leq 2e^{r(\ell (\eta ))-t(\alpha )\frac{\ell (\alpha )}{2}}.
\end{equation}

Putting together (\ref{eq:first-int}), (\ref{eq:second-int}), (\ref{eq:third-int}), and (\ref{eq:fourth-int}), we obtain 
$$
\mathrm{mod}\Gamma_{R'}^{v}\leq c\max\{ e^{r(\ell (\eta ))-t(\alpha )\frac{\ell (\alpha )}{2}}, e^{r(\ell (\eta ))-\frac{\ell (\alpha )}{2}+t(\alpha )\frac{\ell (\alpha )}{2}},e^{r(\ell (\eta'))-t(\alpha)\frac{\ell (\alpha )}{2}}, e^{r(\ell (\eta'))-\frac{\ell (\alpha )}{2}+t(\alpha )\frac{\ell (\alpha )}{2}}\}
$$
for a fixed $c>0$ and all $t(\alpha )\in [0,1]$. 

For $t(\alpha )\in [0,1/2]$ we have
$$
\mathrm{mod}\Gamma_{R'}^{v}\leq c\max\{ e^{r(\ell (\eta ))-t(\alpha )\frac{\ell (\alpha )}{2}}, e^{r(\ell (\eta'))-t(\alpha )\frac{\ell (\alpha )}{2}}\}
$$
and for $t(\alpha )\in [1/2,1]$ we have
$$
\mathrm{mod}\Gamma_{R'}^{v}\leq c\max\{ e^{r(\ell (\eta ))-\frac{\ell (\alpha )}{2}(1-t(\alpha ))}, e^{r(\ell (\eta'))-\frac{\ell (\alpha )}{2}(1-t(\alpha ))}\}
$$
and hence the theorem follows.
\end{proof}

Suppose $\alpha$ and $\beta$  are disjoint simple closed geodesics on a hyperbolic surface. By the collar lemma the standard collars about $\alpha$ and $\beta$ are disjoint.  On the other hand, it is possible that nonstandard collars overlap. Indeed, even  nonstandard half-collars can overlap as can be seen on a pair of pants. In fact, a  pair of pants can have at most two disjoint nonstandard half-collars. Given any two components of a pair of pants one can construct  nonstandard half-collars about each one by using the unique simple orthogeodesic from the boundary component to the third component. These half-collars are disjoint.  This is the key point in the following topological lemma.

\begin{lemma}  \label{lem:disjoint half-collars}
Let $Y$ be a geodesic subsurface of finite area which is not a pair of pants and has  non-empty  boundary.   Then for any choice of pants decomposition of  $Y$ there are choices of nonstandard half-collars about each  boundary component  that  are  pairwise disjoint.
\end{lemma}

\begin{proof} Consider a   pair of pants $P$ in the given decomposition  having at least one geodesic boundary in common with $Y$. Since there is at least one boundary component  of $P$, say $\beta$,   interior to $Y$ (otherwise $Y$ would be a pair of pants)  one can run a simple orthogeodesic from the  boundary component of $Y$ to $\beta$ thus constructing a nonstandard half-collar. As was observed above,  if $P$ contains  two nonstandard half-collars they are disjoint in $P$,  hence in $Y$. 
\end{proof}


\section{Modular tests for parabolicity}
\label{sec:type_problem}

Recall, from Subsection \ref{subsection:Ahlfors-criterion} of the introduction that a characterization of parabolic type Riemann surfaces can be given in terms of  extremal distance. Let $\{X_n\}$ be an exhaustion of $X$ by a family of relatively compact regions with piecewise analytic boundary such that $\overline{X}_n\subset X_{n+1}$. Denote by $\beta_n$ the boundary of $X_n$. 
Let $\lambda_{X_n-X_1}(\beta_1,\beta_n)$ be the extremal distance between $\beta_1$ and $\beta_n$ in $X_n-{X}_1$. We will use the following criterion for parabolicity, cf. \cite[page 229]{Ahlfors-Sario}.

\begin{proposition}[Modular test]\label{prop: parabolicity characterizaton}
The Riemann surface $X$ is parabolic if and only if $$\lambda_{X_n-X_1}(\beta_1,\beta_n)\to\infty\quad \mathrm{as }\quad n\to\infty .$$  
\end{proposition}

To simplify the application of the Modular test, given an exhaustion $\{ X_n\}$ we choose an open set $B_n\subset X_{n+1}- X_{n-1}$ that contains the boundary $\beta_n$ of $X_n$. The boundary of $B_n$ is divided into two sets $a_n$ and $b_n$ such that $a_n\subset X_n$ and $b_n\subset X_{n+1}- X_n$. The interiors of $B_{n-1}$ and $B_n$ are disjoint by the assumption. Denote by $\lambda_n$ the extremal distance between $a_n$ and $b_n$ in $B_n$. Then by the serial rule for  extremal length (see \cite[page 222]{Ahlfors-Sario})
$$
\lambda_{X_n-X_1}(\beta_1,\beta_n)\geq\sum_{k=1}^n\lambda_k
$$
and we obtain the following sufficient condition for parabolicity
\begin{proposition}[see \cite{Ahlfors-Sario}]
\label{prop:parabolic}
 If
$$
\sum_{n=1}^{\infty}\lambda_n=\infty 
$$
then the Riemann surface $X$ is of parabolic type.
\end{proposition}

Assume that the  Riemann surface $X$ has punctures. Since punctures are points at infinity, any such exhaustion would have boundary components bounding  punctured discs  which would add  terms corresponding to the punctures. We remove this difficulty by showing that an exhaustion that contains full neighborhoods of punctures can be used to replace the compact exhaustion in Propositions 
 \ref{prop: parabolicity characterizaton} and \ref{prop:parabolic}. To this end,  since our applications are hyperbolic geometric in nature, we will work in the hyperbolic category. 

A {\it geodesic subsurface}  in $X$  is a subsurface with geodesic boundary. We are interested in finite area geodesic subsurfaces, namely, ones that have  finitely many cusps and finitely many closed geodesics on the boundary. The next proposition   allows us  to  relax the criteria from a compact exhaustion to an exhaustion by finite area geodesic subsurfaces. 

\begin{proposition} \label{prop:finite area parabolicity char.}
Let $X$ be an infinite type hyperbolic surface with 
$\{X_n\}$ an exhaustion by finite area geodesic subsurfaces. Then  $X$ is parabolic if and only if $$\lambda_{X_n-X_1}(\partial X_1,\partial X_n)\to\infty\quad \mathrm{as }\quad n\to\infty .$$  
\end{proposition}

A puncture in a hyperbolic surface has  a {\it standard cusp neighborhood}  of hyperbolic area one which we denote by 
$\mathcal{C}_0$. The set of paths in  $\mathcal{C}_0$ from $\partial \mathcal{C}_0$ that go out the cusp end have infinite extremal length. This leads to the fact that the  standard cusp contains a decreasing sequence of cusp neigborhoods, denoted $\{\mathcal{C}_R\}$,  so that the extremal distance $\lambda_{\mathcal{C}_R\setminus \mathcal{C}_0}(\partial \mathcal{C}_0,\partial \mathcal{C}_R)=R$. We  start with a lemma that builds on this fact. It is a straightforward application of  basic extremal length properties, we leave the proof to the reader. 

\begin{lemma} \label{lem: cusp neighborhoods}
 Suppose $Y$ is a geodesic subsurface with $p \geq 1$
punctures, and $K \subset Y$  a connected compact subset disjoint from the standard cusp neighborhoods of each puncture.
Given $R >0$,  denote by  $Y^{\prime}$  
the subsurface $Y$ with   deleted  cusp neighborhoods about each puncture where each of these neighborhoods is  isomorphic to  $\mathcal{C}_R$. Let $\Gamma'$ be the curve family in $Y'\setminus K$ which connects $\partial K$ to the boundary curves of the deleted neighborhoods of the punctures.
Then  $$\lambda (\Gamma') \geq \frac{R}{p}.$$
\end{lemma}

\begin{proof}
Since $K$ is disjoint from the standard cusp neighborhood of each puncture of $Y$, the family of curves $\Gamma'$ overflows the family of curves $\Gamma'_R$ that connect the boundary of the union $U_0$ of the standard cusp neighborhoods of the punctures to the boundary of the union $U_R$ of the cusp neighborhoods isomorphic to $\mathcal{C}_R$ inside $U_0\setminus U_R$. Then we have  (see Lemma \ref{lemma:mod-properties}, Property 1.)
$$
\mathrm{mod}\Gamma'\leq\mathrm{mod}\Gamma'_R.
$$
Note that $\Gamma_R'$ is a disjoint union of $p$ families of curves connecting the two boundaries of the two neighborhoods of each puncture and the supports of the $p$ families  are disjoint. Since the modulus of each family is $1/R$, the additivity property gives (see Lemma \ref{lem:disjoint})
$$
\mathrm{mod}\Gamma_R'=\frac{p}{R}.
$$
Then
$$
\lambda (\Gamma')=\frac{1}{\mathrm{mod}\Gamma'}\geq \frac{1}{\mathrm{mod}\Gamma'_R}=\frac{R}{p}.
$$
\end{proof}

We are now ready to prove Proposition \ref{prop:finite area parabolicity char.}.
\begin{proof}   Clearly the proposition holds if there are no punctures, and it is easy to check that it holds in the presence of funnels in $X$. For this reason, we henceforth assume that $X$ has no funnels and at least one puncture. 

\begin{figure}[t]
\begin{center}
\includegraphics[scale=1.5]{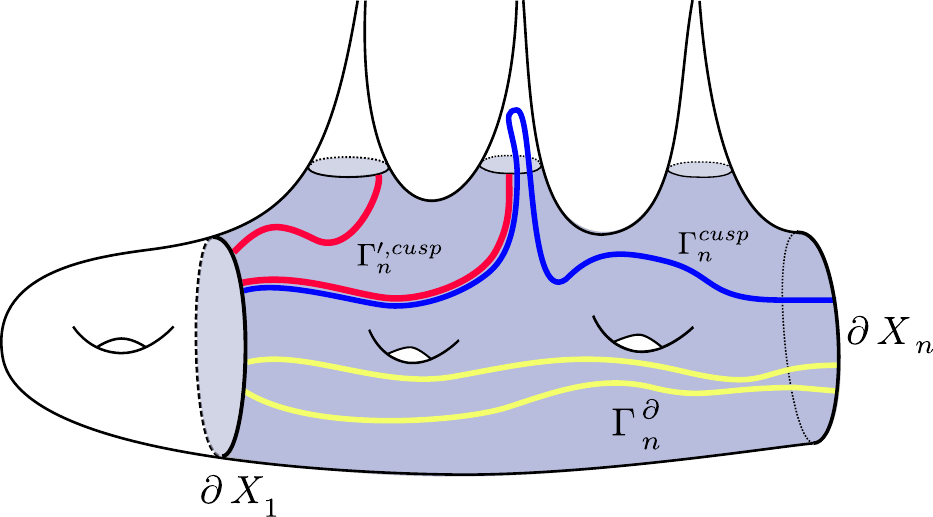}
\caption{Exhausting by finite area sub-surfaces. 
}\label{fig:exhaustion}
\end{center}
\end{figure}

Let $p(n)$ be the number of punctures in $X_n$. For each  puncture in $X_n$,  delete the cusp neighborhood isomorphic to $\mathcal{C}_{p(n)n}$ and denote the  excised domain by
$X_n^{\prime}$  (see Figure  \ref{fig:exhaustion}).  If there is no puncture in $X_n$, set $X_n^{\prime}=X_n$.  Clearly $\{X_n^{\prime}\}$  is a compact exhaustion of $X$.  As a matter of notational convenience  we assume $X_1=X_1^{\prime}$. We let  
\begin{align*}
\Gamma_n = \Gamma(\partial X_1,\partial X_n;X_n),
\Gamma_n' = \Gamma(\partial X_1,\partial X_n';X_n'),
\end{align*}
%
%
%
%
and would like to prove that $\lambda(\Gamma_n)$ and $\lambda(\Gamma_n')$ diverge simultaneously.

Let $\Gamma_n^{cusp}$ and $\Gamma_n^{\partial}$ denote the subfamilies of curves in $\Gamma_n$ that go through the cusp boundary of $X_n'$ and those that do not, respectively. Similarly we define $\Gamma_n^{\prime,cusp}$ and $\Gamma_n^{\prime,\partial}$. More precisely,

\begin{align*}
\Gamma^{cusp}&=\{\g\in \Gamma_n : \g\cap (\partial X_n' \setminus \partial X_n) \neq \emptyset\}, \quad  \Gamma_n^{\partial} = \Gamma_n\setminus \Gamma_{n}^{cusp}\\ 
\Gamma^{\prime,cusp}&=\{\g\in \Gamma_n' : \g\cap (\partial X_n' \setminus \partial X_n) \neq \emptyset\}, \quad \Gamma_n^{\prime,\partial}=\Gamma_n'\setminus \Gamma_{n}^{\prime,cusp}.
\end{align*}

By Lemma \ref{lem: cusp neighborhoods}, with $Y=X_n$, $K=X_1$, and $R=p(n)n$, we have 
\begin{align}\label{ineq:cusp-modulus}
\lambda(\Gamma_n^{\prime, cusp})\geq n.
\end{align}

Noting that $\Gamma_n^{\partial} = \Gamma_n^{\prime,\partial}$, and using monotonicity and subadditivity in Lemma \ref{lemma:mod-properties} we obtain  
\begin{align}\label{ineq:cusp-estimates}
\begin{split}
\frac{1}{\lambda (\Gamma_n^{\partial})} 
&\leq 
\frac{1}{\lambda (\Gamma_n)}
\leq
\frac{1}{\lambda (\Gamma_n^{\partial})}+\frac{1}{\lambda (\Gamma_n^{cusp})}, \\
\frac{1}{\lambda(\Gamma_n^{\prime, \partial})} 
&\leq 
\frac{1}{\lambda (\Gamma_n')}
\leq
\frac{1}{\lambda(\Gamma_n^{\prime, \partial})} + \frac{1}{\lambda (\Gamma_n^{\prime,cusp})}.
\end{split}
\end{align}
Since every curve in $\Gamma_n^{cusp}$ contains a subcurve in $\Gamma_n^{\prime,cusp}$, by the overflowing property of modulus we have $\lambda(\Gamma_n^{\prime,cusp})\leq \lambda(\Gamma_n^{cusp})$. By (\ref{ineq:cusp-modulus}) and (\ref{ineq:cusp-estimates}) we obtain
\begin{align*}
\left|\frac{1}{\lambda (\Gamma_n)} - \frac{1}{\lambda (\Gamma_n')} \right|\leq \frac{1}{\lambda(\Gamma_n^{\prime,cusp})}\leq\frac{1}{n}.
\end{align*}
Therefore, $\lambda (\Gamma_n) \rightarrow \infty$
if and only if  $\lambda (\Gamma_n^{\prime}) \rightarrow \infty$, thus completing the proof.
\end{proof}

We note  that the specific choice of the constant $R=p(n)n$ in the proof is not important.  

\begin{remark} 
 Proposition \ref{prop:finite area parabolicity char.} is part of a much more general phenomena in which a subsurface of the original hyperbolic surface has a compact exhaustion for which the extremal distance goes to infinity.

\end{remark}


\section{Applications to general infinite type  surfaces}
\label{sec:general-theorems}


Let $X$ be an infinite type Riemann surface  equipped with its conformal hyperbolic metric. Recall that a collar   is an annular neighborhood about  a  simple closed geodesic, and that a half-collar has the simple closed geodesic as a boundary component.

We fix an exhaustion $\{ X_n\}_n$ of $X$ by finite area geodesic subsurfaces  (that is, hyperbolic surfaces of finite area  with finitely many cusps and finitely many closed geodesics on their boundaries) such that $X_{n+1}$ contains the closure of $X_n$ in its interior. We denote by $\partial_0X_n$ the collection of geodesic boundary components of $X_n$ and recall that $\ell (\alpha)$ is the length of $\alpha\in\partial_0X_n$.   We first state an abstract theorem that holds for any set of disjoint collars about the $\alpha \in \partial_0X_n$. For a collar $R_{\alpha}$, the extremal distance between its two boundary components is denoted by $\lambda (R_{\alpha})$. We rephrase Proposition  \ref{prop:parabolic} in terms of collars.

\begin{proposition}
\label{abstract theorem on collars}
Let $X$ be an infinite type hyperbolic surface and  $\{X_n\}$  an exhaustion by finite area geodesic subsurfaces as above. Given a collection of disjoint collars  $\{R_{\alpha}\}_{\alpha \in \partial_0X_n}$  if
$$
\sum_{n=1}^{\infty} \frac{1}{\sum_{\alpha\in \partial_0 X_n}
\frac{1}{\lambda (R_\alpha)}}=\infty
$$ 
then $X$ is of parabolic  type.
\end{proposition}

Since standard collars about disjoint closed geodesics are disjoint, Proposition  \ref{prop:parabolic} and Lemma \ref{lem:collar} together with  $\frac{1}{\ell}\arctan \Big{[}\frac{1}{\sinh \frac{\ell}{2}}\Big{]}\asymp 1/(\ell e^{\ell /2})$ give us
 
\begin{proposition}
\label{prop:collar-parabolic}
Let $X$ be an infinite type hyperbolic surface and  $\{X_n\}$  an exhaustion by finite area geodesic subsurfaces as above. If
$$
\sum_{n=1}^{\infty} \frac{1}{\sum_{\alpha\in \partial_0 X_n}\ell (\alpha )e^{\ell (\alpha )/2 }}=\infty
$$ 
then $X$ is of parabolic  type.
\end{proposition}

The above proposition can be strengthened by using Corollary \ref{cor:asymp} when $\ell (\alpha )$ is large. Let $P$
be a single geodesic pair of pants with boundary geodesics $\alpha$, $\beta$ and $\beta'$, where $\beta'$ can degenerate to a puncture. Then the nonstandard half-collars around $\alpha$ and $\beta$ with the other boundaries having non-smooth point on $\beta'$ are disjoint.

Assume that $X_1$ contains at least two pairs of pants. It follows that each boundary geodesic of $X_1$ belongs to a pair of pants in $X_1$ such that at least one other boundary geodesic is in the interior of $X_1$. Thus the set of boundary geodesics of $X_1$ has a corresponding set of one-sided nonstandard collars in $X_1$ that are mutually disjoint by Lemma \ref{lem:disjoint half-collars}. The same property is true  for each boundary geodesic of $X_n$ because it belongs to a pair of pants in $X_n -  X_{n-1}$ that has at least one boundary geodesic in the interior of $X_n$. By Lemma \ref{lem:disjoint half-collars} all nonstandard half-collars around boundary geodesics of the exhaustion $\{ X_n\}_n$ are mutually disjoint.

Let $\alpha \in \partial_0 X_n$ and let $\lambda (R_{\alpha ,\gamma})$ be the extremal distance between the boundary components of the nonstandard half-collar $R_{\alpha ,\gamma}$ around $\alpha$ corresponding to orthogeodesic  $\gamma$.
Fix $\ell_0>0$ and $\gamma_0>0$. 
For $\ell (\gamma )<\gamma_0$ and $\ell (\alpha)>\ell_0$ we have
$1/\lambda (R_\alpha) \lesssim e^{\ell (\alpha)/2}/\ell (\gamma )$; for $\ell (\gamma )\geq\gamma_0$ and $\ell (\alpha)>\ell_0$ we have $1/\lambda(R_{\alpha}) \lesssim e^{\ell (\alpha)/2}$. For $\ell (\alpha )<\ell_0$ we have $1/\lambda(R_{\alpha}) \lesssim \frac{1}{\ell (\alpha )}$ by Lemma \ref{lem:collar}.
By setting
$$
\sigma (R_{\alpha ,\gamma })  :=\max \Big{\{} \frac{e^{\ell (\alpha)/2}}{\ell (\gamma )}, e^{\ell (\alpha)/2},\frac{1}{\ell (\alpha)}\Big{\}}
$$
we obtain

\begin{theorem}
\label{thm:general_parabolic}
Let $X$ be an infinite type hyperbolic surface with an exhaustion $\{X_n\}$. If
$$
\sum_{n=1}^{\infty}\frac{1}{\sum_{\alpha\in \partial_0 X_n} \sigma (R_{\alpha ,\gamma})}=\infty
$$
then $X$ is of parabolic  type.
\end{theorem}  

We use $|\partial_0 X_n|$ to denote the number of components on the boundary of $X_n$.

\begin{corollary}
Let $X$ be an infinite type   hyperbolic surface with an exhaustion 
$\{X_n\}$. Let $\sigma_n=\sup_{\alpha\in \partial_0 X_n} \sigma (\lambda (R_{\alpha ,\gamma}))$ where $R_{\alpha ,\gamma}$ is the nonstandard half-collar around $\alpha$ in $X_n$.
If
\begin{equation}
\sum_{n=1}^{\infty} \frac{1}{|\partial_0 X_n|  \sigma_n}=\infty
\end{equation}
then $X$ is of parabolic  type.
\end{corollary}  

Theorem \ref{thm:general_parabolic} and its corollary are twist free results.  In order to include twists to achieve sufficient conditions for parabolicity we use nonstandard collars. When the nonstandard collar is a collar (that is, not a half-collar)  one has to make  some topological restrictions (either on the topology of the surface or on the exhaustion)  to ensure the nonstandard collars are disjoint.  

Suppose $\{X_n\}$ is an exhaustion of $X$ as discussed at the beginning of this section, and fix a pants decomposition of $X$ whose pants curves include the boundary curves of $\{X_n\}$. Each boundary geodesic $\alpha$ of $X_n$ is an interior geodesic of $X$ and hence it makes sense to talk about the twist $t(\alpha)$ of $\alpha$.

\begin{theorem}
\label{thm:general_parabolicwithtwist}
Let $X$ be an infinite type hyperbolic surface with an exhaustion $\{X_n\}$. Assume   each connected component $Y$ of $X_{n+1}-X_{n}$  is not a pair of pants and  there is a pants decomposition of $Y$ so that the distance from any boundary geodesic of $Y$ to the boundary component of the pair of pants that is interior is  uniformly bounded from below.  If
\begin{equation}\label{eq:general_parabolicwithtwist}
\sum_{n=1}^{\infty}\frac{1}{\sum_{\alpha\in \partial_0 X_n} 
e^{\left(1- |t (\alpha)|\right)\frac{\ell (\alpha)}{2}}}=\infty
\end{equation}
then $X$ is of parabolic  type.
\end{theorem}

\begin{proof} By using  Lemma \ref{lem:disjoint half-collars} on the given pants decomposition we can conclude that the boundary components of $X_n$ have  disjoint nonstandard half-collars. Since the boundary components of $X_{n+1}-X_{n}$ also have disjoint nonstandard  half-collars, they glue  together to allow us to  conclude that the   boundary geodesics of each $X_n$  have nonstandard collars that   are pairwise disjoint.  The assumption on  each connected component $Y$ implies that the nonstandard collars around the boundaries of $X_n$ are simultaneously disjoint for all $n$. 
Using the estimates for the extremal length of nonstandard collars coming from Theorem \ref{thm:gluings} and Remark  \ref{rem:expansion}, formula (\ref{eq: lower bound without min})  in terms of length and twist allows us to conclude the result.
\end{proof}

If we set 
\begin{align*}
L_n&=\text{max}\{\ell (\alpha): \alpha \in \partial_0 X_n\},\\
\tau_n&=\text{min}\{|t_n(\alpha)|: \alpha \in \partial_0 X_n\},
\end{align*} 
we obtain the following corollary.

\begin{corollary}Let $X$ be as in Theorem \ref{thm:general_parabolicwithtwist}.  Then $X$ is parabolic if
\begin{equation}\label{eqn:div-sum}
\sum_{n=1}^{\infty}\frac{1}{|\partial_0 X_n|
e^{(1- \tau_n)\frac{L_n}{2}}}=\infty.
\end{equation}
\label{cor:p_twist}
\end{corollary}

If we assume that $|\partial_0 X_n| \asymp n^p$ for some $p\geq0$ and  $\tau_n \geq \tau > 0$ then by equation (\ref{eqn:div-sum}) $X$ would be parabolic if 
\begin{align*}
L_n\leq \frac{2}{1-\tau} \left[(1-p)\log n +\log\log n\right].
\end{align*}
In particular, we have the following two cases which will be useful in the discussion of abelian covers in Section  \ref{subsec: abelian covers}. Cases $(1)$ and $(2)$ correspond to abelian $\mathbb{Z}$ and $\mathbb{Z}^2$ covers.

\begin{example} We have the following sufficient conditions for parabolicity.
\begin{itemize}
\item[$(1)$] $|\partial_0 X_n| =O(1)$, $\tau_n \geq \tau > 0$, and $L_n\leq \frac{2}{1-\tau} \left[ \log n +\log\log n\right]$ 
\item[$(2)$] $|\partial_0 X_n| = O(n) $, $\tau_n \geq \tau > 0$, and $L_n\leq \frac{2}{1-\tau} \left[\log\log n\right]$ 
\end{itemize}
\end{example}

\begin{remark}
We note that increasing the  twist in equation (\ref{eq:general_parabolicwithtwist}) preserves divergence. 
More specifically, if $|t^{\prime}(\alpha)| \geq |t (\alpha)|$ for all 
$\alpha$, then parabolicity persists. 
\end{remark}

 The hypotheses of the next theorem are  somewhat different from the others in that we make an assumption about the existence of  half-collars with no relation to the length of the core geodesic. 
  
 \begin{theorem}\label{thm: collar width}
  Let $X$ be an infinite type hyperbolic surface with an exhaustion $\{X_n\}$ so that  $\partial X_n$ has one boundary component and  is contained in a   half-collar of width  $\epsilon_n$. Set $\ell_n = \ell(\partial X_n)$.  If 
 \begin{equation} \label{ }
\sum \frac{\arctan (\sinh \epsilon_n)}{\ell_n}=\infty.
\end{equation}
Then $X$ is parabolic. 
 \end{theorem}
 
 \begin{remark}
 The main application here is to  situations  where  the length of the core geodesic is  going  to infinity but the collar width is larger  than the standard collar width.  For example, this happens  if there is a lower bound on the width for all $n$ (see Theorem \ref{thm:top. covering group}).
\end{remark}
 
 \begin{proof}
 In  the proof of Lemma \ref{lem:collar}  using $\epsilon_n$ for the collar width instead of the standard collar width $r(\ell_n/2)$ yields that the curve family joining the  boundary components of the half-collar  has extremal length 
 $\frac{\arctan (\sinh \epsilon_n)}{\ell_n}$. The theorem now follows by applying Proposition  \ref{abstract theorem on collars}  to this  case.  The details are left to the reader. 
 \end{proof}
 

\section{Flute surfaces and pants decomposition}
\label{sec:flute}

In this section we apply  results of previous sections to obtain sufficient conditions for parabolicity for    tight flute surfaces. We say that $X$ is a \textit{tight flute surface} if it can be constructed by consecutively gluing a sequence of tight pairs of pants $P_n$ (see Figure \ref{fig:tight-flute}). More specifically, given a sequence $\ell_n>0$ of lengths  and a sequence $t_n\in[-1/2,1/2)$ of twists, $n\geq 1$, we define the corresponding tight flute surface 
$$X=X(\{\ell_n,t_n\})$$ as follows. 

%
%

Let $P_0$ be a tight pair of geodesic pants with two punctures and one  boundary geodesic $\alpha_1$ of length $\ell (\alpha_1)=\ell_1$. For $n\geq 1$, let $P_n$ be a tight pair of geodesic pants with one puncture and two  boundary geodesics $\alpha_n$ and $\alpha_{n+1}$ of lengths $\ell (\alpha_n)=\ell_n$ and $\ell (\alpha_{n+1})=\ell_{n+1}$.
In particular, $P_{n-1}$ and $P_{n}$ both have boundary geodesics of length $\ell_n$ which (abusing the notation slightly) will both be denoted by $\alpha_n$. Then we glue by an isometry $P_{n-1}$ to $P_n$ along $\alpha_{n}$ for $n\geq 1$. 

The choice in the gluing of $P_{n-1}$ and $P_n$ is given by a (relative or angle) \emph{twist parameter} $t_{n}\in [-1/2,1/2)$ (see Section \ref{sec: Contents, notation, and convention}). The \textit{absolute twist}  $|t_n|\ell_n$ is the shortest distance between the feet of the orthogeodesics $\gamma_{n-1}''$ and $\gamma_n'$  to $\alpha_n$ (see Section \ref{sec: Contents, notation, and convention} for the definition and Figure \ref{fig:tight-flute} for the notation). The surface $X$ obtained by these consecutive gluings is called a tight flute surface (see \cite{Basmajian}). We denote by $X(\{ \ell_n,t_n\})$ the resulting hyperbolic surface. The surface $X(\{ \ell_n,t_n\})$ may not be complete in which case its completion would contain geodesic half-planes and would not be of parabolic type.

One of the main applications of the general results of this paper is a  sufficient condition for parabolicity for flute surfaces. Specifically, from Corollary \ref{cor:twist1/2} and Proposition \ref{prop:parabolic} we obtain the following.

\begin{theorem}\label{thm:flute-twist}
A tight flute surface $X(\{\ell_n,t_n\})$, with $\ell_n\geq \ell_0>0$ is of parabolic  type  if
\begin{equation}\label{eqn:parbolicity}
\sum_{n=1}^{\infty}e^{-\left(1- |t_n|\right)\frac{\ell_n}{2}}=\infty.
\end{equation}
\end{theorem}

\begin{proof}
Let $R_{\alpha_{n},\gamma_{n}'}$ and $R_{\alpha_{n+1},\gamma_{n}''}$ be the nonstandard half-collars in $P_n$ about $\alpha_{n}$ and $\alpha_{n+1}$, where $\gamma_{n}'$ and $\gamma_{n}''$ are as in Figure \ref{fig:tight-flute}. The bi-infinite simple geodesic $\beta_n\subset P_n$ from the puncture to itself is the common boundary of these half-collars.  
Then $R_{\alpha_{n},\gamma_{n}'}$ and $R_{\alpha_{n+1},\gamma_{n}''}$ are disjoint and 
$R_{\alpha_{n},\gamma_{n}'}\cup \beta_n \cup R_{\alpha_{n+1},\gamma_{n}''} = P_n$. 

Let $R_{\alpha_{n},\gamma_{n-1}'',\gamma_n'}^{t(\alpha_n)}$ be the nonstandard collar around $\alpha_n$ in $P_{n-1}\cup P_{n}$. Thus $R_{\alpha_{n},\gamma_{n-1}'',\gamma_n'}^{t(\alpha_n)} = R_{\alpha_{n},\gamma_{n-1}''} \cup \alpha_n \cup R_{\alpha_{n},\gamma_n'}$. If  $X_n =\cup_{i=0}^{n-1} P_i$ then 
$R_{\alpha_{n},\gamma_{n-1}'',\gamma_n'}^{t(\alpha_n)}\subset X_{n+1}- X_{n-1}$. The boundary components of $R_{\alpha_{n},\gamma_{n-1}'',\gamma_n'}^{t(\alpha_n)}$ are $\beta_{n-1}\subset X_n$ and $\beta_n\subset X_{n+1}-X_n$. The interiors of $R_{\alpha_{n},\gamma_{n-1}'',\gamma_n'}^{t(\alpha_n)}$ and $R_{\alpha_{n-1},\gamma_{n-2}'',\gamma_{n-1}'}^{t(\alpha_{n-1})}$ are disjoint and Proposition \ref{prop:parabolic} applies.

If $\lambda_n $ is the extremal distance between the boundary components $\beta_{n-1}$ and $\beta_n$ of $R_{\alpha_{n},\gamma_{n-1}'',\gamma_n'}^{t(\alpha_n)}$   then by Corollary \ref{cor:twist1/2} we have 
$$\sum_n\lambda_n=\sum_n\lambda(R_{\alpha_{n},\gamma_{n-1}'',\gamma_n'}^{t(\alpha_n)}) \geq C \sum_n e^{-(1-|t_n|)\ell_n/2}.$$ Thus (\ref{eqn:parbolicity}) implies that $\sum_{n}\lambda_n=\infty$ and $X$ is of parabolic type by Proposition \ref{prop:parabolic}.
%
%
%
%
%
%
%
%
%
%
%
%
%
%
\end{proof}

\begin{remark}
Note that in the case of a flute surface $X$ even though $X_{n+1}-X_n$ is a pair of pants the nonstandard collars around the geodesics $\alpha_n$ are disjoint. This is a consequence of the fact that pairs of pants $\{ P_n\}$ are glued in a chain by attaching one boundary component of $P_{n+1}$ to the boundary component of $X_n$. 
Thus the  nonstandard collars have disjoint interior. Therefore we do not need to apply Lemma \ref{lem:disjoint half-collars} to flute surfaces.
\end{remark}

%
%

Since $|t_n|\in[0,1/2]$  we obtain the following ``twist-free" corollary of Theorem \ref{thm:flute-twist}.

\begin{corollary}
A tight flute surface $X(\{ \ell_n,t_n\})$ is of parabolic type,  independent of twists, if 
\begin{align}\label{eqn:zero-twists}
\sum_n e^{-\ell_n/2}=\infty.
\end{align}
\end{corollary}

In the case of zero-twist flute surfaces Theorem \ref{thm:flute-twist} gives a complete characterization of parabolicity. We denote by $\delta_n$ the orthogeodesic between $\alpha_n$ and $\alpha_{n+1}$ and by $\ell (\delta_n)$ its length.

\begin{theorem}[Zero-twist flutes]\label{thm:zero-twists}
 Let $X=X(\{\ell_n,0\})$ be a zero-twist tight flute surface. The following are equivalent 
 \begin{enumerate}
\item $X$ is of parabolic type,
\item $X$ is complete, 
\item $\sum_n \ell (\delta_n)=\infty$,
\item $\sum_{n=1}^{\infty}e^{-\ell_{n}/2}=\infty$.
 \end{enumerate}
 \end{theorem} 

\begin{proof} Since every incomplete surface carries a Green's function it is not parabolic, and therefore $(1)\Rightarrow(2)$. To show that $(2)\Rightarrow (3)$, note that if $t_n=0$ then the orthogeodesics $\delta_n$ between $\alpha_{n}$ and $\alpha_{n+1}$ connect to each other at their endpoints to form a geodesic ray of length $\sum_{n=1}^{\infty}\ell (\delta_n)$ that leaves every compact subset of $X$. Therefore if $\sum_{n=1}^{\infty}\ell (\delta_n)<\infty$ then $X$ is incomplete. 

By (\ref{eq: 3 quantities}) in Section \ref{sec:non-standard-collar} we have $\ell (\delta_n)\asymp e^{-\ell_n/2}$ if $\ell (\gamma )=\infty$ and $\ell_n\to\infty$ and hence $(3)\Leftrightarrow(4)$. Finally, by Theorem \ref{thm:flute-twist} we have $(4)\Rightarrow(1)$ since $t_n=0$.
\end{proof}

\begin{example}
Let $X=X(\{\ell_n,0\})$ be a zero-twist surface.
\begin{itemize}
\item[(1)] Suppose $\lim_{n\to\infty} \frac{\ell_n}{\ln n} = c\geq0$. Then $X$ is parabolic if $c\in[0,2)$, not parabolic if $c\in(2,\infty)$, and $X$ could be either parabolic or not parabolic if $c=2$, see e.g. $(3)$ below.
\item[(2)] Let $\ell_n=c\ln n$. Then $X$ is parabolic if and only if $c\in(0,2]$.
\item[(3)] Let $\ell_n=2\ln n +c\ln(\ln n)$. Then $X$ is parabolic if and only if $c\in[0,2]$.
\end{itemize} 
\end{example}


\subsection{Incomplete half-twist tight flutes}
A tight flute surface all of whose twists are $1/2$ is called a {\it half-twist} tight flute. 
By Theorem \ref{thm:flute-twist} we have that 
a half-twist surface $X(\{\ell_n,1/2\})$ is parabolic if 
\begin{align}\label{eqn:half-twists}
\sum_n e^{-\frac{\ell_n}{4}}=\infty.
\end{align}
Even though we do not know if (\ref{eqn:half-twists}) is equivalent  to parabolicity of $X$, in general we will show that this is the case under some mild assumptions on the lengths $\ell_n$, see Theorem \ref{thm:half-twists}. We also provide examples illustrating how our sufficient conditions for parabolicity and non-parabolicity (in fact incompleteness) complement each other, see Example \ref{example:parameters}.
%
%
%
To do this we first obtain a sufficient condition for $X$ to be incomplete and hence not of parabolic type.

As mentioned above every parabolic surface is necessarily complete. However, there are examples of complete flute surfaces which are not parabolic, cf., \cite{Kinjo}. Also,  in  \cite{Basmajian-Saric} it was shown that for any choice of lengths $\{ \ell_n\}$ there is a choice of twists $\{ t_n\}$ such that $X(\{ \ell_n,t_n\})$ is complete.
This choice is not constructive and given an explicit choice of twists $\{ t_n\}$ it is difficult to decide whether a surface is complete or incomplete.
Thus it is natural to look for conditions implying completeness or incompleteness in terms of the Fenchel-Nielsen parameters, $\{\ell_n\}$ and $\{t_n\}$.

One might expect that given a sequence  of lengths $\{\ell_n\}$  then the choice of largest possible twists, i.e., half-twists $|t_n|={1}/{2}$ is
``the best possible'' choice of twists for  $X=X(\{ \ell_n,t_n\})$ implying completeness (see 
Matsuzaki \cite{Matsuzaki-complete} for a related choice $t_n={1}/{4}$). In a sense, our next result identifies a class of surfaces for which this is not true.

\begin{figure}[t]
\begin{center}
\includegraphics[width=4.5in]{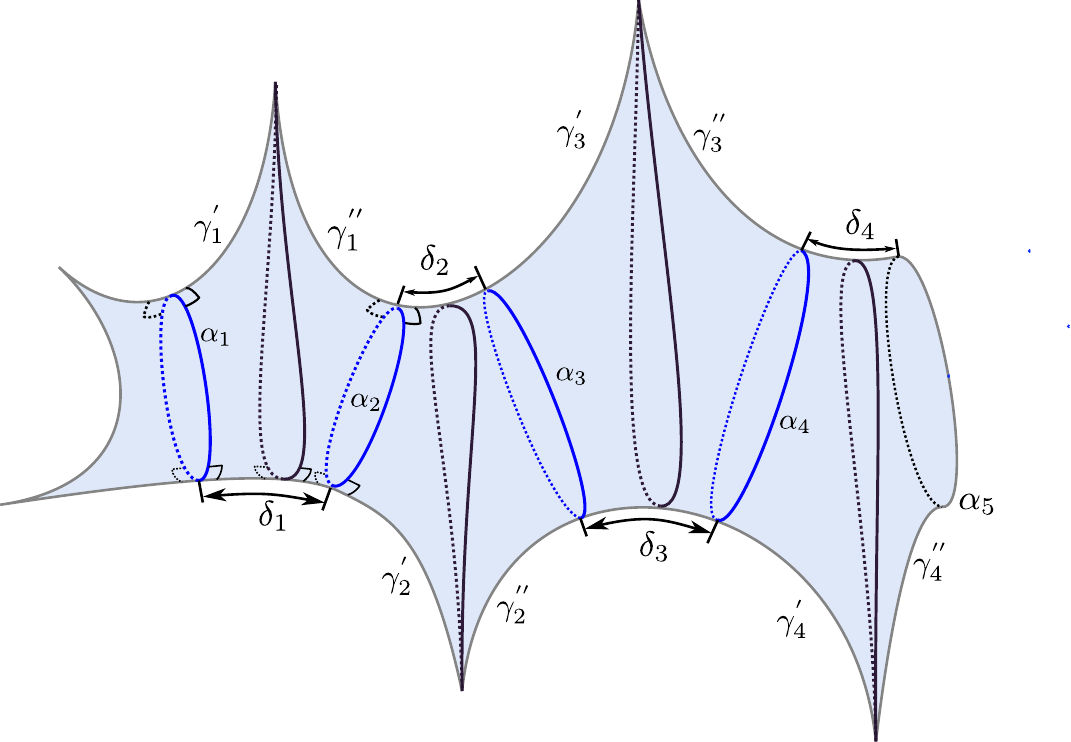}
\caption{A half-twist tight flute surface $X=X(\{\ell_n,1/2\})$. Note that $\ell_n$ and $d_n$ denote the lengths of $\alpha_n$ and the orthogeodesic $\delta_n$, respectively.}
\label{fig:half-twist}
\end{center}
\end{figure} 

\begin{theorem}
\label{thm:half-twist-incomplete}
A half-twist tight flute surface $X(\{ \ell_n,\frac{1}{2}\})$ is incomplete if   
\begin{align}\label{eqn:half-twist-incomplete}
\sum_n e^{-\frac{\sigma_n}{2}}<\infty,
\end{align}
where $\sigma_n = \ell_n-\ell_{n-1}+\cdots +(-1)^{n-1}\ell_1$.
\end{theorem}

Before proving Theorem \ref{thm:half-twist-incomplete} we discuss some of its consequences. An immediate corollary of Theorem \ref{thm:half-twist-incomplete} is that if a half-twist tight flute surface $X(\{ \ell_n,\frac{1}{2}\})$ is parabolic then
\begin{align}\label{eqn:half-twist-parabolic-necessary}
\sum_n e^{-\frac{\sigma_n}{2}}=\infty.
\end{align}
%
%

In view of the above it is natural to seek conditions on  $\{\ell_n\}$ so that  (\ref{eqn:half-twists}) and (\ref{eqn:half-twist-parabolic-necessary}) are equivalent (or not). First, note that if $\ell_n$ has a bounded subsequence then (\ref{eqn:half-twists}) holds and $X$ is parabolic. Therefore, to obtain a non-trivial condition for parabolicity we will always assume that $\ell_n\to\infty$. 


We say that $\{\ell_n\}$ is a \emph{concave sequence} if  and there is  non-decreasing concave function $f:[0,\infty)\to[0,\infty)$ such that $\ell_n=f(n)$ for $n\geq 0$. Equivalently, $\{\ell_n\}$ is concave if it is non-decreasing and for $n\geq 1$ the following holds:
\begin{align}\label{ineq:concave}
2 \ell_{n}\geq \ell_{n+1}+\ell_{n-1}.
\end{align}

For half-twist surfaces corresponding to concave sequences Theorem \ref{thm:half-twist-incomplete} gives the following characterization of parabolicity (analogous to Theorem \ref{thm:zero-twists} in the case of zero-twists).

\begin{theorem}\label{thm:half-twists}
Let $X=X(\{\ell_n,1/2\})$ such that $\{\ell_{n}\}$ is a concave sequence. Then the following are equivalent:
\begin{enumerate}
\item $X$ is parabolic,
\item $X$ is complete,
\item $\sum_n \sqrt{\ell (\delta_n)} =\infty$,
\item $\sum_n e^{-\ell_n/4} =\infty$.
\end{enumerate}
\end{theorem}

\begin{proof}[Proof of Theorem \ref{thm:half-twists}] Just like in the proof of Theorem \ref{thm:zero-twists} the implication $(1)\Rightarrow(2)$ is true in general, $(3)\Leftrightarrow(4)$  holds because of $\ell (\delta_n)\asymp e^{-\ell_n/2}$, and $(4)\Rightarrow(1)$ follows from Theorem \ref{thm:flute-twist}.  Thus we only need to show the implication $(2)\Rightarrow(4)$. 
Equivalently, we will show that $X$ is incomplete if 
\begin{align}\label{eqn:concave-incomplete}
\sum_{n}e^{-\ell_n/4}<\infty.
\end{align}

Observe that from the definition of $\sigma_n$ we have for every $n\geq 1$ 
\begin{align}\label{eqn:sigmas}
\begin{split}
\sigma_n+\sigma_{n-1}&=\ell_n,\\
\sigma_{n+2}-\sigma_n&=\ell_{n+2}-\ell_{n-1}.
\end{split}
\end{align}

Let $\epsilon_n = \ell_{2n}-\ell_{2n-1}$, and $
\eta_n = \ell_{2n+1}-\ell_{2n}$ for $n\geq 1$. Since $\{\ell_n\}$ is non-decreasing we have that $\epsilon_n\geq 0$ and $\eta_n\geq 0$. Moreover, since $\{\ell_n\}$ is concave, rewriting  (\ref{ineq:concave}) we have $\ell_{n+1}-\ell_n \leq \ell_{n}-\ell_{n-1}$. In particular, for all $n\geq 1$ we have
\begin{align}\label{ineq:alternate}
\epsilon_{n+1}\leq \eta_n\leq \epsilon_{n}.
\end{align}
%
%
%
Using (\ref{ineq:alternate}) we obtain 
\begin{align*}
\sigma_{2k+1} & = \sum_{i=1}^k \eta_i+\ell_1 
\leq 
\sum_{i=1}^k \epsilon_i +\ell_1
=  \sigma_{2k} + \ell_1,\\
\sigma_{2k+1} & \geq \sum_{i=1}^{k-1} \eta_i+\ell_1 \geq  \sum_{i=1}^k \epsilon_i -\epsilon_1 +\ell_1 \geq
\sigma_{2k} - \ell_2.
\end{align*}
Thus $|\sigma_{2k+1}-\sigma_{2k}| \leq \max\{\ell_1,\ell_2\}\leq \ell_2$. Therefore, using the second line in (\ref{eqn:sigmas}) and the fact that $\ell_{2k+1}-\ell_{2k}$ is non-increasing, we obtain
\begin{align*}
|\sigma_{2k}-\sigma_{2k-1}| 
&\leq |\sigma_{2k+1}-\sigma_{2k}| +|\sigma_{2k+1}-\sigma_{2k-1}| \\
& \leq \ell_2 + \ell_{2k+1}-\ell_{2k-1}\\
&\leq 2\ell_2.
\end{align*}
In particular, $|\sigma_n - \sigma_{n-1}|\leq 2\ell_2$ for every $n\geq 1$. Since $\ell_n = \sigma_n+\sigma_{n-1}$ we obtain
\begin{align}\label{ineq:alternate1}
|2\sigma_n - \ell_n| =|\sigma_n-\sigma_{n-1}|
\leq 2\ell_2.
\end{align} 
From (\ref{ineq:alternate1}) it follows that  
$e^{-\sigma_n/2} \leq e^{\ell_2/2} e^{-\ell_n/4}.$
Thus by Theorem \ref{thm:half-twist-incomplete} $(\ref{eqn:concave-incomplete})$ implies $(\ref{eqn:half-twist-incomplete})$ and $X$ is incomplete, which completes the proof.
\end{proof}

 \begin{example}
Let $X=X(\{\ell_n,1/2\})$ be a half-twist tight flute surface such that $\{\ell_n\}$ is concave.
\begin{itemize}
\item[(1)] Suppose $\lim_{n\to\infty} \frac{\ell_n}{\ln n} = c\geq0$. Then $X$ is parabolic if $c\in[0,4)$, not parabolic if $c\in(4,\infty)$, and $X$ could be either parabolic or not parabolic if $c=4$, see e.g. $(3)$ below.
\item[(2)] Let $\ell_n=c\ln n$. Then $X$ is parabolic if and only if $c\in(0,4]$.
\item[(3)] Let $\ell_n=4\ln n +c\ln(\ln n)$. Then $X$ is parabolic if and only if $c\in[0,4]$.
\end{itemize} 
\end{example}

\begin{proof}[Proof of Theorem \ref{thm:half-twist-incomplete}]
Let $P_n$, for $n\geq 1$, be a tight pair of pants with two boundary geodesics $\alpha_{n}$ and $\alpha_{n+1}$ of lengths $\ell_{n}$ and $\ell_{n+1}$ and third boundary a puncture. Let $P_0$ be a tight pair of pants with one geodesic boundary $\alpha_1$ of length $\ell_1$ and two other boundaries being punctures. Denote by $\delta_n$ the orthogeodesic connecting $\alpha_{n}$ and $\alpha_{n+1}$ in $P_n$. Denote by $\gamma_n'$ the simple geodesic ray orthogonal to $\alpha_{n}$ ending in the puncture of $P_n$, and denote by $\gamma_n''$ the simple geodesic ray orthogonal to $\alpha_{n+1}$ ending in the puncture of $P_n$. Then $P_n$ has front to back hyperbolic symmetry with $\delta_n$ and $\gamma_n'\cup\gamma_n''$ as two arcs of fixed points of the symmetry.

 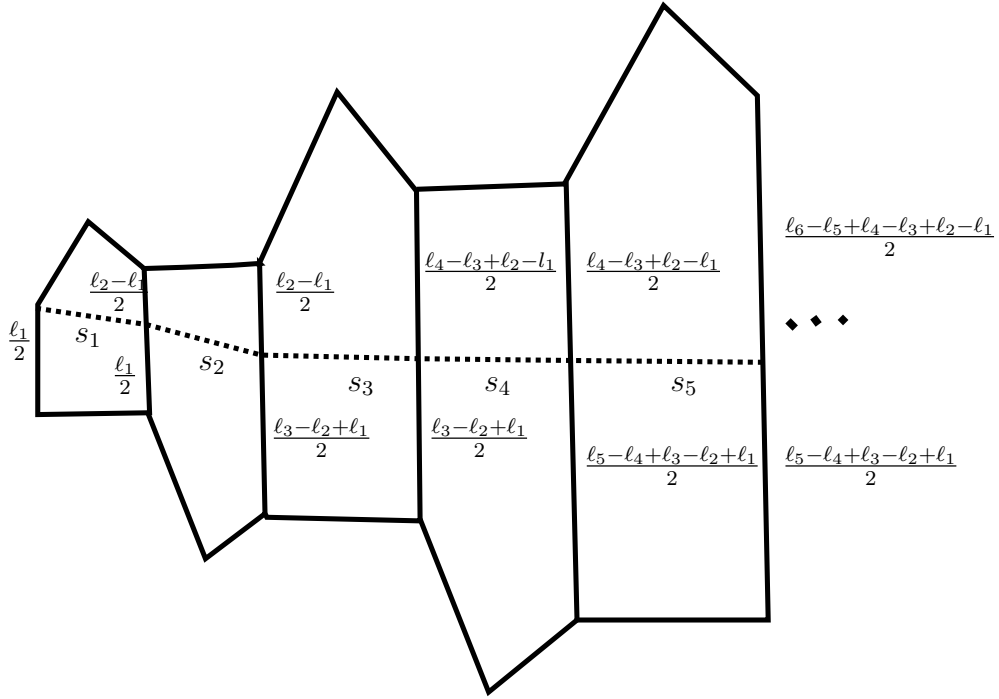
\begin{figure}[t]
\leavevmode \SetLabels
\L(.145*.5) $\frac{\ell_1}{2}$\\
\L(.23*.45) $\frac{\ell_1}{2}$\\
\L(.21*.57) $\frac{\ell_2-\ell_1}{2}$\\
\L(.36*.57) $\frac{\ell_2-\ell_1}{2}$\\
\L(.358*.36) $\frac{\ell_3-\ell_2+\ell_1}{2}$\\
\L(.485*.36) $\frac{\ell_3-\ell_2+\ell_1}{2}$\\
\L(.48*.60) $\frac{\ell_4-\ell_3+\ell_2-l_1}{2}$\\
\L(.61*.60) $\frac{\ell_4-\ell_3+\ell_2-\ell_1}{2}$\\
\L(.61*.32) $\frac{\ell_5-\ell_4+\ell_3-\ell_2+\ell_1}{2}$\\
\L(.77*.32) $\frac{\ell_5-\ell_4+\ell_3-\ell_2+\ell_1}{2}$\\
\L(.77*.65) $\frac{\ell_6-\ell_5+\ell_4-\ell_3+\ell_2-\ell_1}{2}$\\
\L(.2*.51) $s_1$\\
\L(.3*.47) $s_2$\\
\L(.42*.44) $s_3$\\
\L(.53*.44) $s_4$\\
\L(.68*.44) $s_5$\\
\endSetLabels
\begin{center}
\AffixLabels{\centerline{\epsfig{file =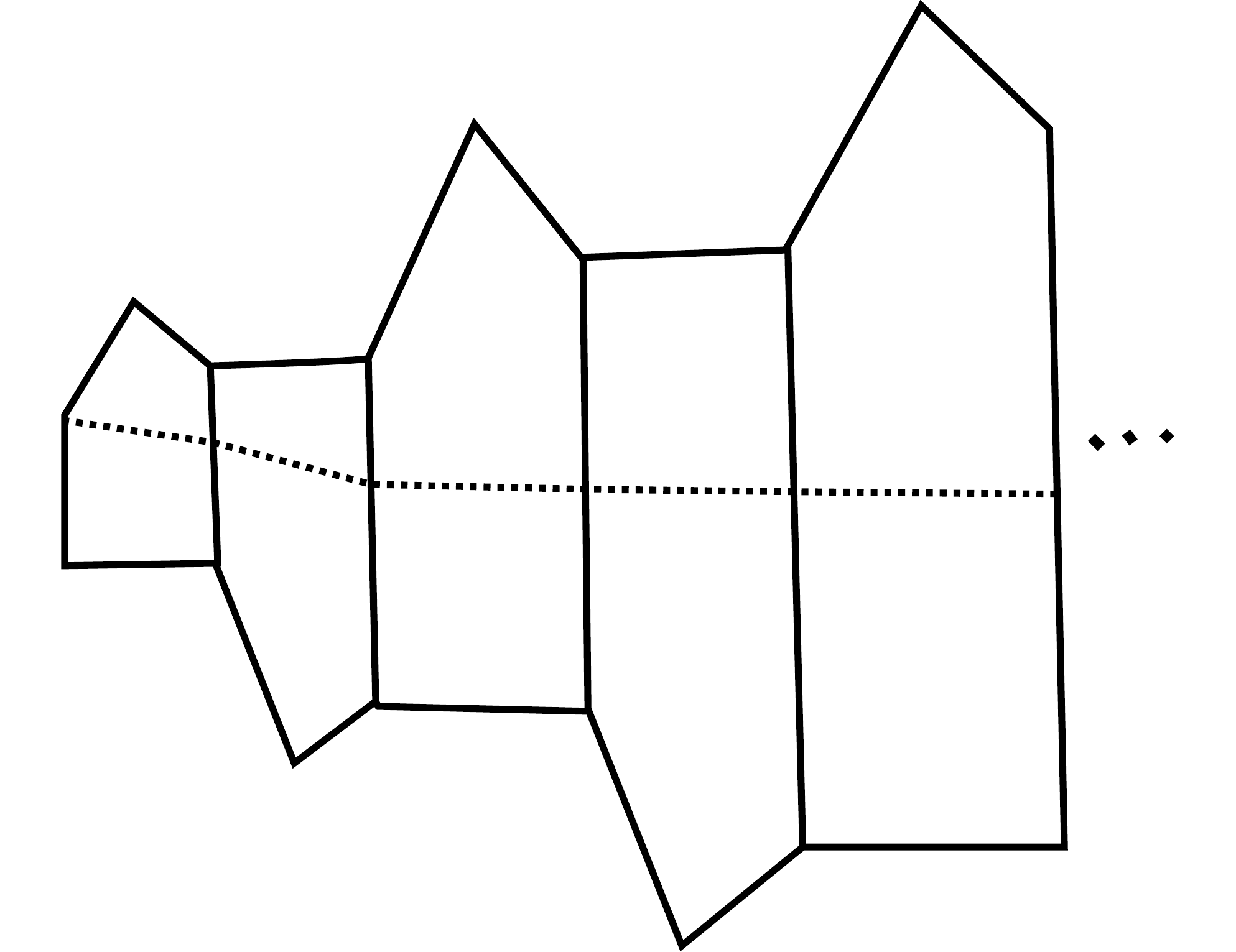,width=12.0cm,angle=0} }}
\vspace{-20pt}
\end{center}
\caption{The dotted path $p$ consisting of summits of Saccheri quadrilaterals.}\label{fig:Saccheri}
\end{figure} 

 We glue $\{ P_n\}_{n=0}^{\infty}$ along boundary geodesics of equal lengths with half-twists $t_n=\frac{1}{2}$ to obtain a half-twist tight flute surface $X$. The geodesic arc $\delta_n$ is continued by $\gamma_{n+1}'$ and they make a geodesic ray. Similarly $\gamma_n''\cup\delta_{n+1}$ is a geodesic ray. The surface $X$ has front to back orientation reversing hyperbolic symmetry with one arc of fixed point being $\gamma_1'\cup\gamma_1''\cup\delta_2\cup\gamma_3'\cup\gamma_3''\cup\delta_4\cup\cdots$ and the other arc of fixed points being $\delta_1\cup\gamma_2'\cup\gamma_2''\cup\delta_3\cup\gamma_4'\cup\gamma_4''\cup\cdots$ (see Figure \ref{fig:half-twist}).

We consider the front side of $X$ and construct a path $p$ (in the front side of $X$) of finite length starting from $\alpha_1$ and intersecting each $\alpha_n$. Such a path $p$ leaves every compact subset of $X$ and hence  it will follow that $X$ is incomplete. The front side of $X$ consists of pentagons $\Sigma_n$ which are the front sides of the pants $P_n$. In the pentagon $\Sigma_1$, we make a Saccheri quadrilateral with the base $\delta_1$, one side being half of $\alpha_1$ (which has length $\frac{1}{2}\ell_1$), the other side being a part of $\alpha_2$ of length $\frac{1}{2}\ell_2$ and denote the length of the summit by $s_1$. We construct a Saccheri quadrilateral in $\Sigma_2$ which has base $\delta_2$, one sides lying on $\alpha_2$ of length $\frac{1}{2}(\ell_2-\ell_1)$ and the other side on $\alpha_3$, and the summit whose length is $s_2$ shares a point in common with the summit in $\Sigma_1$. We continue this construction through all pentagons $\Sigma_n$ and the path $p$ is the concatenation of the summits of the Saccheri quadrilaterals (see Figure \ref{fig:Saccheri}).
 
 We first estimate the length of the base $\delta_n$ of the pentagon $\Sigma_n$. Draw an orthogonal to $\delta_n$ from the ideal vertex of $\Sigma_n$. Let $\delta_n^1$ and $\delta_n^2$ be the two arcs that $\delta_n$ is divided into by the orthogonal. Then by taking $\varphi =0$ in \cite[page 38, Theorem 2.3.1 (i)]{Buser} we have
$$
\sinh\ell (\delta_n^1)\sinh \frac{\ell_n}{2} =1, \ \ \ \ \sinh\ell (\delta_n^2)\sinh \frac{\ell_{n+1}}{2} =1.
$$
Since $\ell_n$ is large, we obtain $\ell (\delta_{n}^1)+\ell (\delta_{n}^2)=\ell (\delta_{n})\asymp e^{-\ell_{n}/2}+e^{-\ell_{n+1}/2}$.
 
Next we estimate the lengths $s_n$. Note that the two lengths of the sides of the Saccheri quadrilateral in $\Sigma_n$ are $\frac{1}{2}(\ell_{n}-\ell_{n-1}+\cdots +(-1)^{n-1}\ell_1)$ for $n\geq 2$. 
 We divide each Saccheri quadrilateral into two quadrilaterals by the common orthogonal to $s_n$ and $\delta_n$. The two new quadrilaterals have three right angles and 
 we have  (see \cite[page 38, Theorem 2.3.1 (v)]{Buser})
$$
\sinh\frac{1}{2}s_n=\sinh \frac{\ell ({\delta_n})}{2}\cosh \frac{\ell_{n}-\ell_{n-1}+\cdots +(-1)^{n-1}\ell_1}{2}.
$$
By $\cosh \frac{\ell_{n}-\ell_{n-1}+\cdots +(-1)^{n-1}\ell_1}{2}\asymp e^{\frac{\ell_{n}-\ell_{n-1}+\cdots +(-1)^{n-1}\ell_1}{2}}$ and $\ell (\delta_{n})\asymp e^{-\ell_{n}/2}+e^{-\ell_{n+1}/2}$, 
the above formula implies that
$$
s_n\asymp e^{\frac{-\ell_{n+1}+\ell_{n}-\ell_{n-1}+\cdots +(-1)^{n+1}\ell_1}{2}}+e^{\frac{-\ell_{n-1}+\cdots +(-1)^{n-1}\ell_1}{2}}.$$ Therefore the length of the path $p$, i.e. the sum $\sum_ns_n$, is finite and $X$ is incomplete.
\end{proof}

\begin{example}\label{example:parameters}
To illustrate the range of applicability  of Theorems  \ref{thm:flute-twist} and \ref{thm:half-twist-incomplete} we consider a two-parameter family of half-twist tight flute surfaces $X_{a,b}$ with $a,b>0$. See Figure \ref{fig:example-parameters}.

For every pair of points $0<a,b<\infty$   we will define a sequence $\Bell_{a,b}=\{\ell_n\}$ and consider the corresponding half-twist tight flute surface $X_{a,b}=X(\Bell_{a,b},1/2)$. The sequence $\Bell_{a,b}$ will be chosen in such a way that 
\begin{align}\label{example:parameter-sigmas}
\begin{split}
\sigma_{2n} &\simeq a \ln n,\\
\sigma_{2n+1}&\simeq b\ln n.
\end{split}
\end{align}
where as before $\sigma_n=\ell_n-\ell_{n-1}+\ldots+(-1)^{n-1}\ell_1$.
For instance, if we let $\ell_1 \in (0, a\ln 2)$ and for $n\geq 1$ define 
\begin{align}\label{example:lengths}\begin{split}
\ell_{2n} &= a \ln(n+1) + b \ln n,\\
\ell_{2n+1}&=(a+b) \ln(n+1),
\end{split}
\end{align} 
then a simple calculation shows that $\sigma_{2n+1}= b\ln(n+1)+\ell_1$ and $\sigma_{2n}= a\ln(n+1)-\ell_1,
$
and therefore (\ref{example:parameter-sigmas}) is satisfied. 

From (\ref{example:lengths}) we have that 
%
%
 $\ell_n\simeq (a+b) \ln n$. Therefore $e^{-\ell_n/4} \asymp   {n^{-(a+b)/4}},$
and by Theorem \ref{thm:flute-twist} $X_{a,b}$ is parabolic if the series $\sum_n n^{-(a+b)/4}$ diverges, or equivalently if $a+b\leq 4$, see the blue triangle in Figure \ref{fig:example-parameters}.

On the other hand, from (\ref{example:parameter-sigmas}) we have
\begin{align*}
\sum_{n} e^{-\sigma_n/2} 
&=\sum_k \left( \frac{1}{e^{\sigma_{2k}/2}} + \frac{1}{e^{\sigma_{2k+1}/2}} \right)
\asymp \sum_k \left( \frac{1}{k^{a/2}} + \frac{1}{k^{b/2}} \right)\\
&\asymp \sum_k \frac{1}{k^{\min(a,b)/2}},
\end{align*}

Thus, by Theorem \ref{thm:half-twist-incomplete}, $X_{a,b}$ is incomplete if $\min\{a,b\}>2$, see the interior of the unbounded shifted quadrant in Figure \ref{fig:example-parameters}.

In particular, for every $m>1$ the surface $X_{a,ma}$ is parabolic for $a\in[0,4/(m+1)]$, incomplete hence not parabolic for $a>2$, and the type is not known for $a\in(4/(m+1),2]$ (this corresponds to the line through the origin in Figure \ref{fig:example-parameters}). 
\end{example}

\begin{figure}[t]
\begin{center}
\includegraphics[width=2in]{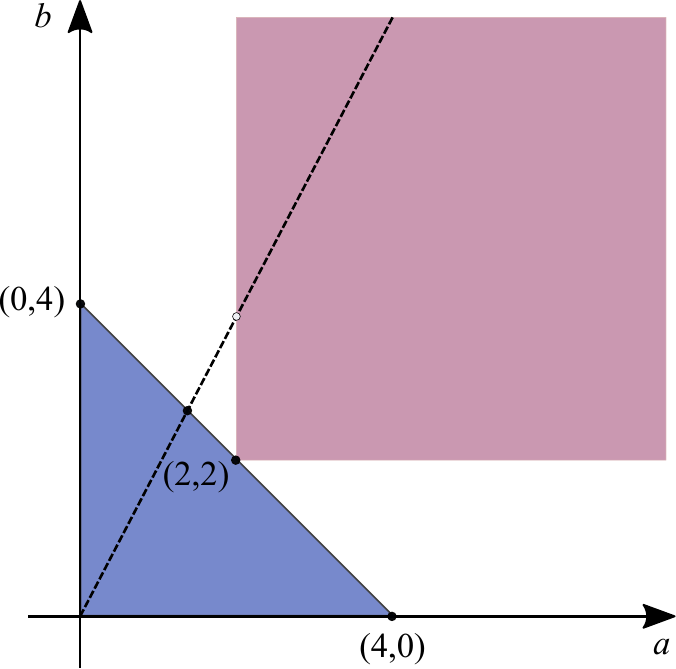}
\caption{The parameter space $(a,b)$ for the family of half-twist tight flute surfaces $X_{a,b}$. In the closure of the blue triangle surfaces are parabolic. In the interior of the pink unbounded region the surfaces are incomplete and hence not parabolic. In the white regions our results are inconclusive. If $m>1$ the surfaces $X_{a,ma}$ are parabolic for $a\in[0,4/(m+1)]$, hyperbolic for $a>2$, and are of unknown type for $a\in(4/(m+1),2]$.}
\label{fig:example-parameters}
\end{center}
\end{figure} 
%
%
%


\subsection{Bi-infinite tight flute surface}

\begin{figure}
\includegraphics[width=3in]{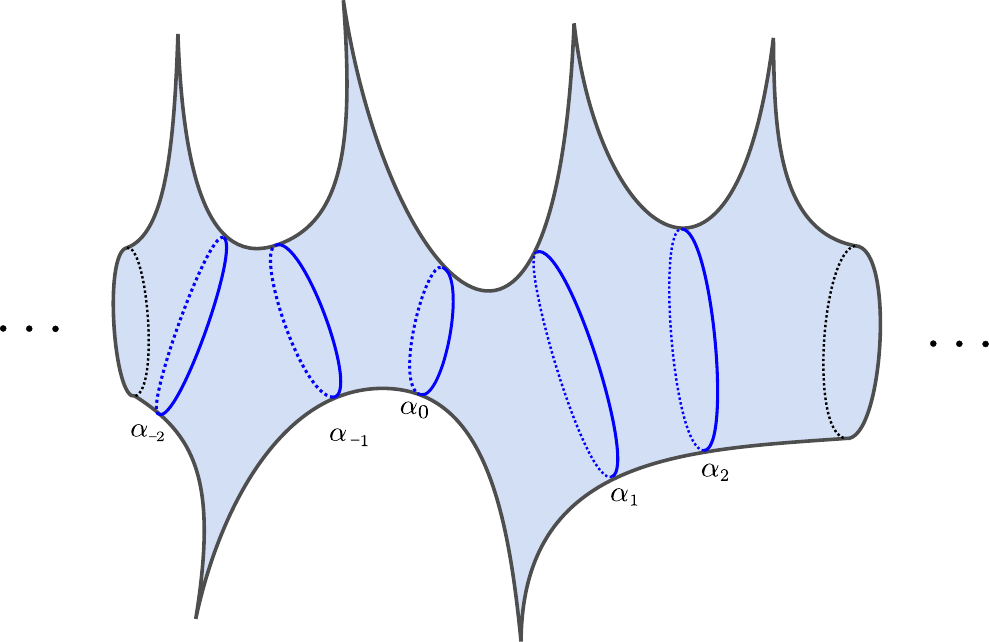}
\qquad
\caption{The bi-infinite tight flute.} 
\label{fig: bi-infinite tight flute}
\end{figure}

We refer to Figure \ref{fig: bi-infinite tight flute} which we call the bi-infinite flute surface.  All of its ends are planar and its space of ends 
is homeomorphic to the subset of the real line, $\{0\} \cup \{1\} \cup \{\frac{1}{n} \}_{n \in \Bbb N}\cup \{1-\frac{1}{n} \}_{n \in \Bbb N}$. The isolated ends are all cusps and by Theorem \ref{thm:general_parabolicwithtwist} we have. 

\begin{theorem}
\label{thm:biinfinite}
Let  $X$ be a bi-infinite tight flute as in  Figure \ref{fig: bi-infinite tight flute}.
 If
$$
\sum_{n=1}^{\infty} \frac{1}{e^{(1-|t_n|)\frac{\ell_n}{2}} +
e^{(1-|t_{-n}|)\frac{\ell_{-n}}{2}}}=\infty
$$
then $X$   is of parabolic type.
\end{theorem}


\section{A trip to the Menagerie: Applications to various topological types} 
\label{sec: A trip to the Managerie: Applications to various topological types}

 In order  to give the  reader  a sense of the scope of the applications of Theorem \ref{thm:general_parabolic}, we give a sampling   of our sufficient conditions  for the parabolicity  (ergodicity of the geodesic flow) theorem   applied to various topological  settings.  With the exception of subsection \ref{subsec: abelian covers} each subsection is devoted to a  particular  topological type (illustrated by a figure in the subsection) of an infinite type hyperbolic surface. The geodesic pants decomposition of a surface is given by the family of closed geodesics $\{\alpha_n\}$. The relative twist around $\alpha_n$ is denoted $t_n=t(\alpha_n)$, and the length of $\alpha_n$ is denoted by
 $\ell_n=\ell (\alpha_n)$. 
 Any other  notation in the figure should be self-explanatory.   Finally, in subsection \ref{subsec: abelian covers} we address the question of when a topological abelian cover of a compact surface is of parabolic type.

\subsection{Loch-Ness monster (Infinite genus and one non-planar  end)}

Let $X^{\infty}_1$ be as in  Figure \ref{fig:loch-ness}.
The surface $X^{\infty}_1$ is obtained from the tight flute surface by replacing each puncture with a closed geodesic  and attaching a finite genus to each closed geodesic(genus of the attached surfaces may vary and the supremum might be equal to infinity). The surface has one topological non-planar end and infinite genus. We note that if an infinite subsequence of $\{ \alpha_n\}_n$ have lengths bounded from  above then the  surface $X^{\infty}_1$ is of parabolic  type using the estimate for standard collars in Lemma \ref{lem:collar}. It is therefore of interest to assume that the lengths of $\alpha_n$ converge to infinity as $n$ increases.
Using Corollaries \ref{cor:r(d)tol} and \ref{cor:twist1/2} we obtain


\begin{theorem}
\label{thm:infinitegenus}
Let  $X^{\infty}_1$  be a hyperbolic Loch-Ness monster as in Figure \ref{fig:loch-ness}. Suppose there exists $M>0$  so that  $\ell(\beta_n) \leq M$, for all $n$. Then $X^{\infty}_1$ is of parabolic  type if
$$
  \sum_{n=1}^{\infty}  e^{-\left(1-|t(\alpha_n)| \right)
\frac{\ell (\alpha_n)}{2}} =\infty.
$$ 
\end{theorem}


\subsection{Ladder surface (Infinite genus and two non-planar  ends)}

We  denote by $X_2^{\infty}$ the infinite genus surface with two non-planar ends. Denote by $\alpha_n$, for $n\in\mathbb{Z}$, the geodesics which together with $\beta_n$ make a geodesic pants decomposition of the flute part of $X_2^{\infty}$.
Using C	orollaries \ref{cor:r(d)tol} and \ref{cor:twist1/2} and Theorem \ref{thm:general_parabolic} we obtain

\begin{figure}[h]
\begin{tabular}{c c}
\includegraphics[width=3.5in]{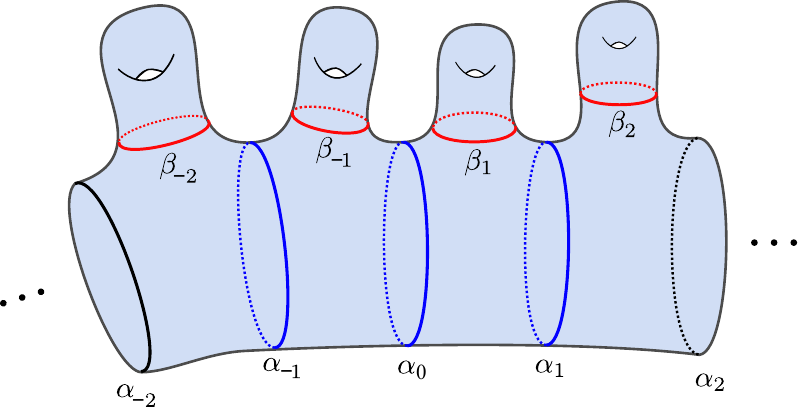}
\end{tabular}
%
%
%
%
\caption{Surface  $X^{\infty}_2$ (the infinite ladder) of infinite genus with  two non-planar  ends.}
\label{fig:loch-ness-bi-infinite}
\end{figure}

\begin{theorem}
\label{thm:biinfinite}
Let  $X^{\infty}_2$ be as in  Figure \ref{fig:loch-ness-bi-infinite}.
Assume there exists $M>0$ such that  $\ell(\beta_n) \leq M$,  for all $n$. Then $X^{\infty}_2$ is of parabolic  type if
\begin{equation} \label{eq:ladder surface}
\sum_{n=1}^{\infty} \frac{1}{e^{(1-|t(\alpha_n)|)\frac{\ell (\alpha_n)}{2}} +
e^{(1-|t(\alpha_{-n})|)\frac{\ell (\alpha_{-n})}{2}}}=\infty
\end{equation}

\end{theorem}


\subsection{The complement of the Cantor set}
Let $X_{\infty}$ be a genus zero surface whose space of topological ends is a Cantor set as in Figure  \ref{fig:cantor}. The surface 
$X_{\infty}$ is homeomorphic to the complement of a Cantor set on the Riemann sphere. 
The surface $X_{\infty}$ is obtained by gluing pairs of pants with boundary geodesics $\alpha_j^n$, for $n=1,2,\ldots$ and $j=1,2,\ldots ,2^{n+1}$, as follows. In step $n=1$, we glue two pairs of pants along a geodesic boundary $\alpha_1^1$. The obtained surface is a sphere minus four disks with four geodesic boundary curves $\alpha_j^1$, for $j=1,2,3,4$. In step $n$ the obtained surface has genus zero and $2^{n+1}$ boundary geodesics $\alpha_j^n$, for $n=1,2,\ldots$ and $k=1,2,\ldots ,2^{n+1}$. In the next step we glue a pairs of pants to each  $\alpha_j^n$ and obtain genus zero surface with $2^{n+2}$ boundary geodesics  $\alpha_{j}^{n+1}$ for $j=1,2,\ldots ,2^{n+2}$.
It is well-known that if the lengths of $\alpha_j^n$ are bounded from below then the surface $X_{\infty}$ is not of parabolic  type (\cite{McM}).  

\begin{figure}[h]
\includegraphics[width=2.5in]{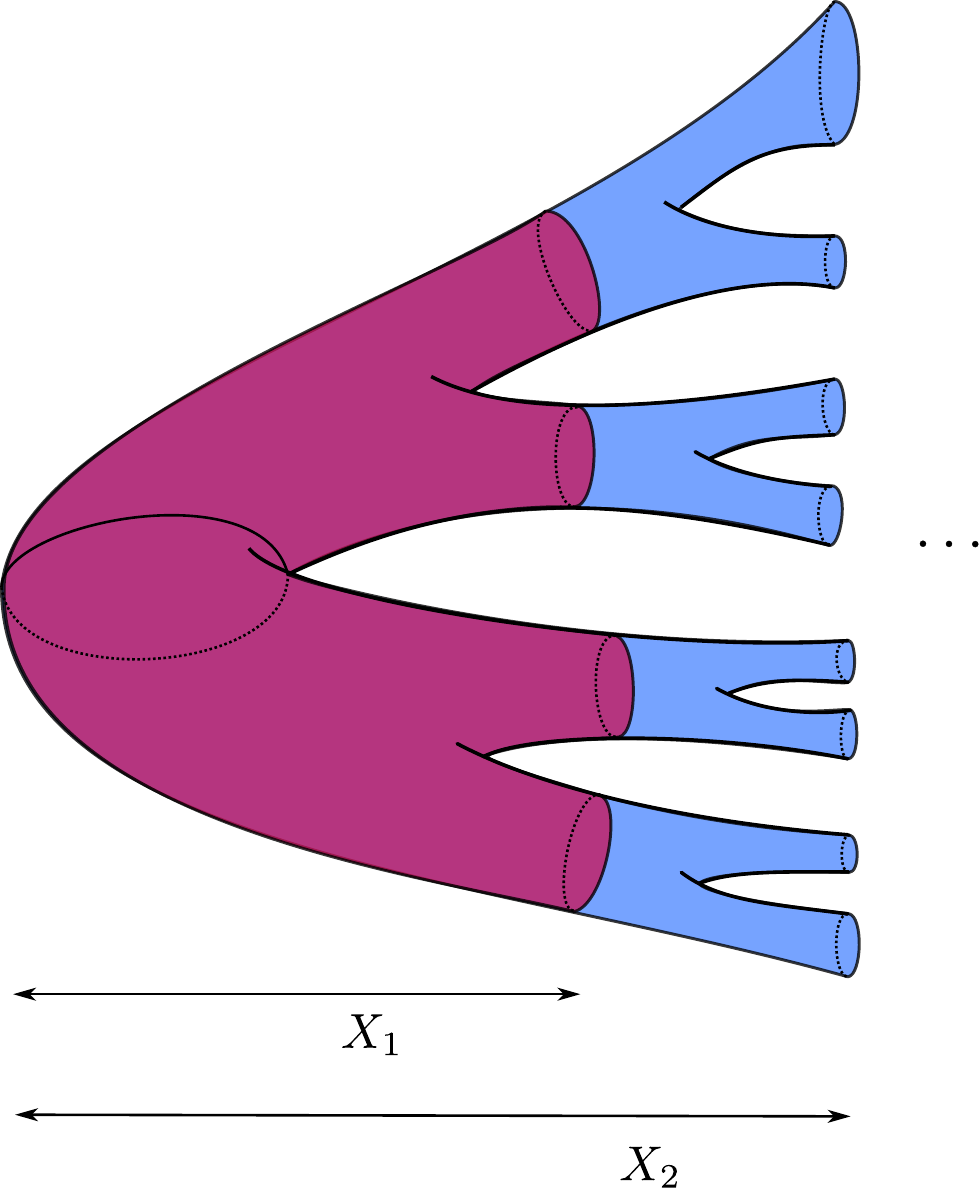}
\qquad
\caption{Complement of a Cantor set.} 
\label{fig:cantor}
\end{figure}
\begin{theorem}
\label{thm:cantor}

Consider the compact exhaustion of   the complement of the Cantor set  $X_{\infty}$ as in Figure \ref{fig:cantor}.  Then $X_{\infty}$ is of parabolic  type if for every $n\geq 1$ and all $\alpha \in \partial_0 X_n$ we have
   $$ \ell (\alpha) \leq C\frac{n}{2^n}.$$
\end{theorem}

\begin{proof}
Denote by $R$ the standard one-sided collar around the geodesic $\alpha \in \partial_0 X_n$. 
Since $\ell (\alpha) \leq C\frac{n}{2^n}$, Corollary \ref{cor:asymp} implies that
$1/\lambda (R) \leq C'\frac{n}{2^n}$.
Thus 
$$
\sum_{n=1}^{\infty}\frac{1}{\sum_{j=1}^{2^n}1/\lambda (R)}\geq\frac{1}{C'}\sum_{n=1}^{\infty}\frac{1}{2^nn/2^{n+1}}=\infty
$$
and the surface $X_{\infty}$ is of parabolic   type by Theorem \ref{thm:general_parabolic}.
\end{proof}

We remark that the same theorem holds if we replace 
$X_{\infty}$ with the blooming Cantor tree-i.e.,  a Riemann surface whose space of ends is a Cantor set and each end is accumulated by genus.

\subsection{Surfaces with  $|\partial_0 X_n|$ bounded}

Let $\{X_n\}$ be an exhaustion  of $X$ so that 
 the number of boundary components of $X_n$  is constant as for the surface   in Figure  \ref{fig: general surface}. Note that this implies that $X$ has finitely many non-isolated ends as well as finitely many 
 non-planar ends. Setting 
 $L_n=\text{max}\{\ell (\alpha): \alpha \in \partial_0 X_n\}$ and $\tau_n=\text{min}\{ |t(\alpha )|: \alpha \in \partial_0 X_n\}$,   an application of Corollary \ref{cor:p_twist}  yields,

\begin{theorem}   Let $\{X_n\}$ be an exhaustion of $X$
with a constant number of components and assume  for each $n$  that   the boundary geodesics in $\partial_0 X_n$ have the same twist $t_n$. 
If
\begin{equation} \label{eq:level n the same}
\sum_{n=1}^{\infty}
\frac{1}{ e^{\left(1- |t_n|\right)\frac{L_n}{2}}}=\infty, 
\end{equation}
then $X$ is of parabolic  type.

\end{theorem}

\begin{figure}[h]
\includegraphics[width=4in]{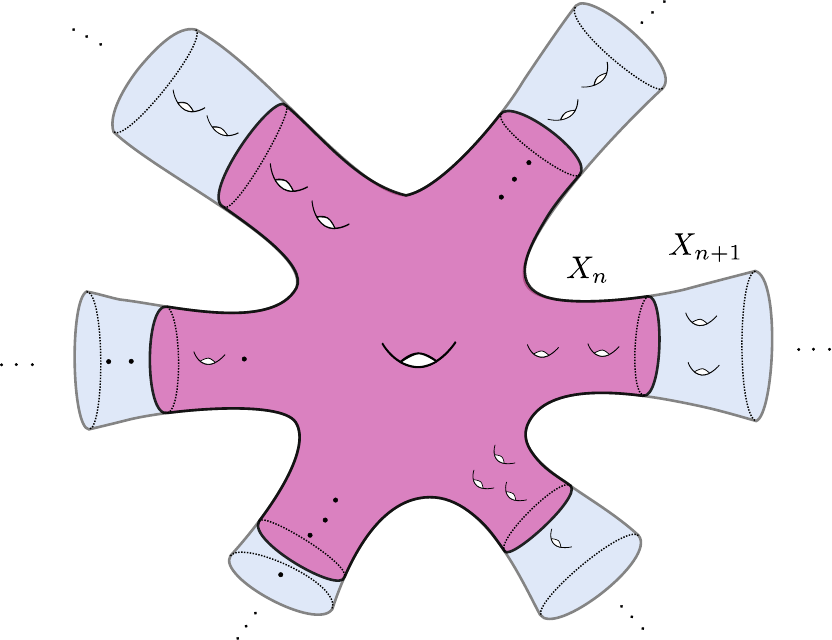}
\qquad
\caption{Surface with  a finiteness property.} 
\label{fig: general surface}
\end{figure}


\subsection{Abelian covers of compact surfaces}
\label{subsec: abelian covers}

All covers discussed in this section are regular covers. We consider a topological  abelian   cover $X$  of a compact Riemann surface $Y$, that is, a regular cover with a properly discontinuous  action by a torsion-free  abelian  group $G$ for which 
$X/G$ is topologically $Y$. 
Such a cover induces a compact exhaustion of $X$ in the following way: 
choose  a fundamental domain,  $P$,    for the action of $G$  on $X$ so that its closure  $\bar{P}$ is a compact subsurface whose boundary projects to the simple closed geodesics on $Y$ which induce the abelian covering. 
The translates of $\bar{P}$ by $G$ tile the surface $X$. Given  $P$,    define $|g|$ to be the least number of translates of 
$\bar{P}$ that a path traverses from $\bar{P}$
to $g\bar{P}$. Let 
$$X_n =\bigcup\{g\bar{P}: g \in G, |g| \leq n\} \subset X,$$
The $\{X_n\}$  form a compact exhaustion of $X$.  We say that this   exhaustion is  {\it induced by the  cover}.  We consider  a fixed pants decomposition of $Y$ whose lifted curves in $X$ together with 
 the curves in $\partial X_n$ form a pants decomposition of $X$. We call this a {\it pants decomposition induced by the cover}. If the  cover is given by an isometric action of $G$, we say it is a {\it geometric cover}. When $G$ is either $\mathbb{Z}$ or $\mathbb{Z}^2$,  Mori (\cite{Mo})  showed that a geometric cover   is parabolic.  Rees (\cite{Rees}) extended  these results in various directions as well as  to  higher dimensions and  both authors    (Mori and Rees)  independently showed  that if $G=\mathbb{Z}^r, \text{ for }
 r \geq 3$, a geometric cover  is not  parabolic. In particular,  these results show that in the case of   $\mathbb{Z}$ and $\mathbb{Z}^2$ covers, to achieve parabolicity  it is sufficient for the cover to be  given by an isometric action. In this subsection,  we use our methods  to  generalize this sufficiency  condition  for parabolicity.    Let $P$ be a  fundamental domain as above.

\begin{itemize}

  \item  {\bf $ G=\mathbb{Z}$, lifting a single curve:} 
In this  case  $\partial{P}\subset X$ is comprised of two simple closed curves that are lifts of  a non-separating simple closed curve on $Y$.  The compact subsurfaces $\{X_n\}$ of the exhaustion induced by the cover  have  two boundary curves for each $n$.  $X$ is topologically a ladder surface.   

\item {\bf $G=\mathbb{Z}^2$, lifting  a pair of disjoint curves:} 
In this case $\partial{P}\subset X$ is the union of four simple closed curves  and 
 the covering  is given by   two disjoint
non-separating  simple closed curves on $Y$ and the compact subsurfaces $\{X_n\}$ of the exhaustion induced by the cover  have at most $4n$   boundary curves  for each $n$.  $X$ is topologically a Loch-Ness monster.

 \item  {\bf $G=\mathbb{Z}^2$, lifting  a pair of intersecting curves:}
  In this case $\partial{P}\subset X$ is a topological rectangle and   the cover is given by  the lift of two intersecting non-separating simple closed curves on $Y$. A schematic picture of the cover is provided in Figure \ref{fig: z2 cover}, where each square represents a copy of $P$, $X_n$ is the $n$ by $n$ square centered at $P$,  and 
   $\partial X_n$   consists of the closed curve  $\alpha_n$.  $X$ is topologically  a Loch-Ness monster.  
 
 \end{itemize}
 
\begin{figure}[h]
\includegraphics[width=4in]{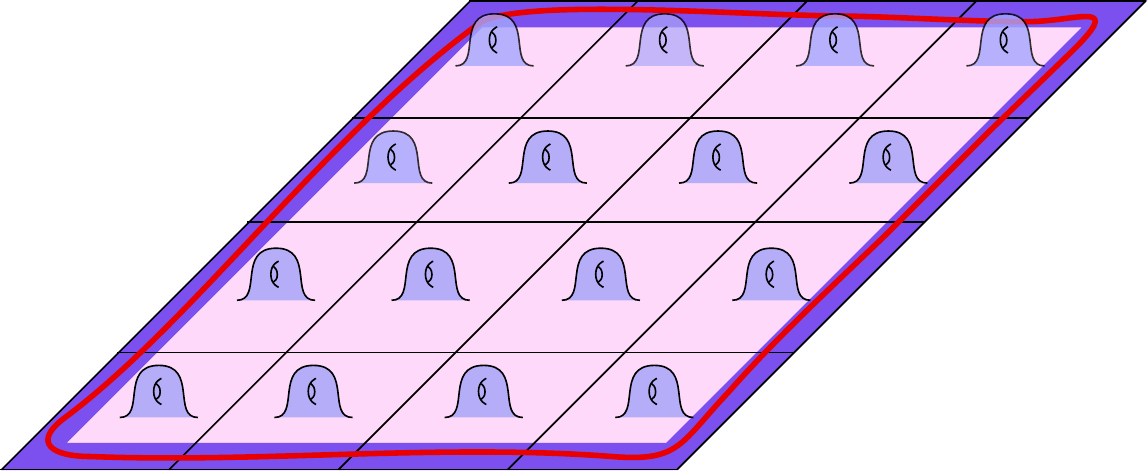}
\qquad
\caption{$\mathbb{Z}^2$-cover.} 
\label{fig: z2 cover}
\end{figure}


We have the  following theorem for topological 
abelian covers. Recall that $L_n=\text{max}\{\ell (\alpha): \alpha \in \partial_0 X_n\}$ and $\tau_n=\text{min}\{|t_n(\alpha)|: \alpha \in \partial_0 X_n\}$.

\begin{theorem} \label{thm:top. covering group}

Let $X$ be an infinite type hyperbolic surface  that  topologically  covers a  compact Riemann surface where the covering group $G$  is either $\mathbb{Z}$ or 
$\mathbb{Z}^2$. Let $\{X_n\}$ be the   compact exhaustion of $X$ induced by the cover.  Any of the following items are sufficient to imply that $X$ is of parabolic type. 
\begin{enumerate}
  \item If $G= \mathbb{Z}$ and  the cover is given by a single non-separating simple closed curve,  assume 
  $$ \label{eq:level n the same}
\sum_{n=1}^{\infty}
\frac{1}{ e^{\left(1- \tau_n\right)\frac{L_n}{2}}}=\infty.
$$

  \item If  $G=\mathbb{Z}^2$  and    the cover is given by two disjoint  non-separating simple closed curves,  assume
  
$$  \label{eq:n boundary components}
\sum_{n=1}^{\infty}\frac{1}
{ne^{\left(1- \tau_n\right)\frac{L_n}{2}}}=\infty .
$$

  \item If $G=\mathbb{Z}^2$ and  the cover is given by  two non-separating simple closed curve that intersect,  for each $n$ assume  that the geodesic in the homotopy class of  
  $\partial_0X_n$ has length $\ell_n$ and  a half-collar of width $\epsilon_n$ satisfying,
  
$$ \label{ }
\sum_{n=1}^{\infty} \frac{\arctan (\sinh \epsilon_n)}{\ell_n}=\infty.
$$

\end{enumerate}

\end{theorem} \label{thm: topological cover}

\begin{proof}  Note that on the boundary of $X_n$,  there are two  components  for each $n$ in the case of the cover corresponding to a single simple closed curve, at most $4n$ components if the cover corresponds to two disjoint curves, and one boundary component for each $n$ if it corresponds to two intersecting  curves. 
Items  (1) and (2) now follow from Theorem
\ref{thm:general_parabolicwithtwist},  and the fact that 
$X_n =\bigcup\{g\bar{P}: g \in G, |g| \leq n\}$ where the translate  $g\bar{P}$ is homeomorphic  to $\bar{P}$.
Item (3)  follows from Theorem \ref{thm: collar width}.
\end{proof}

As a special case we recover the results of Mori and Rees for $G=\mathbb{Z}$ or  $\mathbb{Z}^2$.

\begin{corollary}[\cite{Mo}, \cite{Rees}] 
\label{cor: Mori-Rees result}
Let $X$ be an infinite type  hyperbolic surface  that  geometrically   covers a  compact Riemann surface $Y$ where the covering group is either $\mathbb{Z}$ or 
$\mathbb{Z}^2$. Then $X$ is of parabolic type. 
\end{corollary}

\begin{proof} If the cover is associated to a single curve or two disjoint curves then the maximal  length, $L_n$,  of the boundary geodesics of $X_n$ is constant for all $n$, and hence the corollary follows from items (1) and (2)  of  Theorem \ref{thm:top. covering group}.

We next take up the case that $G=\mathbb{Z}^2$ and the cover is given by the intersection of two simple closed geodesics on $Y$,  say $\beta_1$ and 
$\beta_2$. Cutting $Y$ open along these geodesics gives us $P \subset X$ along with its $\mathbb{Z}^2$-translates depicted in Figure \ref{fig: z2 cover} where  the horizontal, resp. vertical, geodesics  project to $\beta_1$, resp. 
$\beta_2$. The fact  that  $\beta_1$ and $\beta_2$
have standard collars guarantees that there is an 
embedded $\epsilon$-neighborhood about the vertical and horizontal geodesics. In fact, the constant $\epsilon$ can be taken to be the smaller of the standard collar widths of $\beta_1$ and $\beta_2$.  

Note  that in this case $\partial X_n$ consists of one component, whose geodesic representative we call  
$\alpha_n$. The geodesic $\alpha_n$ for geometric and topological reasons lies  in the $n$ by $n$ square
determined by $X_n$, see Figure \ref{fig: z2 cover}), and together with 
$\partial X_n$ bounds a topological  annulus. Using the fact that $\partial X_n$
has an embedded $\epsilon$-neighborhood,  we may conclude  that $\alpha_n$ has a half-collar of width
$\epsilon$.  Finally,  noting that 
$\ell (\alpha_n) < 2n( \ell (\beta_1) + \ell (\beta_2))$,  in particular $\ell (\alpha_n)$ grows  at most linearly  in $n$, we apply   Theorem \ref{thm:top. covering group}, item (3) to conclude parabolicity.

\end{proof}

\begin{example}\label{ex: generalization of geometric cover} Suppose $X$ is an infinite type hyperbolic Riemann surface, and $\pi : X \rightarrow Y$ is a geometric   cover over the compact  Riemann surface $Y$ where the covering group is either $\mathbb{Z}$ or $\mathbb{Z}^2$.  Let 
$\{X_n\}$ be a compact exhaustion of $X$ induced by the cover and  fix  a pants decomposition of $X$ that includes the simple closed curves in $\bigcup \partial_0  X_n$. As we saw in Corollary  \ref{cor: Mori-Rees result}, $X$ is of parabolic type. 
On the other hand, by varying the parameters along just the curves in  the boundary of the exhaustion and applying 
Theorem \ref{thm:top. covering group} we can 
  obtain a large class $\{X^{\alpha} \}_{\alpha \in \mathcal{A}}$  of  hyperbolic Riemann surfaces of parabolic type with the following properties:
  \begin{itemize}
  \item each $X^{\alpha}$ has the same topology as $X$,
  \item $X^{\alpha}$ is  quasiconformally distinct from $X^{\beta}$  for any pair $\alpha \neq \beta$  (see   \cite{BasmajianKim}), 
  \item for each $\alpha$,  the  Fenchel-Nielsen parameters of $X^{\alpha}$ agree with the parameters of the geometric cover, $X$,  along  all the pants curves except the ones on the boundary of the compact exhaustion.
\end{itemize}. \end{example}

\begin{remark}
For rank $ r \geq 3$ the analogue of  the sufficient conditions in items (2) and (3) of Theorem \ref{thm: topological cover}  must  have  the boundary lengths of the  induced cover going to zero.
This is ostensibly because the  $n^{\text{th}}$  term of the series contains $\frac{1}{n^{r-1}}$.
This is underscored by the result mentioned earlier  (\cite{Mo}) or  \cite{Rees})  that  if $G=\mathbb{Z}^r, \text{ for }
 r \geq 3$, then  a geometric cover of the  closed Riemann surface  is not  parabolic.   On the other hand,  one does not  have to change  the hyperbolic structure  much to regain  parabolicity in the rank $r \geq 3$ case. This point is illustrated by the following corollary which follows from Theorem  \ref{thm:general_parabolicwithtwist}.  
\end{remark}

\begin{corollary} \label{cor:st collar-parabolic}
Suppose  $\pi : X \rightarrow Y$ is  a  $\mathbb{Z}^r$-covering of  the   compact Riemann surface $Y$ by   the  infinite type  hyperbolic surface $X$, where the covering is  given by 
$r \geq 3$ disjoint non-separating simple closed curves on $Y$.  
Assume that the compact exhaustion $\{X_n\}$ is induced by the cover and recall that $L_n=\text{max}\{\ell (\alpha): \alpha \in \partial_0 X_n\}$. If 
\begin{equation} \label{eq: coveringforbigrank}
\sum_{n=1}^{\infty}\frac{1}
{n^{r-1} L_n e^{\frac{L_n}{2}}}=\infty,
\end{equation}
then $X$ is of parabolic  type. 
\end{corollary}

\begin{example}
Consider the  hypothesis of the above corollary and    lift  a pants decomposition of $Y$ whose pants curves include the $r$ non-separating curves associated to the cover. Lifting these curves  we obtain  a pants decomposition of $X$;
 it is not difficult to see that for each $n$, $\partial X_n$ is a pants curve in this decomposition.  Save the curves in $\partial X_n$ which we assume  
   satisfy  expression (\ref{eq: coveringforbigrank}),
  the Fenchel-Nielsen parameters of all  the other pants curves can be equal to  the Fenchel-Nielsen parameters of the corresponding curve in  $Y$. In this fashion we provide examples of parabolic surfaces that are  $\mathbb{Z}^r$-topological coverings for $r\geq 3$  of  a   compact Riemann surface  which are geometric coverings away from the $r$ non-separating curves associated to the cover. This stands in contrast to  the results of Mori and Rees which show that geometric covers are not parabolic. 

\end{example}

\end{document}